\pdfoutput=1
\documentclass[11pt]{article}
\usepackage{xcolor}
\usepackage{graphicx}
\usepackage{amsmath,amsfonts,amssymb,graphics,amsthm}
\usepackage{hyperref}
\usepackage{comment}
\usepackage{tabularx}
\usepackage[protrusion=true,expansion=true]{microtype}
\usepackage{enumerate}
\usepackage{bbm}
\usepackage{mathrsfs}
\usepackage[margin=1in]{geometry}
\usepackage[normalem]{ulem}

\newcommand{\xin}[1]{{\color{blue}{#1}}}

\hypersetup{
	colorlinks=false,
	linktocpage,
}

\numberwithin{equation}{section}

\newtheorem{theorem}{Theorem}[section]

\newtheorem{lemma}[theorem]{Lemma}
\newtheorem{proposition}[theorem]{Proposition}

\newtheorem{remark}[theorem]{Remark}
\newtheorem{definition}[theorem]{Definition}

\theoremstyle{remark}

\newcommand{\C}{\mathbbm{C}}
\newcommand{\D}{\mathbbm{D}}
\newcommand{\E}{\mathbbm{E}}
\newcommand{\N}{\mathbbm{N}}

\newcommand{\Q}{\mathbbm{Q}}

\newcommand{\R}{\mathbbm{R}}
\renewcommand{\P}{\mathbbm{P}}
\newcommand{\bbH}{\mathbbm{H}}

\newcommand{\eps}{\varepsilon}
\newcommand{\1}{\mathbf{1}}

\newcommand{\sph}{\mathrm{sph}}
\newcommand{\disk}{\mathrm{disk}}

\newcommand{\sm}{\mathsf{m}}

\newcommand{\Gg}{\Gamma_{\frac{\gamma}{2}}}
\newcommand{\Sg}{S_{\frac{\gamma}{2}}}
\newcommand{\cMtwo}{\mathcal{M}_{2}^\mathrm{disk}}

\newcommand{\cMthree}{\mathcal{M}_{2,\bullet}^\mathrm{disk}}

\newcommand{\Leb}{\mathrm{Leb}}
\newcommand{\LF}{\mathrm{LF}}
\newcommand{\QD}{\mathrm{QD}}
\newcommand{\QS}{\mathrm{QS}}

\newcommand{\mm}{m^\lambda_\gamma}
\newcommand{\wtmm}{\wt m^\lambda_\gamma}

\newcommand{\conf}{\operatorname{conf}}

\newcommand{\haar}{\mathbf{m}}

\let\Re\undefined
\DeclareMathOperator{\Re}{Re}

\DeclareMathOperator{\Var}{Var}
\DeclareMathOperator{\SLE}{SLE}

\def\cX{\mathcal{X}}

\def\cS{\mathcal{S}}

\def\cM{\mathcal{M}}

\def\cH{\mathcal{H}}

\def\cF{\mathcal{F}}

\def\cD{\mathcal{D}}
\def\cC{\mathcal{C}}

\def\alb#1\ale{\begin{align*}#1\end{align*}}
\def\allb#1\alle{\begin{align}#1\end{align}}

\newcommand{\aryb}{\begin{eqnarray*}}
\newcommand{\arye}{\end{eqnarray*}}
\def\alb#1\ale{\begin{align*}#1\end{align*}}
\newcommand{\eqb}{\begin{equation}}
\newcommand{\eqe}{\end{equation}}
\newcommand{\eqbn}{\begin{equation*}}
\newcommand{\eqen}{\end{equation*}}

\newcommand{\BB}{\mathbbm}
\newcommand{\ol}{\overline}

\newcommand{\op}{\operatorname}

\newcommand{\rta}{\rightarrow}

\newcommand{\wt}{\widetilde}
\newcommand{\wh}{\widehat}

\newcommand{\bdy}{\partial}

\let\originalleft\left
\let\originalright\right
\renewcommand{\left}{\mathopen{}\mathclose\bgroup\originalleft}
\renewcommand{\right}{\aftergroup\egroup\originalright}

\DeclareMathAlphabet{\mathpzc}{OT1}{pzc}{m}{it}

\begin{document}

\title{Integrability of SLE via conformal welding of random surfaces}
\author{
	\begin{tabular}{c}Morris Ang\\[-5pt]\small MIT\end{tabular}\; 
	\begin{tabular}{c}Nina Holden\\[-5pt]\small ETH Z\"urich\end{tabular}\; 
	\begin{tabular}{c}Xin Sun\\[-5pt]\small  University of Pennsylvania \end{tabular}
} 
\date{  }

\maketitle

\begin{abstract}
We demonstrate how to obtain integrability results  for the  Schramm-Loewner evolution (SLE) from Liouville conformal field theory (LCFT) and the mating-of-trees framework for Liouville quantum gravity (LQG). In particular,  we prove an exact formula for the law of a conformal derivative of a classical variant of SLE called $\SLE_\kappa(\rho_-;\rho_+)$. Our proof is built on two connections between SLE, LCFT, and mating-of-trees. Firstly, LCFT  and mating-of-trees  provide equivalent but  complementary methods to describe natural  random surfaces in LQG. Using a novel tool that we call the \emph{uniform embedding} of an LQG surface, we extend earlier equivalence results by allowing fewer marked points and more generic singularities. 
Secondly, the conformal welding of these random surfaces produces SLE curves as their interfaces. In particular, we rely on the conformal welding results proved in our companion paper~\cite{ahs-disk-welding}. Our paper is an essential part of a program proving  integrability results for SLE, LCFT, and mating-of-trees based on  these two connections.

\end{abstract}


\section{Introduction}

Two dimensional (2D) conformally invariant random processes have been an active area of research in probability theory for the last two decades. In this paper, we consider the interplay between three central topics in this area: Schramm-Loewner evolution (SLE), Liouville conformal field theory (LCFT), and the mating-of-trees framework for Liouville quantum gravity (LQG).
SLE~\cite{schramm0} is a classical family of random planar curves  which describe the scaling limits of many 2D statistical physics model at criticality, e.g.~\cite{smirnov-cardy,lsw-lerw-ust,ss-dgff, chel-smirnov}. 
LQG is a family of  random planar geometries~\cite{shef-kpz,dddf-lfpp,gm-uniqueness} that naturally arise in the study of string theory and 2D quantum gravity~\cite{polyakov-qg1}. 
It also  describes the scaling limit of a large class of random planar maps, see e.g.~\cite{shef-burger,gwynne-miller-saw,hs-cardy-embedding}.
LCFT is the 2D quantum field theory that governs LQG which is recently made rigorous by~\cite{dkrv-lqg-sphere} and follow-up works.
Mating-of-trees~\cite{wedges} is an encoding of SLE on the LQG background via Brownian motions.
See~\cite{lawler-book, berestycki-lqg-notes, vargas-dozz-notes, gwynne-ams-survey,ghs-mating-survey}
and references therein for more background  on these rapidly developing topics. 
 
One key feature shared by the three topics is the rich integrable  (i.e., exactly solvable) structure. 
First, since its discovery, many exact formulas
for SLE have been proved or conjectured;  see e.g.~\cite{ssw-radii,schramm-left-passage,als-cle-radius,dub-watts,sw-watts}.
Moreover, as an important example of 2D conformal field theory, 
LCFT enjoys rich integrability predicted by theoretical physics~\cite{bpz-conformal-symmetry,do-dozz,zz-dozz,fzz}, some of which were recently proved in~\cite{krv-dozz,remy-fb-formula,rz-gmc-interval,rz-boundary,gkrv-bootstrap,ARS-FZZ}. 
Finally, mating-of-trees expresses  many observables defined by SLE and LQG via Brownian motion and related processes; see e.g.  \cite{wedges,sphere-constructions,ag-disk,ghm-conformal-dim,msw-cle}.
 In this paper we demonstrate how to obtain integrable results  for SLE from LCFT and mating-of-trees by proving 
an exact formula for a classical variant of SLE called $\SLE_\kappa(\rho_-;\rho_+)$; see Theorem~\ref{thm-conformal-derivative0}.

Our paper is part of a  program by the first and the third authors connecting the aforementioned three types of integrable structures and proving new results in each direction.
The foundation of the  program are two bridges  between these objects. The first bridge is that
LCFT~\cite{dkrv-lqg-sphere} and mating-of-trees~\cite{wedges} provide equivalent but  complementary methods to describe natural  random surfaces in LQG.  
This equivalence was first demonstrated for the quantum sphere in~\cite{ahs-sphere} and recently extended to the  quantum disk in~\cite{cercle-quantum-disk}.
Using what we call the uniform embedding of quantum surfaces,  
we provide more conceptual and unified proofs for  these facts and greatly extend them; see Section~\ref{subsec:equivalence}.

The second bridge is that  random surfaces behave well under conformal welding with  SLE curves as their interfaces.  
The conformal welding results needed for our paper are established in  our companion paper~\cite{ahs-disk-welding}, extending 
the seminal works~\cite{shef-zipper,wedges}.  The way we use it to  prove Theorem~\ref{thm-conformal-derivative0} is instrumental for the entire program.
In particular, it is crucial to the forthcoming work of the first and the third authors on the integrability of the conformal loop ensemble~\cite{AS-CLE}, as well as  their joint work with Remy~\cite{ARS-FZZ} on the proof of the FZZ formula in LCFT. See~Section~\ref{subsec:welding} for an overview of this method.

\subsection{An integrability result on $\SLE_\kappa(\rho_-;\rho_+)$}\label{subsec:integrability}
We now present our main result concerning the integrability of SLE. 
For $\kappa>0$ and $\rho_-,\rho_+>-2$, the (chordal) $\SLE_\kappa(\rho_-;\rho_+)$ is the natural generalization of the chordal $\SLE_\kappa$ where one keeps track of two extra marked points on the domain boundary called force points. The parameters $\rho_\pm$ indicate to what extent the force points attract or repulse the curve. 
In our paper the force points are always located infinitesimally to the left and right, respectively, of the starting point of the curve. 
$\SLE_\kappa(\rho_-;\rho_+)$ was introduced in~ \cite{lsw-restriction} and  studied in e.g.\ \cite{dubedat-rho,ig1}. 
See Appendix~\ref{app:sle-def}  for more background.

Let $\eta$ be an $\SLE_\kappa(\rho_-; \rho_+)$ curve on the upper half plane $\bbH$ from $0$ to $\infty$, which is a random curve on $\bbH \cup \R$. 
When $\rho_- > -2$ and $\rho_+ > (-2)\vee(\frac\kappa2 - 4)$, then the point $1$ is almost surely not on the trace of $\eta$~\cite{ig1}. 
Therefore, we can define the open set $D$ to be the connected component of $\bbH \backslash \eta$ which contains the point $1$. Let $\psi: D \to \bbH$ be the conformal map which fixes $\psi(1)=1$ and maps the first (resp. last) point on $\partial D$ traced by $\eta$ to 0 (resp. $\infty$). Note that if $\rho_->\frac\kappa4-2$ then the curve does not touch $(0,\infty)$ so that $\psi$ will fix $0,1,\infty$ in this case.

Our first main result gives the exact distribution of $\psi'(1)$ in terms of its moment generating function. To describe this result, we need the double gamma function $\Gamma_{b}(z)$ which arises in LCFT. We recall its precise definition in~\eqref{eq:double-gamma}.  Using $\Gamma_{b}(z)$, we introduce
\begin{equation}\label{eq:def-F}
F(x, \kappa, \rho_-, \rho_+) := \frac{ \Gamma_{\frac{\sqrt{\kappa}}{2}} 
	(\frac2{\sqrt\kappa}-\frac{\sqrt\kappa}2 + \frac{\rho_+}{\sqrt\kappa} + \frac x2) \;\;\Gamma_{\frac{\sqrt{\kappa}}{2}} (\frac4{\sqrt\kappa} + \frac{\rho_+}{\sqrt\kappa} - \frac x2 )}{\Gamma_{\frac{\sqrt{\kappa}}{2}} (\frac4{\sqrt\kappa} -\frac{\sqrt\kappa}2 +\frac{\rho_- + \rho_+}{\sqrt\kappa} + \frac x2) \Gamma_{\frac{\sqrt{\kappa}}{2}} (\frac6{\sqrt\kappa} +\frac{\rho_- + \rho_+}{\sqrt\kappa}-\frac x2)}.
\end{equation}

\begin{theorem}\label{thm-conformal-derivative0}
	Fix $\kappa >0 $, $\rho_- > -2$ and $\rho_+ > (-2)\vee(\frac\kappa2 - 4)$.  Let $\lambda_0 = \frac1\kappa(\rho_+ + 2)(\rho_+ + 4 - \frac\kappa2)$. 
	For any $\lambda < \lambda_0$, let $\alpha$ be a solution to $1 - \frac\alpha2 (\frac{\sqrt\kappa}2 + \frac2{\sqrt\kappa} - \frac\alpha2) = \lambda$. Then we have
	\alb
	\E[\psi'(1)^\lambda] &= \frac{F(\alpha, \kappa, \rho_-, \rho_+)}{F(\sqrt \kappa, \kappa, \rho_-, \rho_+)}.
	\ale
	Moreover, for any $\lambda \geq \lambda_0$ we have $\E[\psi'(1)^\lambda] = \infty$. 
\end{theorem}
In Theorem~\ref{thm-conformal-derivative0}, the value of $F(\alpha, \kappa, \rho_-, \rho_+)$ does not depend on which value of $\alpha$ is chosen as the solution of the quadratic equation. Moreover, the point $1$ in $\psi'(1)$ is merely for concreteness. The result for other points follows from rescaling.

Our proof of Theorem~\ref{thm-conformal-derivative0} does not use stochastic calculus coming from the Loewner evolution definition of $\SLE_\kappa(\rho_-; \rho_+)$, as is done in many exact calculations concerning SLE, see~e.g.~\cite{lawler-book}.
Instead, we rely on
the following ingredients: the description of natural quantum surfaces in LQG via LCFT; conformal welding of finite volume quantum surfaces from~\cite{ahs-disk-welding}; the integrability results of Remy and Zhu~\cite{rz-gmc-interval,rz-boundary} on boundary  LCFT; and mating-of-trees description of some special quantum surfaces.
We will elaborate on these ingredients in Sections~\ref{subsec:equivalence} and~\ref{subsec:welding}   

\subsection{Two perspectives on random surfaces in Liouville quantum gravity}\label{subsec:equivalence}

A key ingredient in our proof of Theorem~\ref{thm-conformal-derivative0}  is a thorough understanding of two perspectives on random surfaces in LQG when the underlying complex structure enjoys an abundance of conformal symmetries. The first perspective is the quantum surface and the second one is the path integral formalism of LCFT.

We start by recalling some basic geometric concepts in   LQG. We will keep the review brief and  provide more details and references in Section~\ref{subsec-GFF}.
The free boundary Gaussian free field (GFF) on a planar domain $D\subsetneq C$ is the Gaussian process on $D$ with covariance kernel  given by the Neumann Green function on $D$, which can be viewed as a random generalized function on $D$~\cite{shef-gff}.
There are other variants of the GFF
which have the same regularity. 
Suppose $h$ is a variant of  GFF defined on $D$.  The $\gamma$-LQG area measure $\mu_h$ associated with $h$ is formally defined by $e^{\gamma h}d^2z$, which is made rigorous by regularization and normalization~\cite{shef-kpz}.

Fix $\gamma\in (0,2)$. Suppose $f:D\rightarrow \wt D$ is a conformal map between two domains $D$ and $\wt D$. For a generalized function $h$  on $D$, define
\begin{equation}\label{eq:coordinate}
f \bullet_\gamma h = h \circ f^{-1} + Q  \log \left| (f^{-1})' \right| \quad  \textrm{where } Q= \frac\gamma2 + \frac2\gamma.
\end{equation}
If $h$ is a variant of GFF,  then the pushforward of the $\gamma$-LQG area measure $\mu_{h}$ under $f$ equals $\mu_{\wt h}$ a.s\ where $\wt h=f \bullet_\gamma h$.
Equation \eqref{eq:coordinate} is called the coordinate change formula for $\gamma$-LQG. 

Suppose $h$ is the free boundary GFF on $D$.
If $D$ has a boundary segment $L\subset \R$, then we can define $\gamma$-LQG boundary  length measure $\nu_h=e^{\frac\gamma2 h} dz$ on $L$ similarly as $\mu_h$. 
For general domains, the definition of $\nu_h$ can be extended via conformal maps and the coordinate change formula \eqref{eq:coordinate}. It is also possible to define a random metric on $D$ associated with $h$ (see~\cite{dddf-lfpp,gm-uniqueness}) but the metric will not be considered in our paper. 
	
In light of the coordinate change formula, Sheffield~\cite{shef-zipper}   introduced the notion of quantum surface. 
Suppose  $h$ and $\wt h$ are generalized functions on two domains $D$ and $\wt D$, respectively. 
For $\gamma\in (0,2)$, we say that $(D, h) \sim_\gamma (\wt D, \wt h)$ if there exists a conformal map $f:  D \to  \wt D$ such that $\wt h = f\bullet_\gamma h$.
A   \emph{quantum surface}   is an equivalence class of pairs $(D,h)$ under this equivalence relation, and an
\emph{embedding}  of the quantum surface is a choice of $(D, h)$ from the equivalence class. We can also  consider  quantum surfaces decorated  
by other structures such as points or curves, via a natural generalization of the equivalence relation; see Section~\ref{sec:disk-welding}.

Liouville conformal field theory (LCFT) is the quantum field theory corresponding to the Liouville action which originates from Polyakov's work on quantum gravity and bosonic string theory~\cite{polyakov-qg1}.  It associates a random field to each two dimensional Riemannian manifold which all together form a conformal field theory. LCFT was first rigorously constructed  on the sphere by  David, Kupiainen, Rhodes and Vargas~\cite{dkrv-lqg-sphere} by making sense of the path integral for the Liouville action.  It was later extended to other surfaces~\cite{hrv-disk,remy-annulus,drv-torus,gkrv-genus}.

We will focus on the LCFT on the Riemann sphere $\wh \C$ and  the upper half plane $\bbH$.  
The basic inputs are the Liouville fields $\LF_{ \C}$ and $\LF_{\bbH}$.
These are infinite measures on the space of generalized functions on $ \C$ and $\bbH$, obtained from an additive perturbation of GFF.
See Definitions~\ref{def-LF-sphere} and \ref{def-LF-rz}. For $z_1,\cdots, z_k\in  \C$, and $\alpha_1,\dots,\alpha_k$, 
one can add \emph{insertion}s to  $\LF_{ \C}$ by making sense of $\prod_{i=1}^k  e^{\alpha_i \phi(z_i)}\LF_{ \C}$, which we denote by 
$\LF_{ \C}^{(z_1,\alpha_1),\cdots , (z_k,\alpha_k)}$; see Definition~\ref{def-RV-sph}. 
We can similarly define Liouville fields on $\bbH$ with insertions, where for $z_k\in \bdy \bbH$,
we need to replace $e^{\alpha_i \phi(z_i)}$ by $e^{\frac{\alpha_i}{2} \phi(z_i)}$. 
The Liouville correlation functions, which are the fundamental observables in LCFT,  are defined in terms of certain averages  over  these random fields. 

Quantum surfaces and LCFT provide two perspectives on random surfaces in LQG.
For those arising as the scaling limit of canonical measures on discrete random surfaces (a.k.a.\ random planar maps), both perspectives provide natural and instrumental descriptions.
Their close relation has been demonstrated  by Aru, Huang and the third author~\cite{ahs-sphere} for  the \emph{quantum sphere} and by Cercl\'{e}~\cite{cercle-quantum-disk} for the \emph{quantum disk}. The quantum sphere with $k$ marked points (Definition~\ref{def-QS}) is a quantum surface with spherical topology with $k$ marked points defined by Duplantier, Miller and Sheffield~\cite{wedges}. They arise as the scaling limit of natural planar maps models on the sphere; see~\cite{ghs-mating-survey} for a review. 
We similarly have the quantum disk with $m$ interior marked points and $n$ boundary marked points; see Definition~\ref{def-QD}.
We use $\QS_k$  and $\QD_{m,n}$  to denote their distributions, respectively.  We also write $\QS_0$ as $\QS$ and $\QD_{0,0}$ as $\QD$.
Without constraints on area or boundary length, these measures are infinite. 
The main result of~\cite{ahs-sphere} says that  modulo a multiplicative constant, 
$\LF_{ \C}^{(z_1,\gamma), (z_2,\gamma),(z_3, \gamma)}$ equals $\QS_3$ embedded on $(\C, z_1,z_2,z_3)$.
By~\cite{cercle-quantum-disk}, the same holds with $\C$ replaced by $\bbH$, $\QS_3$ replaced by $\QD_{0,3}$, and $z_1,z_2,z_3$ assumed to be on  
$\bdy \bbH$. 

One major difference between the two perspectives is that
for LCFT, the number of marked points is often assumed to be such that
 the marked surface has a unique conformal structure. On the other hand, many important quantum surfaces do not have enough marked points to fix the conformal structure, such as $\QS_k$ and $\QD_{0,k}$ for $k\le 2$.
The starting point of our paper is the observation that even without enough marked points, Liouville fields, possibly with insertions,
describe natural  quantum surfaces that are embedded in a uniformly random way.

To concretely demonstrate our point, let $D$ be either a simply connected domain conformally equivalent to  $\wh\C$ or $\bbH$. 
Let $\conf(D)$ be the group of conformal automorphisms of $D$ where the group multiplication is the function composition $f\cdot g=f\circ g$.
Let $\haar_{D}$ be a Haar measure on $\conf(D)$, which is both left and right invariant. 
Suppose $\mathfrak f$ is a sample from $\haar_{D}$ and $h$ is a function on $D$.
We call the random function $\mathfrak f\bullet_\gamma h$ the \emph{uniform embedding} of $(D,h)$ via $\haar_D$.
By the invariance property of $\haar_D$, the law of $\mathfrak f\bullet_\gamma h$ only depends on $(D,h)$ as a quantum surface.
We write $\haar_{\wh\C}\ltimes \QS$ as the law of $\mathfrak f\circ h$ 
where $(\C, h)$ is an embedding of a sample from the quantum sphere measure $\QS$, and  $\mathfrak f$ is independently sampled from $\haar_{\wh \C}$.
We call $\haar_{\wh \C}\ltimes \QS$ the \emph{uniform embedding of $\QS$} via $\haar_{\wh \C }$. 
We define  $\haar_{\bbH} \ltimes \QD$  in the exact same way.
Here although $\haar_{\wh \C},\haar_{\bbH},\QS,\QD$ are only $\sigma$-finite measures,
we adopt  probability terminologies such as sample, law, and independence. 
\begin{theorem}\label{thm-haar}
For $\gamma\in (0,2)$, there exist 
constants $C_1$ and $C_2$ such that
\begin{equation}\label{key}
\haar_{\wh \C}	\ltimes\QS  = C_1 \cdot \LF_{\C} \quad \textrm{and}\quad \haar_{\bbH} 	\ltimes\QD  = C_2 \cdot \LF_\bbH.
\end{equation}
\end{theorem}

We can also consider the uniform embedding of quantum surfaces with marked points. 
For example, for $a,b\in D\cup \bdy D$, let $\conf(D,a,b)$ be the subgroup of $\conf(D)$ fixing $a, b$ and $\haar_{D,a,b}$ be a Haar measure on $\conf(D,a,b)$.
For example,  $\QD_{0,2}$  can be identified as a measure on $C^\infty_0(D)'/\conf(D,a,b)$ for some domain $D$ with boundary points $a,b$, where $\conf(D,a,b)$ is the subgroup of $\conf(D)$ fixing $a, b$.   Then $\haar_{D,a,b} \ltimes \QD_{0,2}$ can be defined in the same way as   $\haar_{\wh\C} 	\ltimes\QS$ and
 $\haar_{\bbH} 	\ltimes\QD$. We will prove Theorem~\ref{thm-haar} in Section~\ref{sec:equivalence}.
The key to our proof  is the LCFT description of the uniform embedding of  $\QS_2$ and $\QD_{0,2}$ in the cylinder and strip coordinates:
 \eqb
 \begin{split}
 	&\haar_{\cC,-\infty,+\infty} \ltimes \QS_{2} =C\LF_{\cC}^{(\gamma,+\infty),(\gamma,-\infty)} \quad \textrm{and} \quad  \\
 	&\haar_{\cS,-\infty,+\infty} \ltimes \QD_{0,2} =C\LF_{\cS}^{(\gamma,-\infty),(\gamma,+\infty)}
 \end{split}
 \label{eq-field-equiv}
 \eqe
 where $\cC$ is a horizontal cylinder and $\cS$ is a  horizontal strip.
Although this is essentially equivalent to the results  in~\cite{ahs-sphere} and~\cite{cercle-quantum-disk} 
our proof   is much shorter. Thanks to the choice of coordinate, 
the identities \eqref{eq-field-equiv} are equivalent to  an interesting fact about drifted Brownian motion  that we prove as Proposition~\ref{prop-Bessel-equiv}. 
This fact also gives the analogous result if the singularity at the marked points is more general
(see Section~\ref{subsec:welding} and  Definition~\ref{def-thick-disk}). 

Our method for proving Theorem~\ref{thm-haar} is also used to give the LCFT description of $\QD_{1,0}$ in~\cite[Section 3]{ARS-FZZ}. It can be extended to quantum surfaces decorated with SLE curves. For example, in \cite{ahs-loop} we proved that the SLE loop coupled with $\LF_{\C}$ is the uniform embedding of the welding of two independent copies of $\QD$.

The LCFT description of quantum surfaces has  two advantages. First, it is a common operation to  add marked points to quantum surfaces according to
some  quantum  intrinsic measure. This operation is tractable on the LCFT side via the Girsanov theorem; see Section~\ref{subsec-3pt}. The second advantage is that LCFT correlation functions are  exactly solvable. In Section~\ref{subsec:welding} we will explain how these ideas can be applied  to prove Theorem~\ref{thm-conformal-derivative0}.

\subsection{Integrability of SLE through conformal welding and LCFT}\label{subsec:welding}

The starting point of our proof of Theorem~\ref{thm-conformal-derivative0} is the conformal welding result we proved in~\cite{ahs-disk-welding}. 
For $\gamma\in (0,2)$ and $\kappa=\gamma^2\in (0,4)$, 
if we run an independent $\SLE_\kappa$ on top of a certain type of 
$\gamma$-LQG quantum surface, the two quantum surfaces on the two sides of the SLE curve are independent. 
Moreover, the original curve-decorated quantum surface can be recovered by gluing the two smaller quantum surfaces according to the quantum boundary lengths.
The recovering procedure is called \emph{conformal welding}. Such results were
first established by Sheffield~\cite{shef-zipper} and later extended in~\cite{wedges}. 
They play a fundamental role in the mating-of-trees theory.


In~\cite{ahs-disk-welding} we proved conformal welding  results 
for a family of finite-area quantum surfaces, generalizing their infinite-volume counterpart proved in~\cite{shef-zipper} and~\cite{wedges}.
We recall them now. 
For $W>0$, let $\cM^\disk_2(W)$ be the 2-pointed quantum disk of weight $W$ introduced in~\cite{wedges};  see Section~\ref{subsec-GFF}.
For $W\ge \gamma^2/2$,  $\cM^\disk_2(W)$ is  an infinite  measure on quantum surfaces with two boundary marked points. The log-singularity of the field at each marked point is  $-\beta\log|\cdot|$ where 
\begin{equation}\label{eq:weight-beta}
\beta = Q+\frac{\gamma}{2}-\frac{W}{\gamma}.
\end{equation}
The 2-pointed quantum disk $\QD_{0,2}$ is the  special case of $\cM^\disk_2(W)$ where $W=2$. For $W\in (0,\gamma^2/2)$,   
$\cM^\disk_2(W)$ is a Poissonian collection of samples from $\cM^\disk_2(\gamma^2-W)$, viewed as an ordered chain of 2-pointed quantum surfaces.

In~\cite{ahs-disk-welding}, we showed that   the conformal welding of independent samples from $\cM^\disk_2(W_-)$ and $\cM^\disk_2(W_+)$ 
gives a sample from $\cM^\disk_2(W_-+W_+)$ decorated by an independent SLE$_{\kappa}(W_--2;W_+-2)$ running between the two marked point. 
To put it more formally, we can write this result as 
\begin{equation}\label{eq:welding-eq}
	\begin{split}
		&\cMtwo(W_- + W_+) \otimes \SLE_\kappa (\rho_-; \rho_+)\\ 
		&\qquad= c_{W_-, W_+} 
		\int_0^\infty \op{Weld}  (\cMtwo(W_-; \cdot , x) , \cMtwo(W_+; x, \cdot )) \, dx. 
	\end{split}
\end{equation}
In~\eqref{eq:welding-eq}, $\rho_-=W_--2$,  $\rho_+=W_+-2$, and $ c_{W_-, W_+}$ is a positive constant which we call the \emph{welding constant}. 
The measure $ \cMtwo(W_-; \cdot, x)$ is defined by the disintegration  $\cMtwo(W_-)=\int_0^\infty  \cMtwo(W_-; \cdot, x)\, dx$ where $x$ represents the quantum length of the right  boundary arc. 
We similarly define $ \cMtwo(W_+;x,\cdot )$ for the left boundary. The operator $\op{Weld}$ means conformal welding along the boundary arc with length $x$.
See Section~\ref{sec:disk-welding} for more details on~\eqref{eq:welding-eq}.

At the highest level, our proof of Theorem~\ref{thm-conformal-derivative0} is done in four steps. 
\begin{enumerate}
\item Use LCFT to define a variant of $\cMtwo(W)$ where we add a third boundary marked point with a generic log singularity.
\item Prove a version of the conformal welding equation \eqref{eq:welding-eq}  for the three-point variant of $\cMtwo(W)$.
\item Show  that the  welding constants   encode the moments of $\psi'(1)$ in Theorem~\ref{thm-conformal-derivative0}. 
\item Use the  integrability from LCFT and mating-of-trees to compute the  welding constants.
\end{enumerate} 

We now summarize the key ideas and inputs for implementing the 4-step proof.
To keep the picture simple,  we first assume that $W_-\ge \gamma^2/2$ and $W_+\ge \gamma^2/2$ in~\eqref{eq:welding-eq}.  
In this case the $\SLE_\kappa(\rho_-;\rho_+)$ curve does not touch the boundary of the weight $(W_-+W_+)$ quantum disk.

\begin{figure}[ht!]
	\begin{center}
		\includegraphics[scale=0.8]{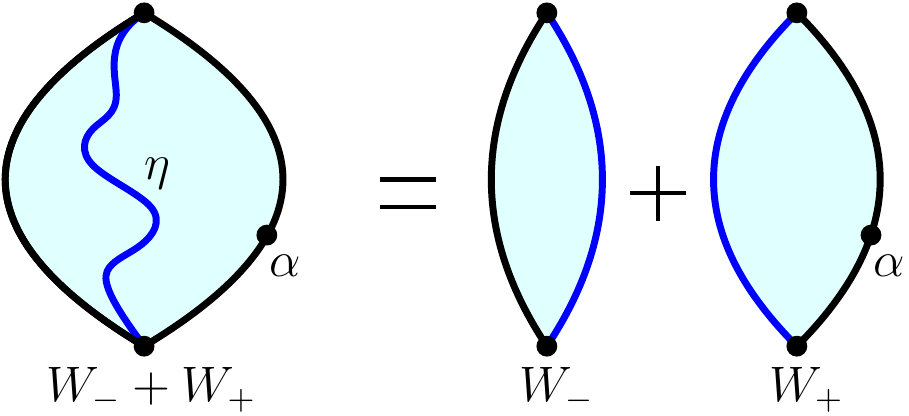}%
	\end{center}
	\caption{\label{fig:welding}  Illustration of the conformal welding result in \eqref{eq:welding-alpha}.}
\end{figure}

We define $\cM^\disk_{2,\bullet}(W; \alpha)$ 
to be the  measure on quantum surfaces such that after being embedded in $(\bbH,0,\infty,1)$, the field is distributed as $\frac\gamma{2(Q-\beta)^{2}}\LF_{\bbH}^{(\beta,0),(\beta,\infty), (\alpha, 1)}$. For $W=2$ and $\alpha=\gamma$, by~\cite{cercle-quantum-disk}, 
$\cM^\disk_2(2;\gamma)$ agrees with  $\QD_{0,3}$.
As alluded in Section~\ref{subsec:equivalence}, we give a concise proof of this result which also extends to surfaces with other singularities. 
The new method also allows us to show that 
$\cM^\disk_{2,\bullet}(W; \gamma)$ for a general $W\ge \gamma^2/2$ is obtained by adding to $\cM^\disk_{2}(W)$
an extra point on the right boundary according to the quantum length measure. 

We then extend the welding equation \eqref{eq:welding-eq}.
For all $\alpha \in \R$  we prove that 
\eqb\label{eq:welding-alpha}
\begin{split}
	&\cMthree(W_- + W_+;\alpha)\otimes \sm(\rho_-; \rho_+; \alpha)\\
	&=   c_{W_-, W_+} \int_0^\infty  \op{Weld}  (\cMtwo(W_-; \cdot, x),\cMthree(W_+; \alpha; x,\cdot)) \, dx.
\end{split}
\eqe
Here, $\sm(\rho_-; \rho_+; \alpha)$ is a measure on curves obtained from reweighting  $\SLE_\kappa(\rho_-;\rho_+)$ by $ \psi'(0)^{1-\Delta}$ with $\Delta= \frac\alpha2 (Q - \frac\alpha2)$  and $\psi$ as defined in Theorem~\ref{thm-conformal-derivative0}. 
For $\alpha=\gamma$,  this equation is straightforward from~\eqref{eq:welding-eq} by adding a quantum typical point on the right boundary. \
For general $\alpha$, this follows from an application of the Girsanov theorem. 
The extra factor of  $ \psi'(0)^{1-\Delta}$ arises in a similar fashion as  the $Q\log$ term 
in the $\gamma$-LQG coordinate change formula~\eqref{eq:coordinate}.

By definition, the total mass  of $\sm(\rho_-; \rho_+;\alpha)$ equals $\E[\psi'(0)^{1-\Delta}]$.
Therefore, forgetting the curve in~\eqref{eq:welding-alpha}, the integral 
$$\int_0^\infty \op{Weld}  (\cMtwo(W_-; \cdot , x) ,\cMthree(W_+; \alpha; x,\cdot))\, dx
$$  
equals $C(\alpha)
\cMthree(W_- + W_+;\alpha)$ as measures on quantum surfaces, where $C(\alpha)=c_{W_-, W_+}^{-1} \E[\psi'(0)^{1-\Delta}]$. 
To determine $C(\alpha)$, we only need to match the distribution of a single observable on both sides. 
The one we choose is the left boundary length of $\cMthree(W_- + W_+;\alpha)$.

Let $L$ and $R$ be the left and right, respectively, boundary lengths of a sample from $\cM^\disk_2(W)$. 
Then both of $\cMthree(W;\alpha)[e^{-s L}]$ and $\cM^\disk_2(W)[1-e^{-s_1L-s_2R}]$ are LCFT correlation functions computed by Remy and Zhu~\cite{rz-gmc-interval,rz-boundary}. 
In particular, let $R(\beta; s_1,s_2)=\cM^\disk_2(W)[1-e^{-s_1L-s_2R}]$ with $W=\gamma(Q+\frac{\gamma}{2}-\beta)$. 
Then $R$  is the so-called \emph{boundary reflection coefficient}. 
This allows us to compare the left boundary length of $\cMthree(W_- + W_+;\alpha)$ on both sides  of~\eqref{eq:welding-alpha} and express  $\E[\psi'(0)^{1-\Delta}]$
in terms of certain explicitly known LCFT correlation functions.

Due to the integration on the right side of~\eqref{eq:welding-alpha}  the expression for $\E[\psi'(0)^{1-\Delta}]$ using LCFT is far from the neat product form in Theorem~\ref{thm-conformal-derivative0}. 
However, for $W_-=2$ or $\gamma^2/2$, the mating-of-trees theory provides a simple description of the area and boundary lengths distribution of $\cM^\disk_2(W_-)$  in terms of 2D Brownian motion in cones. The case with $W_-=2$ is known from~\cite{ag-disk} and \cite{wedges}. 
The case with $W_-=\gamma^2/2$ is obtained in our paper~\cite{ahs-disk-welding}.
 This allows us  to prove Theorem~\ref{thm-conformal-derivative0} for $\kappa\in (0,4)$, $\rho_-\in \{ 0,  \kappa/2 -2\}$, and $\rho_+\ge\kappa/2-2$ 
(recall that $\rho_-=W_- - 2$ and $\kappa=\gamma^2$).

The same argument can also be run when  $W_+ \in (0,\gamma^2/2)$ to cover the range $\rho_+\in (-2,\kappa/2-2)$.
For $W\in (0,\gamma^2/2)$, $\cM^\disk_2(W)$ is a chain of $\cM^\disk_2(\gamma^2 - W)$-quantum disks. 
In this case we can still define $\cM^\disk_{2,\bullet }(W;\alpha)$. This new quantum surface is not so natural from  the  perspective of either~\cite{wedges}  or~\cite{dkrv-lqg-sphere} but it becomes natural after we combine the two.
Due to Campbell's formula for Poisson point process,
both of the boundary length distributions of $\cM^\disk_{2}(W)$ and $\cM^\disk_{2,\bullet }(W;\alpha)$ can be computed in terms of their counterparts with $W$ replaced by $\gamma^2-W$. This allows us to carry out the proof as before. Our computation shows that the boundary length distribution of $\cM^\disk_{2}(W)$ in the thin regime is an analytic continuation of the boundary length distribution in the thick regime, which provides 
a probabilistic counterpart for a well-known  numerical fact on the reflection coefficient: $R(\beta;s_1,s_2)R(2Q-\beta;s_1,s_2)=1$. 

To prove the general case of Theorem~\ref{thm-conformal-derivative0}, we consider the pair of SLE
curves which are the interfaces when conformally welding $\cM_2(W_1)$, $\cM_2(W_2)$, and $\cM_2(W_3)$. 
This allow us to derive a multiplicative relation on $\E[\psi'(1)^\lambda]$ with different parameters. Specializing to $W_1=2$ or $W_1=\gamma^2/2$ and 
using the proved case of Theorem~\ref{thm-conformal-derivative0} with $\rho_-=W_1-2$, we obtain two functional equations on 
$\E[\psi'(1)^\lambda]$. In the $\rho_-$-variable, 
it is a pair of explicit shift equations relating the value of  $\E[\psi'(1)^\lambda]$ at $\rho_-$ to the value at  $\rho_-+\gamma^2/2 $ or $\rho_-+2 $.
Setting $\beta=Q+\frac{\gamma}{2}-\frac{W_1}{\gamma}$ as in~\eqref{eq:weight-beta}, the two shifts in $\rho_-$ 
transfer to $\beta\rta \beta+\frac{\gamma}{2}$ and $\beta\rta \beta+\frac{2}{\gamma}$, respectively. Interestingly, the numerical values $\frac{\gamma}{2},\frac{2}{\gamma}$ for the shifts turn out to be exactly those appearing in shift relations for DOZZ formula~~\cite{do-dozz,zz-dozz,Teschner-shift,krv-dozz} and other correlation functions in LCFT (see e.g.~\cite{krv-dozz,rz-boundary}).

Similarly as in the LCFT context,
if $\kappa=\gamma^2$ is irrational, then this pair of shift relations has a unique meromorphic solution. On the other hand, we can check  that the explicit function in Theorem~\ref{thm-conformal-derivative0} is such a solution. This gives Theorem~\ref{thm-conformal-derivative0} for  irrational $\kappa\in (0,4)$. By a standard continuity argument, it extends to all $\kappa\in (0,4]$.  Finally the result for $\kappa>4$ follows from the SLE duality~\cite{zhan-duality1,dubedat-duality,ig1}.

The core of the argument outlined above is to  compare boundary lengths of quantum surfaces 
on the two sides of the conformal welding equation~\eqref{eq:welding-alpha}. 
It is equally interesting to compare quantum area and to consider quantum surfaces with marked points in the bulk.
This idea is explored  in~\cite{ARS-FZZ} by the first and the third author with Remy to prove the
Fateev-Zamolodchikov-Zamolodchikov (FZZ) formula~\cite{fzz} for the one-point disk partition function of LCFT. 
Moreover in~\cite{AS-CLE} of the first and the third authors, this idea is used to prove two integrable results 
on the conformal loop ensemble (CLE). One result relates the  three-point correlation function of CLE on the sphere~\cite{cle-3pt-phy} to the DOZZ formula in LCFT. 
The other addresses a conjecture of Kenyon and Wilson (recorded in \cite{ssw-radii}) on the electrical thickness of CLE loops.

\bigskip

\noindent{\bf Organization of the paper.} In the rest of the paper, we first develop the idea of uniform embedding in Section~\ref{subsec:equivalence} and prove Theorem~\ref{thm-haar} in Section~\ref{sec:equivalence}. Then  in Section~\ref{sec-rz-lengths} we relate some explicit boundary  LCFT correlation functions computed by Remy and Zhu~\cite{rz-gmc-interval,rz-boundary} to variants of quantum disks. In Section~\ref{sec:welding} we prove the welding equation~\eqref{eq:welding-alpha}.
In Section~\ref{sec-shift} we prove Theorem~\ref{thm-conformal-derivative0} based on~\eqref{eq:welding-alpha} following the outline in Section~\ref{subsec:welding}.

\bigskip

\noindent\textbf{Acknowledgements.} 
We are also grateful to Manan Bhatia, Ewain Gwynne,  Matthis Lehmkuehler, Steffen Rohde, Scott Sheffield, Pu Yu, and  Dapeng Zhan for helpful discussions.
M.A. was partially supported by NSF grant DMS-1712862. 
N.H.\ was supported by Dr.\ Max R\"ossler, the Walter Haefner Foundation, and the ETH Z\"urich
Foundation, along with grant 175505 of the Swiss National Science Foundation.
X.S.\ was supported by  the Simons Foundation as a Junior Fellow at the Simons Society of Fellows,  and by the  NSF grant DMS-2027986 and the Career award 2046514.

\section{Quantum surface and Liouville field}\label{sec:equivalence}

In this section we develop the ideas outlined in Section~\ref{subsec:equivalence}. 	
In Sections~\ref{subsec-GFF} and~\ref{subsec-path} we review background on quantum surfaces and LCFT which will be used throughout the paper.
In Section~\ref{subsec-translations} we show that when two-pointed quantum disks are embedded in the strip or cylinder with uniformly chosen translation, the field is described by LCFT. In Section~\ref{subsec-3pt} we discuss how to add a third point   sampled from quantum measure, 
which will recover the main result in \cite{cercle-quantum-disk}. 
(The sphere case is treated in parallel in Appendix~\ref{app:sphere}.)
Finally in Section~\ref{subsec-haar} we prove Theorem~\ref{thm-haar}.

We will frequently consider non-probability measures and extend the terminology of probability theory to this setting. 
In particular, suppose $M$ is a measure on a measurable space $(\Omega, \cF)$ such that $M(\Omega)$ is not necessarily $1$, and $X$ is an $\cF$-measurable function. 
Then we say that $(\Omega, \cF)$ is a sample space and that $X$ is a random variable.  We call  the pushforward measure $M_X = X_*M$  the \emph{law} of $X$.
We say that $X$ is \emph{sampled} from $M_X$. We also write $\int f(x) \,M_X(dx)$ as $M_X[f]$ or $M_X[f(x)]$ for simplicity.
For a finite positive measure $M$, we denote its total mass by $|M|$
and let $M^{\#}=|M|^{-1}M$ denote the corresponding probability measure.

\subsection{Preliminaries on the Gaussian free field and quantum surfaces}\label{subsec-GFF}
We recall the Gaussian free field (GFF) on  the upper half-plane $\bbH$ and the horizontal strip  $\cS=\R \times (0,\pi)$.  For $\cX \in \{\bbH,\cS \}$, we fix a finite measure $m$ on $\cX$. 
Consider the Dirichlet inner product 
$\langle f,g\rangle_\nabla := (2 \pi)^{-1}\int_\cX \nabla f \cdot \nabla g$. 
Let $H(\cX)$  be the Hilbert space closure of $\{f \textrm{ is smooth on }\cX  \textrm{ and }   \int_\cX f\,dm =0 \}$
with  respect to $(\cdot, \cdot)_\nabla$.
Let $(\xi_i)_{i=1}^\infty$ be i.i.d. standard Gaussian  random variables and $(f_i)_{i=1}^\infty$ be  an orthonormal basis for $H(\cX)$. Then the summation 
\eqb\label{eq-ONB}
h_\cX:= \sum_i \xi_i f_i
\eqe
does not converge in $H(\cX)$  but a.s.\ converges in the space of distributions \cite[Section 4.1.4]{wedges}.
We call $h_\cX$ a GFF on $\cX$ with normalization $\int_\cX h \,dm=0$, and denote its law by $P_\cX$.

In this paper for each  $\cX \in \{\bbH,\cS \}$ we will only consider one normalization measure $m$. 
For $\cX=\bbH$, it is the uniform measure on the unit  semi-circle centered at the origin.
For $\cX=\cS$, it is the uniform measure on   $\{0\} \times (0,\pi)$.
This way, $h_\cS$ and $h_\bbH$ are related by the exponential map between $\cS$ and $\bbH$. It will be convenient to have their explicit covariance kernels $G_\cX(z,w)=\E[h_\cX(x) h_\cX(y)]$: 
\begin{align}\label{eq:covariance}
G_\bbH(z,w) = -\log |z-w| - \log|z-\ol w| + 2 \log|z|_+ + 2\log |w|_+. \\
	G_\cS(z,w)= -\log |e^z - e^w| - \log | e^z - e^{\ol w}| +  \max(2\Re z, 0) + \max(2\Re w, 0).\nonumber
\end{align}Here $|z|_+$ means $\max\{|z|,1\}$. Moreover, $G_\cX(z,w)=\E[h_\cX(x) h_\cX(y)]$ means that 
for any compactly supported  test function $\rho$ on $\cX$, the variance of  $(h,\rho)$  is $\iint_\cX G_\cX(z,w) \rho(z)\rho(w)\,d^2z\,d^2w$. 
See~\cite[Definition 1.1]{rz-boundary} for~\eqref{eq:covariance} in the case of $G_\bbH$, and the other identity follows from   $G_\cS(z,w) = G_\bbH(e^z, e^w)$.

We now recall the radial-lateral decomposition of $h_\cS$.
Let $H_1(\cS) \subset H(\cS)$ (resp. $H_2(\cS) \subset H(\cS)$) be the subspace of functions which are constant (resp. have mean zero) on $\{t\} \times [0,\pi]$  for each $t \in \R$.
This gives the orthogonal decomposition $H(\cS) = H_1(\cS) \oplus H_2(\cS)$. 
If we write $h_\cS=h^1_\cS+h^2_\cS$ with  $h^1_\cS\in H_1(\cS)$ and $h^2_\cS\in H_2(\cS)$, then $h^1_\cS$ and $h^2_\cS$ are independent. Moreover, $\{h^1_\cS(t)\}_{t\in \R}$ has the distribution of $\{B_{2t}\}_{t\ge \R}$ where $B_t$ is a standard two-sided Brownian motion.  See \cite[Section 4.1.6]{wedges} for more details.  

We now recall the concept of a quantum surface.
For $n\in \N$, consider tuples $(D, h, z_1,\cdots,z_n)$ such that $D\subset\C$ is a domain, $h$ is a distribution on $D$, and 
$z_i \in  \cup D \cup\bdy D$. Let $(\wt D, \wt h, \wt z_1,\cdots,\wt z_n)$ be another such tuple.
We say 
$$(D, h, z_1,\cdots,z_n )  \sim_\gamma (\wt D, \wt h, \wt z_1,\cdots,\wt z_n)$$
if there is a conformal map $\psi: \wt D \to D$ such that $\wt h=f \bullet_\gamma h =h \circ f^{-1} + Q  \log \left| (f^{-1})' \right|$ and $\psi(\wt z_i) = z_i$ for all $i$. 
An equivalence class for $ \sim_\gamma $ is called a quantum surface with $n$ marked points. 
We write $ (D, h, z_1,\cdots,z_n )/{\sim_\gamma}$ as the marked quantum surface represented by $(D, h, z_1,\cdots,z_n )$.
When it is clear from context, we simply let $(D, h, z_1,\cdots,z_n )$ denote the marked quantum surface it represents.

Suppose $\phi$ is a random function on $\bbH$ which can be written as $h+g$ where $h$ is sampled from $P_\bbH$ and
$g$ a possibly random  function that is continuous on $ \bbH \cup \partial \bbH$ except at finitely many points.
For $\eps > 0$ and $z \in  \bbH \cup \partial \bbH$, we write $\phi_\eps(z)$ for the average of $\phi$ on $\partial B_\eps(z) \cap \bbH$, and define the random measure $\mu_\phi^\eps:= \eps^{\gamma^2/2} e^{\gamma \phi_\eps(z)}\,d^2z$ on $\bbH$, where $d^2z$ is Lebesgue measure on $\bbH$. Almost surely, as $\eps \to 0$, the measures $\mu_\phi^\eps$ converge weakly to a limiting measure $\mu_\phi$ called the \emph{quantum area measure} \cite{shef-kpz,shef-wang-lqg-coord}. 
We also define the \emph{quantum boundary length measure} $\nu_\phi:= \lim_{\eps \to 0} \eps^{\gamma^2/4}e^{\frac\gamma2 \phi_\eps(x)} dx$.
Suppose $f: \bbH \to D$ is a conformal map and $\wt\phi=f \bullet_\gamma \phi$. If $D =\bbH$, 
then $\mu_{\wt \phi}$  is the pushforward of $\mu_\phi$ under $f$ and the same holds for $\nu_{\wt \phi}$.
We can use this to unambiguously extend the definition of the quantum area and boundary length measures  to any 
$(D,\wt \phi)$ that is equivalent to $(\bbH,\phi)$ as a quantum surface. 

We now recall various notions of quantum disk introduced in \cite{wedges}. 
\begin{definition}
	\label{def-thick-disk}
	For $W \geq \frac{\gamma^2}2$, let $\beta =  Q + \frac\gamma2  - \frac W\gamma<Q$. Let 
	\[Y_t =
	\left\{
	\begin{array}{ll}
		B_{2t} - (Q -\beta)t  & \mbox{if } t \geq 0 \\
		\wt B_{-2t} +(Q-\beta) t & \mbox{if } t < 0
	\end{array}
	\right. , \]
	where $(B_s)_{s \geq 0}$ is a standard Brownian motion  conditioned on $B_{2s} - (Q-\beta)s<0$ for all $s>0$,\footnote{Here we condition on a zero probability event. This can be made sense of via a limiting procedure.}  and $(\wt B_s)_{s \geq 0}$ is an independent copy of $(B_s)_{s \geq 0}$.
	Let $h^1(z) = Y_{\Re z}$ for each $z \in \cS$.
	Let $h^2_\cS$ be independent of $h^1$  and have the law of the projection of  $h_\cS$ onto $H_2(\cS)$. 
	Let $\wh h = h^1+h^2_\cS$.
	Let  $\mathbf c$ be a real number  sampled from $\frac\gamma2 e^{(\beta-Q)c}dc$ independent of $\wh h $ and $\phi=\wh h +\mathbf c$.
	Let $\cMtwo(W)$ be the infinite measure  describing the law of $(\cS, \phi, -\infty, +\infty)/{\sim_\gamma}$.
	We call a sample from $\cMtwo(W)$ a (two-pointed) \emph{quantum disk of weight $W$}.
\end{definition}

The parameter $\beta$ measures the magnitudes of log-singularities at the corresponding marked points. We use the weight $W$ as the chief parameter 
for its convenience in stating conformal welding results in Section~\ref{sec:welding}.
For  $\cMtwo(2)$  we have 
$\beta=\gamma$. In this case the marked points are \emph{quantum typical}, namely, conditioning on the quantum surface, 
the two marked points are sampled according to the quantum length and area measure, respectively; see the discussion below Definition~\ref{def-QD}.
This allows us to define  general quantum disks marked with quantum typical points.  
In the following definition we recall the convention that $M^{\#}=|M|^{-1}M$ is the probability measure proportional to a finite measure $M$.
\begin{definition}\label{def-QD}
	Let  $(\cS, \phi, +\infty,-\infty)/{\sim_\gamma}$ be a sample from $\cMtwo(2)$.
	Let $\QD$ be the law of $(\cS, \phi)/{\sim_\gamma}$ under the reweighted measure $\nu_\phi(\partial \cS)^{-2}\cMtwo(2)$.
	For integers $m,n\ge 0$,  let $(\cS,\phi)$ be a sample from $\mu_\phi(\cS)^m\nu_\phi(\partial \cS)^n\QD$, and then
	independently sample $z_1,\cdots, z_m$ and $w_1,\cdots, w_n$ according to $\mu_\phi^\#$ and $\nu_\phi^\#$,  respectively.
	Let $\QD_{m,n}$ be  the law of 
	$$
	(\cS, \phi, z_1,\cdots, z_m, w_1,\cdots, w_n)/{\sim_\gamma}.
	$$
	We  call a sample from $\QD_{m,n}$ a \emph{quantum disk with $m$ interior and $n$ boundary marked points}. 
\end{definition}

By  \cite[Propositions A.8]{wedges} $\cMtwo(2) = \QD_{0,2}$,
which means the  marked points on 
$\cMtwo(2)$ are quantum typical.

We conclude this subsection  with a remark on the function space that variants of the GFF  take values in, which applies throughout the paper. 
\begin{remark}\label{rmk:ftn-space}
	For $\cX\in \{\bbH, \cS\}$, let $g$ be a smooth metric on $\cX$ such that 
	the metric completion of $(\cX, g)$ is a compact Riemannian manifold.
	Let $H^1(\cX, g)$ be the Sobolev  space whose norm is the sum of the $L^2$-norm with respect to $(\cX,g)$ and the Dirichlet energy.
	Let $H^{-1}(\cX)$ be the dual space of $H^1(\cX, g)$. 
	Then the function space $H^{-1}(\cX)$ and its topology does not depend on the choice of $g$, and is  a Polish  (i.e.\ complete separable metric) space.
	Moreover,  the GFF measure $P_\cX$ is supported on $H^{-1}(\cX)$.  This follows from a straightforward adaptation of results in \cite{shef-gff,dubedat-coupling} as pointed out in~\cite[Section 2]{dkrv-lqg-sphere}.
	Random functions on $\cX$ in our paper, such as the ones in Definition~\ref{def-thick-disk} and~\ref{def-sphere} are the summation of two types: 1. a sample from $P_\cX$; 
	2. a function on $(\cX, g)$ that is continuous everywhere except having log singularities at finitely many points.
	Both of these  functions belong to $H^{-1}(\cX)$. So we view their laws as measures on the Polish space $H^{-1}(\cX)$.
\end{remark}

\subsection{Preliminaries on Liouville conformal field theory}\label{subsec-path}
In this section we review some random fields arising in the context of LCFT. 
We define the Liouville field on  $\cX \in \{ \bbH, \cS\}$ with boundary insertions following~\cite{hrv-disk,rz-boundary}.
We will not discuss bulk insertions as they are not needed here. 

\begin{definition}\label{def-LF-rz}
	Let $(h, \mathbf c)$ be sampled from $P_\bbH\times [e^{-Qc}dc]$ and set $\phi =  h(z) -2Q \log |z|_+ +\mathbf c$. 
	We write  $\LF_{\bbH}$ as the law of $\phi$ and call  a sample from   $\LF_{\bbH}$ a \emph{Liouville field on $\bbH$}.
\end{definition}

Let $(\beta_i,s_i) \in  \R \times \partial\bbH$ for $i = 1, \dots, m$, where $m \geq 1$ and the $s_i$ are distinct. 
The Liouville field with insertions  $(\beta_i,s_i)_{1\le i\le m}$  
is defined formally by $\prod_{i=1}^{m} e^{\frac{\beta_i}2 \phi(s_i)}\LF_{\bbH}(d\phi)$.
To make it rigorous we need to replace $e^{\frac\beta2 \phi(s)}$ by  the regularization $\eps^{\frac{\beta^2}4}e^{\frac\beta2\phi_\eps(s)}$
and send $\eps\to 0$. 
We first give a definition without taking limit and then justify it in the subsequent lemma.
\begin{definition}\label{def-RV-H} 
	Let $(\beta_i, s_i) \in  \R \times \partial \bbH$ for $i = 1, \dots, m$, where $m \geq 1$ and the $s_i$ are pairwise distinct. Let $(h, \mathbf c)$ be sampled from $C_{\bbH}^{(\beta_i, s_i)_i} P_\bbH\times [e^{(\frac12\sum_i \beta_i - Q)c}dc]$ where
	\[C_\bbH^{(\beta_i, s_i)_i} = \prod_{i=1}^m |s_i|_+^{-\beta_i(Q-\frac{\beta_i}2)} e^{\sum_{j=i+1}^m \frac{\beta_i\beta_j}4 G_\bbH(s_i,s_j)}. \]
	Let $\phi(z) = h(z) - 2Q \log |z|_+ + \sum_{i=1}^m \frac{\beta_i}2 G_\bbH(z, s_i) + \mathbf c$. We write $\LF_\bbH^{(\beta_i, s_i)_i}$ for the law of $\phi$ and call a sample from $\LF_\bbH^{(\beta_i, s_i)_i}$ the \emph{Liouville field on $\bbH$ with insertions $(\beta_i, s_i)_{1 \leq i \leq m}$}.
\end{definition}
\begin{lemma}\label{lem:inserting}
		Suppose $s\notin \{ s_1, \cdots, s_m \}$. Then 
		in  the topology of vague convergence of measures, we have 
		\eqb 
		\lim_{\eps\to 0}\eps^{\frac{\beta^2}4}e^{\frac\beta2 \phi_\eps(s)}\LF_{\bbH}^{(\beta_i,s_i)_i}  (d\phi) =\LF_{ \bbH}^{(\beta_i,s_i)_i, (\beta, s)}. 
		\eqe
	\end{lemma}
	\begin{proof}
			Consider bounded continuous functions $f$ on $H^{-1}(\bbH)$ and $g$ on $\R$, and suppose $g$ is compactly supported. For $h$ sampled from $P_\bbH$ let $\wt \phi := h + \sum_i \frac{\beta_i}2 G(\cdot, s_i) - 2Q \log |\cdot|_+$ and let $\E_{\bbH}$ denote the expectation over $P_\bbH$. Then
			\begin{align*}
				&\lim_{\eps \to 0} C_{ \bbH}^{(\beta_i, s_i)_i}\int_\R \E_{ \bbH}[  \eps^{\frac{\beta^2}4} e^{\frac\beta2 (\wt\phi_\eps(s) + c)} f(\wt \phi) g(c) ] e^{(\frac12\sum_i \beta_i - Q)c} \, dc \\
				&=  |s|_+^{\frac{\beta^2}2-Q\beta}e^{\frac14\sum_i \beta \beta_i G_\bbH(s, s_i)}C_{ \bbH}^{(\beta_i, s_i)_i} \\
				&\qquad\cdot\lim_{\eps \to 0} \int_\R \E_{ \bbH}[e^{\frac\beta2 h_\eps(s)}]^{-1}\E_{ \bbH}[e^{\frac\beta2  h_\eps(s)} f(\wt \phi) g(c) ] e^{(\frac\beta2 + \frac12\sum_i \beta_i - Q)c} \, dc \\
				&= C_{ \bbH}^{(\beta, s), (\beta_i, s_i)_i} \int_\R \E_{ \bbH}[ f(\wt \phi + \frac\beta2 G_\bbH(\cdot, s)) g(c) ] e^{(\frac\beta2+ \frac12\sum_i \beta_i - Q)c} \, dc.
			\end{align*}
			The first equality follows from expanding the definition of $\wt \phi$ and noting that $\Var(h_\eps(s)) = -2\log \eps + 4 \log |s|_+ + o(1)$ so 
			\[
			\E[e^{\frac\beta2 h_\eps(s)}] = (1+o(1)) \eps^{-\frac{\beta^2}4} |s|_+^{\frac{\beta^2}2}.
			\]
			 For the second equality, we have the prefactor 
			 $$|s|_+^{\frac{\beta^2}2-Q\beta}e^{\frac14\sum_i \beta\beta_i G_\C(s, s_i)}C_{ \bbH}^{(\beta_i, s_i)_i}  = C_{ \bbH}^{(\beta, s), (\beta_i, s_i)_i}$$ 
			 by definition.   Moreover, 
			Girsanov's theorem gives $\E_\bbH[e^{\frac\beta2 h_\eps(s)}]^{-1} \E_\bbH[e^{\frac\beta2 h_\eps(s)} f(\wt \phi)] = \E_\bbH[f(\wt \phi + \frac\beta2 G_{\bbH, \eps}(\cdot, s))]$ where $G_{\bbH, \eps} (w, s)$ is the average of $G_\bbH(w, \cdot)$ on $\partial B_\eps(s)\cap \bbH$. Since $\lim_{\eps \to 0} G_{\bbH, \eps} (\cdot, s) \to G_\bbH(\cdot, s)$ in $H^{-1}(\bbH)$, the equality follows from the bounded convergence theorem. 
	\end{proof}


Definitions~\ref{def-LF-rz} and~\ref{def-RV-H}  correspond to  the LCFT on $\bbH$ with background metric $g(x) = |x|_+^{-4}$, as defined in~\cite[Section 3.5]{hrv-disk}. See also \cite[Section 5.3]{rz-boundary} for more details. 
When the \emph{Seiberg bounds} $\sum \beta_i > 2Q, \beta_i < Q$ hold, the measure $e^{-\mu \mu_\phi(\bbH) - \mu_\partial \nu_\phi(\partial \bbH)} \LF_{ \bbH}^{(\beta_i, s_i)_i}(d\phi)$ is finite for
\emph{cosmological constants} $\mu, \mu_\partial  > 0$. Its total mass gives the Liouville correlation functions on $\bbH$. 
In this section the finiteness of $e^{-\mu \mu_\phi(\bbH) - \mu_\partial \nu_\phi(\partial \bbH)} \LF_{ \bbH}^{(\beta_i, s_i)_i}(d\phi)$ is irrelevant, 
so we do not put any constraint on $(\beta_i)_{1\le i\le m}$.

LCFT on the half-plane is conformally covariant. To state this, for a measure $M$  on distributions on a domain $D$,
and a conformal map $f: D \to \wt D$,  we define $f_* M$ as the pushforward of $M$ under the map 
$\phi \mapsto \phi\circ f^{-1}+Q\log|(f^{-1})'|$,
and recall  the conformal automorphism group $ \conf(\bbH)$ of  $\bbH$. 
\begin{proposition}\label{prop-hrv-invariance}
	For $\beta\in\R$, set $\Delta_\beta := \frac\beta2(Q - \frac\beta2)$. 
	Let $f\in \conf(\bbH) $ and  $(\beta_i, s_i)\in \R \times \partial \bbH$ be such that $f(s_i) \neq \infty$ for all $1\le i\le m$.  
	Then  $\LF_{\bbH}=  f_*\LF_{\bbH}$ and \[\LF_{\bbH}^{(\beta_i,f(s_i))_i} = \prod_{i=1}^m |f'(s_i)|^{-\Delta_{\beta_i}} f_*\LF_{\bbH}^{(\beta_i,s_i)_i}.\]
\end{proposition}
\begin{proof}
	Theorem 3.5 in \cite{hrv-disk} is stated for LCFT on the unit disk, but the result holds also for LCFT on $\bbH$ by their Proposition 3.7.  Rephrasing using $\bbH$, 
	in \cite[Theorem 3.5]{hrv-disk} they consider $e^{-\mu \mu_\phi(\bbH) - \mu_\partial \nu_\phi(\partial \bbH)} \LF_{ \bbH}^{(\beta_i, s_i)_i}(d\phi)$ for $\mu,\mu_\partial>0$.
	But this  readily implies the statement for $\mu=\mu_\partial = 0$, i.e., proves Proposition~\ref{prop-hrv-invariance}. 
\end{proof}

In Definition~\ref{def-RV-H} we did not consider the case $s_1=\infty$. We now give a definition of this field and check that it can be obtained by sending $s \to \infty$.

\begin{definition}\label{def:LF-H-inf}
	Let $\beta \in \R$ and $(\beta_i, s_i) \in  \R \times \partial \bbH$ for $i = 2, \dots, m$, where $m \geq 1$ and the $s_i$ are pairwise distinct. Let $(h, \mathbf c)$ be sampled from $C_{\bbH}^{(\beta, \infty),(\beta_i, s_i)_i} P_\bbH\times [e^{(\frac\beta2 + \frac12\sum_i \beta_i - Q)c}dc]$ where
	\[C_\bbH^{(\beta, \infty), (\beta_i, s_i)_i} = \prod_{i=2}^m |s_i|_+^{-\beta_i(Q-\frac{\beta_i}2 - \frac\beta2)} e^{\sum_{j=i+1}^m \frac{\beta_i\beta_j}4 G_\bbH(s_i,s_j)}. \]
	Let $\phi(z) = h(z) +(\beta- 2Q) \log |z|_+ + \sum_{i=2}^m \frac{\beta_i}2 G_\bbH(z, s_i) + \mathbf c$. We write $\LF_\bbH^{(\beta, \infty), (\beta_i, s_i)_i}$ for the law of $\phi$ and call a sample from $\LF_\bbH^{(\beta, \infty), (\beta_i, s_i)_i}$ the \emph{Liouville field on $\bbH$ with insertions $(\beta, \infty), (\beta_i, s_i)_{2 \leq i \leq m}$}.
\end{definition}

\begin{lemma}\label{lem:LF-H-inf}
	With notation as in Definition~\ref{def:LF-H-inf}, we have the convergence in the vague topology on measures on $H^{-1}(\bbH)$ (see Remark~\ref{rmk:ftn-space})
	\[\lim_{r \to +\infty} r^{\beta(Q-\frac\beta2)} \LF_{\bbH}^{(\beta, r), (\beta_i, s_i)_i} = \LF_{\bbH}^{(\beta, \infty), (\beta_i, s_i)_i}. \]
\end{lemma}

\begin{proof}
	In the topology of $H^{-1}(\bbH)$, we have $G_\bbH(\cdot, r) \to 2\log |\cdot|_+$ as $r \to +\infty$. Thus we have $h - 2Q \log |\cdot|_+ + \frac\beta2 G_\bbH(\cdot, r) + \sum_{i=2}^m \frac{\beta_i}2 G_\bbH(\cdot, s_i) \to h +(\beta- 2Q) \log |\cdot|_+ + \sum_{i=2}^m \frac{\beta_i}2 G_\bbH(\cdot, s_i)$, and moreover $\lim_{r \to +\infty} r^{\beta(Q-\frac\beta2)} C_{\bbH}^{(\beta, r), (\beta_i, s_i)_i} = C_{\bbH}^{(\beta, \infty), (\beta_i, s_i)}$.   This yields the result. 
\end{proof}

When $\beta_1=\beta_2$ it  is more convenient to put the field on the strip $\cS$ and  put these two  insertions at $\pm \infty$. 
We will use this in the three-point case.
\begin{definition}\label{def-RV}
	Let $(h, \mathbf c)$ be sampled from $C_\cS^{(\beta, \pm\infty), (\beta_3, s_3)} P_\cS \times [e^{(\beta + \frac{\beta_3}2 - Q)c}dc]$ where $\beta \in \R$ and $( \beta_3,s_3) \in \R \times \partial\mathcal \cS$, and 
	\[C_\cS^{(\beta, \pm\infty), (\beta_3,s_3)} = e^{(-\frac{\beta_3}2(Q-\frac{\beta_3}2) + \frac{\beta\beta_3}2)\left| \Re s_3 \right| }. \]
	Let $\phi(z) = h(z) - (Q-\beta)\left|\Re z\right| + \frac{\beta_3}2 G_\cS(z, s_3) + \mathbf c$. We write $\LF_\cS^{(\beta, \pm\infty), (\beta_3, s_3)}$ for the law of $\phi$.  In the special case $\beta_3=0$, we instead write $\LF_\cS^{(\beta, \pm\infty)}$.
\end{definition}

Our next lemma explains how the fields of Definitions~\ref{def-RV-H},~\ref{def:LF-H-inf}  and~\ref{def-RV} are related under change of coordinates. We state this for two specific choices of conformal maps, and in light of Proposition~\ref{prop-hrv-invariance}, this covers all cases. Let $\exp: \cS\to \bbH$ be the exponentiation map $\exp(z) = e^z$. 

\begin{lemma}\label{lem-equiv-C-cS}
	Let $\beta \in \R$ and $(\beta_3, s_3) \in \R \times \partial \cS$, then 
	\[\LF_\bbH^{(\beta, \infty), (\beta, 0), (\beta_3, e^{s_3})} = e^{ - \frac{\beta_3}{2}(Q - \frac{\beta_3}2) \Re s_3} \exp_*\LF_\cS^{(\beta, \pm\infty), (\beta_3, s_3)} .\]
	Similarly, if $\beta_1, \beta_2, \beta_3 \in \R$ and $f  \in \mathrm{Conf}(\bbH)$ satisfies $f(0) = 0, f(1) = 1, f(-1) = \infty$, then
	\[\LF_{\bbH}^{(\beta_1, \infty), (\beta_2, 0), (\beta_3, 1)} =  2^{\Delta_{\beta_1}-\Delta_{\beta_2} + \Delta_{\beta_3}} \cdot f_* \LF_\bbH^{(\beta_1, -1), (\beta_2, 0), (\beta_3, 1)}. \]
\end{lemma}

	\begin{proof}
		If $h$ is sampled from $P_\cS$ then $\wt h := h \circ \log$ has law $P_\bbH$, and  $G_\cS(z,w) = G_\bbH(e^z, e^w)$. Thus 
		\eqbn
		\begin{split}
			&\wt h (\cdot) + (\beta - 2Q) \log \left| \cdot \right|_+ + \frac\beta2 G_\bbH(\cdot, 0) + \frac{\beta_3}2 G_\bbH(\cdot, e^{s_3})\\
			&\qquad\qquad\qquad = \exp \bullet_\gamma (h(\cdot) - (Q-\beta) \left|\Re \cdot \right| + \frac{\beta_3}2 G_\cS(\cdot, s_3)).
		\end{split}
		\eqen
		Combining this with $C_\bbH^{(\beta, \infty), (\beta, 0), (\beta_3, e^{s_3})} = e^{ - \frac{\beta_3}{2}(Q - \frac{\beta_3}2) \Re s_3} C_\cS^{(\beta, \pm\infty), (\beta_3, s_3)}$ gives the first assertion. 
		
		For $r>0$ let $f_r(z) := \frac{2rz}{(r+1)z + r-1}$, which is the conformal map such that $f_r(0) = 0, f_r(1) = 1, f_r(-1) = r$. By Proposition~\ref{prop-hrv-invariance} and using the $r \to \infty$ asymptotics $|f_r'(-1)| = (1+o_r(1))\frac{r^2}2$, $|f_r'(0)| = 2+o_r(1)$ and  $|f_r'(1)| = \frac12 + o_r(1)$, we have 
		\[ r^{2\Delta_{\beta_1}} \LF_{\bbH}^{(\beta_1, r), (\beta_2, 0), (\beta_3, 1)} = (1+o_r(1))  2^{\Delta_{\beta_1} - \Delta_{\beta_2} + \Delta_{\beta_3}}
		(f_r)_* \LF_{\bbH}^{(\beta_1, -1), (\beta_2, 0), (\beta_3, 1)} \]
		as $r \to \infty$.
		The $r \to \infty$ limit of the left hand side is $\LF_{\bbH}^{(\beta_1, \infty), (\beta_2, 0), (\beta_3, 1)}$ by Lemma~\ref{lem:LF-H-inf}. Similarly, since $f_r \to f$ in the topology of uniform convergence of an analytic function and its derivative on compact sets, we have $\lim_{r \to \infty}(f_r)_* \LF_{ \bbH}^{(\beta_1, -1), (\beta_2, 0), (\beta_3, 1)} = f_* \LF_{\bbH}^{(\beta_1, -1), (\beta_2, 0), (\beta_3, 1)}$ in the vague topology. This gives the second assertion.
	\end{proof}

We conclude  with an observation that is useful in Section~\ref{subsec-3pt}.
\begin{lemma}\label{lem-interpret-partition-fn}
	Let $\E_\cX$ denote the expectation over the probability measure $P_\cX$ for $\cX \in \{\bbH, \cS\}$. 
	Let $\E_\cS[\nu_h(du)]$ be the  measure on $\R$ given by $A \mapsto \E_\cS[\nu_h(A)]$. 
	We similarly define $\E_\bbH[\nu_h(du)]$. Then 
	\eqbn
	\begin{split}
		C^{(\beta, \pm\infty), (\gamma,u)}_\cS du &= e^{\frac{\gamma}2 (- (Q-\beta)\left|\Re u\right|)}  \E_\cS[\nu_h(du)], \quad \\
		C^{(\gamma,u)}_{\bbH} du &=e^{ -2\gamma Q \log |u|_+}  \E_\bbH[\nu_h(du)] .
	\end{split}
	\eqen
\end{lemma}
\begin{proof}
		We present the argument for the first identity; the other uses an identical argument. For any smooth compactly supported function $g: \R \to \R$, by \cite[Theorem 1.1]{berestycki-gmt-elementary} we have
		\eqb\label{eq-C-weighted}
		\begin{split}
			&\int_\R e^{\frac{\gamma}2 (- (Q-\beta)  \left|\Re u\right|)} g(u) \E_\cS[\nu_h(du)]\\
			&\qquad\qquad\qquad= \lim_{\eps \to 0} \int_\R e^{\frac{\gamma}2 (- (Q-\beta)\left|\Re u\right|)}g(u)\E_\cS[\eps^{\gamma^2/4} e^{\frac\gamma2 h_\eps(u)}]\,du.
		\end{split}
		\eqe
		Now, since $\Var(h_\eps(u)) = -2\log \eps + 2 \left| \Re u \right|+ o_\eps(1)$, we have $\E_\cS[\eps^{\gamma^2/4} e^{\frac\gamma2 h_\eps(u)}] = (1+o_\eps(1))e^{\frac{\gamma^2}4\left| \Re u \right|}$, where the error terms are uniformly small for $u$ in the support of $g$.  Therefore the limit in~\eqref{eq-C-weighted} equals $\int_\R e^{\frac{\gamma}2 (- (Q-\beta)\left|\Re u\right|)}g(u) e^{\frac{\gamma^2}4\left| \Re u\right|} \, du = \int_\R g(u) C^{(\beta, \pm\infty), (\gamma,u)}_\cS\, du$, as desired. 
\end{proof}

\subsection{LCFT description of two-pointed quantum disks}\label{subsec-translations}
The main result of this section is the following theorem.
\begin{theorem}\label{thm-two-disk-equivalence}
	Fix $W > \frac{\gamma^2}2$.
	Let  $\phi$ be as in Definition~\ref{def-thick-disk} so that $(\cS, \phi, +\infty, -\infty)$ is an embedding of  a sample from $\cM_2^\disk(W)$.
	Let $T\in \R$ be sampled from the Lebesgue measure $dt$ independently of 	$\phi$.  
	Let $\wt \phi(z)=\phi (z+T)$. 
	Then the law of $\wt \phi$ is given by $ \frac\gamma{2(Q-\beta)^{2}} \LF_\cS^{(\beta,\pm\infty)}$ where $\beta = Q + \frac\gamma2 - \frac W\gamma$.
\end{theorem}
By Definition~\ref{def-RV},
the proof of Theorem~\ref{thm-two-disk-equivalence} reduces to the following proposition on Brownian motion. See Figure~\ref{fig-three-processes}.
\begin{proposition}\label{prop-Bessel-equiv}
	Fix $a > 0$. Then $\{X_1(t)\}_{t\in \R}$ and $\{X_2(t)\}_{t\in \R}$ defined below agree in  distribution.
	
	\begin{itemize}
		
		\item 
		Let $(\wh B_t)_{t \geq 0}$ be  standard Brownian motions conditioned on $\wh B_t-at<0$ for all $t>0$. 
		Let $ (\wt B_t)_{t \geq 0}$ be an independent copy of $(\wh B_t)_{t \geq 0}$. Let 
		\[ Y_t =
		\left\{
		\begin{array}{ll}
			\wh B_t -at  & \mbox{if } t \geq 0 \\
			\wt B_{-t}+at & \mbox{if } t < 0
		\end{array}
		\right. , \]
		
		Sample $(\mathbf c, T) \in \R^2$   from $e^{-2ac} \,dc\, dt$ independent of $Y$. Set $X_1(t) = Y_{t-T}+\mathbf c$ for $t \in \R$. 
		
		\item Let $(B_t)_{t \in \R}$ be standard two-sided Brownian motion with $B_0 = 0$.
		Sample $\mathbf c'$ from $\frac1{2a^2} e^{-2ac}dc$ independent of $B$.
		Set $X_2(t) = B_t - a|t| + \mathbf c'$ for $t \in \R$.
	\end{itemize}
\end{proposition}
The  starting point of the proof of Proposition~\ref{prop-Bessel-equiv} is the following lemma.
\begin{lemma}\label{lem:BM-above}
	Let $(\wt W_t)_{t \geq 0}$ be a standard Brownian motion conditioned on $\wt W_t - at < 0$ for all $t>0$. Let  $(W_t)_{t \geq 0}$ be a standard Brownian motion independent of $(\wt W_t)_{t \geq 0}$. For $M \in \R$,  let
	\[ A^M_t =
	\left\{
	\begin{array}{ll}
		W_t -at - M  & \mbox{if } t \geq 0 \\
		\wt W_{-t}+at - M & \mbox{if } t < 0
	\end{array}
	\right.  \]
	and $x $ be the a.s. unique time such that $A^M_x =\max_{t\in \R} A^M_t$.  Then $(Y_t + \mathbf c)_{t \in \R}$ conditioned on $\{\mathbf c > -M\}$ agrees in distribution with $(A^M_{t+x})_{t \in \R}$, where $Y_t$ and $\mathbf c$ are as defined in Proposition~\ref{prop-Bessel-equiv}.
\end{lemma}
\begin{proof}
	Consider the excursion measure $\Lambda$ away from zero of the Bessel process with dimension $(2-2a)$.  Let $\Lambda_M$ the probability measure corresponding to $\Lambda$ conditioning on the event that the maxima of the excursion is bigger than $e^M$. Then Lemma~\ref{lem:BM-above} follows from comparing  two ways of representing $\Lambda_M$ in terms of drifted Brownian motion. As explained in 
	Proposition 3.4 and Remark 3.7 in \cite{wedges},
	given a sample $e$ from $\Lambda_M$, if we reparameterize $\log e$ by its quadratic variation then it becomes a process on $\R$, which is well defined modulo horizontal translations.   If we  fix the process by  requiring that 0 is the  smallest time when it hits $-M$, then we get a  process whose law is the same as  $A^M$.
	If we fix the process by requiring that it achieves the maximal value at $t=0$, then we get a process whose law is the same as  the conditional law of $(Y_t + \mathbf c)_{t \in \R}$ conditioned on $\{\mathbf c > -M\}$.  This gives Lemma~\ref{lem:BM-above}.
\end{proof}

\begin{figure}[ht!]
	\begin{center}
		\includegraphics[scale=0.5]{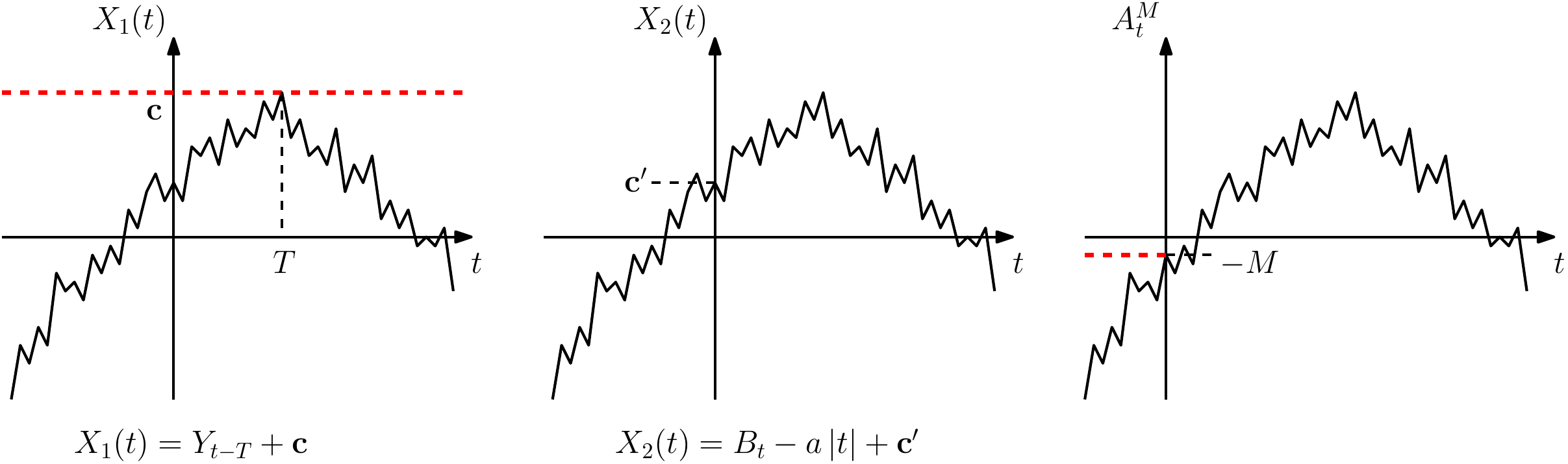}%
	\end{center}
	\caption{\label{fig-three-processes}  Illustration of the processes from   Proposition~\ref{prop-Bessel-equiv} and Lemma~\ref{lem:BM-above}.
\textbf{Left:}  The maximal value  of $X_1(t)$  is $\mathbf c$. It is achieved at time $T$ whose law is the Lebesgue measure on $\R$.
			\textbf{Middle:} On both $[0,\infty)$ and $(-\infty,0]$, the process $X_2(t)$ is a drifted Brownian motion  starting from $\mathbf c'$.
			\textbf{Right:}  The process  $\{A^M_t\}_{t\le 0}$ is a drifted Brownian motion conditioned on staying below $-M$, and  $\{A^M_t\}_{t\ge 0}$ is a drifted Brownian motion starting from $-M$.} 
\end{figure}

\begin{lemma}\label{lem:0}
	The law of $X_1(0) $   in Proposition~\ref{prop-Bessel-equiv}  is $\frac1{2a^2} e^{-2ac}dc$.
\end{lemma}
\begin{proof}
	We write $\BB M_1$ for the measure on the sample space that generates $X_1$. We must show that 
	$\BB M_1[X_1(0)>-M] = \int_{-M}^{\infty}\frac1{2a^2} e^{-2ac}\,dc= \frac1{4a^3} e^{2aM}$ for any $M\in \R$.
	By Lemma~\ref{lem:BM-above} and with the notations $A^M$ and $x$ defined there, we have 
	\[
	\BB M_1[X_1(0)>-M \mid  \mathbf c >-M]= \int_\R \P[A^M_{x+t} > -M] \, dt=\int_\R \P[A^M_{t} > -M] \, dt,
	\]where the last equality follows from Fubini's Theorem and the translation invariance of Lebesgue measure. 
	For each $t>0$ we have $\P[A^M_{t} > -M] = \P[W_t > at]= \P[Z > a\sqrt t]$ where $\{W_s: s \geq 0\}$ is standard Brownian motion and $Z \sim N(0,1)$, and for $t < 0$ we a.s.\ have $A^M_t \leq -M$. Therefore
		\[
		\int_\R \P[A^M_{t} > -M] \, dt=\int_0^\infty \int_{a\sqrt t}^\infty \frac1{\sqrt{2\pi}} e^{-x^2/2}\,dx \, dt  = \int_0^\infty \frac{x^2}{a^2} \frac1{\sqrt{2\pi}} e^{-x^2/2}\,dx = \frac1{2a^2}. 
		\]
	
	Since $\BB M_1[\mathbf{c}>  -M]=\int_{-M}^\infty e^{-2ac} \, dc = \frac1{2a} e^{2aM}$,  we conclude
	\[
	\BB M_1[X_1(0)>-M]=\BB M_1[X_1(0)>-M \mid  \mathbf c >-M]\BB M_1[\mathbf c>-M]= \frac1{4a^3} e^{2aM}. \qedhere
	\]
\end{proof}

Using Lemmas~\ref{lem:BM-above} and~\ref{lem:0} we can  show that the laws of $ X_1(0)$ and  $X_2(0)$ agree. 

\begin{proof}[Proof of Proposition~\ref{prop-Bessel-equiv}]
	In the setting of Lemma~\ref{lem:BM-above}, given $A^M$, 
	let $\tau$ be sampled from Lebesgue measure on $\R$.
	Then by Lemma~\ref{lem:BM-above}, 
	the conditional law of $\{X_1(t): t\in \R  \}$ given	$\mathbf c>-M$ is the same as the law of $\{A^M_{t+\tau}   :t\in  \R \}$. Consequently, the conditional law of $\{X_1(t): t\in \R  \}$ given $X_1(0) > -M$ is the same as the conditional law of $\{A^M_{t+\tau}   :t\in  \R \}$ given $A^M_\tau > -M$. By definition, on the event that $A^M_\tau > -M$, we must have $\tau > 0$. 
	By the Markov property of Brownian motion, conditioning on the event that $A^M_\tau>-M$ and the value of $A^M_\tau$, the processes $\{A^M_{t+\tau} -A^M_\tau   :t\ge 0 \}$ and $\{A^M_{t+\tau} -A^M_\tau   :t\le 0 \}$
	are conditionally independent. Moreover,  the conditional law  of $\{A^M_{t+\tau} -A^M_\tau   :t\ge 0 \}$ 
	equals the law of $(B_t -at)_{t\ge 0}$ where $B_t$ is a standard Brownian motion. 
	Varying $M$, we see that conditioning on $X_1(0)$, the conditional law of 
	$\{X_1(t)-X_1(0):t\ge 0\}$ and $\{X_1(t)-X_1(0):t\le 0\}$ are conditionally independent. Moreover, 
	the conditional law  of $\{X_1(t)-X_1(0):t\ge 0\}$
	equals the law of $(B_t -at)_{t\ge 0}$. 
	
	On the other hand, by the symmetry built into the definition of $X_1$, we see that   $\{X_1(-t): t\in \R\}$ has the same law as $\{X_1(t): t\in \R\}$.
	Therefore  conditioning on $X_1(0)$, the conditional law  of $\{X_1(-t)-X_1(0):t\ge 0\}$
	is also given by $(B_t -at)_{t\ge 0}$. Since by Lemma~\ref{lem:0} the law of $X_1(0)$ is the same as $X_2(0)$, by the definition of $X_2$ we see that the law of 
	$X_1$ is the same as that of  $X_2$. 
\end{proof}

\begin{proof}[Proof of Theorem~\ref{thm-two-disk-equivalence}]
Consider Proposition~\ref{prop-Bessel-equiv} where we have replaced  $e^{-2ac} \, dc \, dt$ with  $\frac\gamma4 e^{-2ac} \, dc \, dt$ in the definition of $X_1$, and replaced $\frac{1}{2a^2} e^{-2ac} \, dc$ with  $\frac{\gamma}{8a^2} e^{-2ac} \, dc$ in the definition of $X_2$. The proposition still holds since we are merely scaling both laws by $\frac\gamma4$. Choose $a = \frac12(Q-\beta)$ and let $X_1$ and $X_2$ be defined as in this modified setting.

	Recall $h^1,h^2_\cS,\mathbf c$ from Definition~\ref{def-thick-disk} so that $\phi=h^1+h^2_\cS+\mathbf c$.
	By definition, the law of $\{h^1(z+T)+\mathbf c\}_{z\in \cS}$ equals that of  $\{X_1(2\Re z)\}_{z\in \cS}$; the prefactor $\frac\gamma4$ in $\frac\gamma4 e^{-2ac} \, dc \, dt$ matches the product of $\frac\gamma2$ (from Definition~\ref{def-thick-disk}) and $\frac12$ (reparametrized Lebesgue measure in definition of $X_1$).
	Since the law of $h^2_\cS$ is translation invariant,   the law of  $\{\wt \phi(z) =\phi (z+T): z\in\cS \}$ agrees with 
	$\{X_1(2\Re z)+ h^2_\cS(z) :z\in \cS \}$, where $h^2_\cS$ is independently sampled from $X_1$.

	On the other hand, by Definition~\ref{def-RV},
	suppose $X_2(t)$ is independent of $h^2_\cS$, 
	then the law of $\{X_2(2\Re z)+ h^2_\cS(z):z\in \cC \} $ is $\frac{\gamma}{8a^2} \LF_\cS^{(\beta,\pm\infty)}$. Since $\frac{\gamma}{8a^2} =  \frac\gamma{2(Q-\beta)^2}$,
	and the laws of $X_1$  and $X_2$ agree by Proposition~\ref{prop-Bessel-equiv},  the law of $\wt \phi$ is $ \frac\gamma{2(Q-\beta)^2} \LF_\cS^{(\beta,\pm\infty)}$ as desired.
\end{proof}

\subsection{Adding a third point to a two-pointed quantum disk}\label{subsec-3pt}
In this section we show that $\LF_\cS^{(\beta, -\infty), (\beta, +\infty), (\gamma, 0)}$ describes a two-pointed quantum disk with an additional marked point defined as follows.
\begin{definition}\label{def-three-pt-thick}
	Fix $W > \frac{\gamma^2}2$.
	Let  $(D,\phi,a,b)$ be an embedding of a sample from  $\cM_2^\disk(W)$ and $\nu_\phi$ be the quantum length measure. 
	Let $L$ be the $\nu_\phi$-length of the right boundary of $(D,a,b)$, namely the counterclockwise arc from $a$ to $b$.
	Now suppose $(D,\phi,a,b)$ is from the reweighted measure $L\cM_2^\disk(W)$. Given $\phi$,  
	sample $\mathbf z$ from the probability measure proportional to the restriction of $\nu_\phi$ to the right boundary. 
	We write $\cMthree(W)$ as the law of  the marked quantum surface $(D,\phi,a,b,\mathbf z)/{\sim_\gamma}$.
\end{definition}

\begin{proposition}\label{prop-2-pt-weight-length}
	For $W > \frac{\gamma^2}2$, let $\phi$ be sampled from $\frac\gamma{2(Q-\beta)^{2}} \LF_\cS^{(\beta, -\infty), (\beta, +\infty), (\gamma, 0)}$ 
	where $\beta = \gamma + \frac2\gamma - \frac W \gamma$. 
	Then  $(\cS, \phi, -\infty, +\infty, 0)/{\sim_\gamma}$ is  a sample from  $\cMthree(W)$.
\end{proposition}

\begin{remark}\label{rmk:equivalence-disk}
Our $\cM_{2,\bullet}^\disk(2)$ equals $\QD_{0,3}$ restricted to the event $E$ that the three boundary points are in the clockwise order. Setting $\alpha=\gamma$
in Proposition~\ref{prop-2-pt-weight-length} and using the change of coordinate from Proposition~\ref{prop-hrv-invariance} and Lemma~\ref{lem-equiv-C-cS} gives the following. Suppose $(\bbH, \phi,s_1,s_2,s_3)$ is an embedding of $\QD_{0,3}|_E$,
where $s_1, s_2, s_3$ are three fixed distinct clockwise-oriented points on $\partial \bbH$. Then the law of $\phi$ is $C\LF_\bbH^{(\gamma, s_1),(\gamma, s_2),(\gamma, s_3)}$
with $C=\frac\gamma{2(Q-\gamma)^{2}}$.  The main result of \cite{cercle-quantum-disk} is equivalent to this statement without identifying $C$. 
We can also recover the  result of \cite{ahs-sphere} on $\QS_3$, see Proposition~\ref{prop-AHS}.
\end{remark} 

To prove Proposition~\ref{prop-2-pt-weight-length}, we start with an infinite-measure variant of the rooted measure in LQG. The argument via Girsanov's theorem is standard  we give the full detail as variants of it will be used repeatedly.
In the statement we recall Remark~\ref{rmk:ftn-space} that $P_\cS$ is understood as a measure on $H^{-1}(\cS)$. Moreover, we write $\E_\cS$ as the expectation over $P_\cS$.
\begin{lemma}\label{lem-girsanov}
	Let $Q(dh,du)=\nu_h(du)P_\cS(dh)$ for $(h,u)\in H^{-1}(\cS)\times \R$. 
	Namely, $Q$ is  the (infinite) measure on $H^{-1}(\cS)\times \R$
	such that for  non-negative measurable functions $f$ on $H^{-1}(\cS)$ 
	and $g$ on $\R$ we have
	\[
	\int f(h)g(u)\, Q(dh,du)=  \int f(h)\left(\int_\R g(u)\nu_h(du)\right) \,P_\cS(dh).
	\]
	Let $\rho$ be such that $\rho(u)du=\E_\cS[\nu_h(du)]$ with the latter measure defined in Lemma~\ref{lem-interpret-partition-fn}.
	Then
	\[
	\int f(h)g(u)\,Q(dh,du)= \int_\R \E_\cS\left[f(h + \frac\gamma2 G_\cS(\cdot, u))\right] g(u)\rho(u) \,du.
	\]
\end{lemma}

\begin{proof}
	It suffices to assume that $g$ is  a compactly supported continuous function on $\R$ and $f$ is a  bounded and continuous function on $H^{-1}(\cS)$.
	For $\eps>0$, let $\nu_{h, \eps}(dx) = \eps^{\gamma^2/4} e^{\frac\gamma2 h_\eps(x)}dx$.	
	Since $\lim\limits_{\eps\to 0}\int_\R g(u)\,\nu_{h,\eps}(du)=\int_\R g(u)\,\nu_h(du)$ in $L^1$ with respect to $P_\cS$ (see e.g. \cite[Theorem 1.1]{berestycki-gmt-elementary}), we have 
	\begin{equation}\label{eq:eps-nu}
		\begin{split}
			\lim_{\eps \to 0}\int f(h)\left(\int_\R g(u)\,\nu_{h,\eps}(du)\right) \,P_\cS(dh)&=\int f(h)\left(\int_\R g(u)\,\nu_h(du)\right) \,P_\cS(dh)\\
			&= \int f(h)g(u) \,Q(dh,du).
		\end{split}
	\end{equation}
	Let $G_{\cS,\eps}(z, u)=\E_\cS[h(z)h^\eps(u)]$, where the latter is understood via the $\eps$-circle average of $G_\cS (z,)$.
	By Girsanov's theorem, the left side of~\eqref{eq:eps-nu} equals 
	$$
	\int_\R \E_\cS\left[f(h + \frac\gamma2 G_{\cS,\eps}(\cdot, u))\right] g(u)\E_\cS[\eps^{\gamma^2/4} e^{\frac\gamma2 h_\eps(u)}]\,du.
	$$
	Since $\rho(u)du=\lim\limits_{\eps\to 0}\E_\cS[\eps^{\gamma^2/4} e^{\frac\gamma2 h_\eps(u)}]du$ and 
	$
	\lim_{\eps\to 0}G_{\cS,\eps}(\cdot, u)=G_{\cS}(\cdot, u)
	$ 
	in $H^{-1}(\cS)$, we get the desired result.
\end{proof}
The following lemma is a variant of~Lemma~\ref{lem-girsanov} for Liouville fields. For notational convenience we use the notion 
$M[f(\phi)]= \int f(\phi) M(d\phi)$.
\begin{lemma}\label{lem-inserting}
	Let $\LF_\cS^{(\beta, \pm\infty),  (\beta_j,s_j)_j}$ be as in Definition~\ref{def-RV}. 
	Let $f$ and $g$ be non-negative measurable functions as in Lemma~\ref{lem-girsanov}. Then
	\begin{equation}\label{eq:girsanov-cS}
		\LF_\cS^{(\beta, \pm\infty)}\left[ f(\phi)\int_\R g(u)\nu_\phi\,(du)\right]=
		\int_\R  \LF_\cS^{(\beta, \pm\infty), (\gamma, u)}[f(\phi)]g(u)\,du.
	\end{equation}
\end{lemma}
\begin{proof}
	By Definition~\ref{def-RV} the left side of~\eqref{eq:girsanov-cS} can be written as 	
	\[
	\iint f(h - (Q-\beta)\left|\Re \cdot\right| +c)\left(\int_\R g(u)e^{\frac\gamma2\left(-(Q-\beta)\left|\Re u\right| +c\right)}\,\nu_h(du)\right) \,P_\cS(dh) e^{(\beta-Q)c}\,dc. 
	\] 
	By Lemma~\ref{lem-girsanov}, the integration over $ P_\cS$ with a fixed $c$ is given by 
	\begin{equation}\label{eq:fix-c}
		\iint f(h+ \frac\gamma2 G_\cS(\cdot, u) - (Q-\beta)\left|\Re \cdot\right| +c)  g(u)e^{\frac\gamma2(-(Q-\beta)\left|\Re u\right| +c)} \rho(u) \cdot e^{(\beta-Q)c} P_\cS(dh) du
	\end{equation}
	where $\rho(u)$ is as in Lemma~\ref{lem-girsanov}. Recall $C^{(\beta, \pm\infty), (\gamma,u)}_\cS$ in the definition of $\LF_\cS^{(\beta, \pm\infty), (\gamma, u)}$.
	By Lemma~\ref{lem-interpret-partition-fn} we have
	\[
	C^{(\beta, \pm\infty), (\gamma,u)}_\cS du =e^{\frac{\gamma}2 (- (Q-\beta)\left|\Re u\right|)}  \rho(u)du.
	\]
	Therefore 
	the integral in~\eqref{eq:fix-c} becomes 
	\[
	\iint f(h - (Q-\beta)\left|\Re \cdot\right| + \frac\gamma2 G_\cS(\cdot, u) +c)  g(u)e^{\frac\gamma2c}  \cdot e^{(\beta-Q)c} \cdot C^{(\beta, \pm\infty), (\gamma,u)}_\cS P_\cS(dh) du.
	\]
	Further integrating over $c$ we get $\int_\R  \LF_\cS^{(\beta, \pm\infty), (\gamma, u)}[f(\phi)]g(u)\,du$ as desired.  
\end{proof}

\begin{proof}[Proof of Proposition~\ref{prop-2-pt-weight-length}]
	Our proof is based on Theorem~\ref{thm-two-disk-equivalence} and~\eqref{eq:girsanov-cS}.
	For $u\in \R$, let $M^u_{\bullet}$ be such that $(\cS,\phi,-\infty +\infty,u)/{\sim_\gamma}$ is a sample from $\cM_{2,\bullet}^\disk$ if $\phi$ is a sample from  $M^u_{\bullet}$. 
	We must show that $M^0_{\bullet}  =\frac\gamma{2(Q-\beta)^{2}} \LF_\cS^{(\beta, \pm\infty), (\gamma, 0)}$.
	
	Let $M_0$ be the law of $\phi$ in Definition~\ref{def-thick-disk} where $(\cS, \phi, +\infty, -\infty)/{\sim_\gamma}$ is a sample from $\cM_2^\disk(W)$.
	Let $Q_0$ be the  measure on $\{(\phi,z): h \textrm{ is a distribution on }\cS, u\in \R \}$ such that for  non-negative measurable functions $f$ and $g$ we have
	\[
	\int f(\phi)g(z) \,Q_0(d\phi,dz)=  \int f(\phi)\left(\int_\R g(z)\,\nu_\phi(dz)\right) \,M_0(d\phi).
	\]
	Then the $Q_0$-law of $(\cS,\phi,-\infty,+\infty,z)/{\sim_\gamma}$ is $\cM_{2,\bullet}^\disk(W)$.
	Let $(\phi, \mathbf z)$ and $ T$ be sampled from $Q_0\times dt$, where $dt$ is the Lebesgue measure  on $\R$. Set $\mathbf u=\mathbf z-T$ and $\wt \phi (\cdot) = \phi(\cdot + T)$. Let $M$ be the law of $(\wt \phi , \mathbf u)$.  Then by definition
	\begin{equation}\label{eq:u1}
		M[f(\wt \phi)  g(u)]= \int_\R M^u_{\bullet }[f(\wt \phi)]  g(u)\, du.
	\end{equation}
	On the other hand, note that the $\nu_\phi(\R)^{-1}Q_0$-law of $\phi$ is $M_0$. Therefore, by Theorem~\ref{thm-two-disk-equivalence}, the law of $\wt \phi$ under $\nu_{\wt \phi}(\R)^{-1}M$ is $ \frac\gamma{2(Q-\beta)^{2}} \LF_\cS^{(\beta,\pm\infty)}$. Moreover, conditioning on $\wt \phi$, the conditional  law of $\mathbf u$ is 
	the probability measure proportional to $\nu_{\wt \phi}|_\R$. Therefore, 
	$$M[f(\wt \phi)  g(u)]= \frac\gamma{2(Q-\beta)^{2}} \LF_\cS^{(\beta, \pm\infty)}\left[ f(\wt\phi)\int_\R g(u)\,\nu_{\wt\phi}(du)\right].$$
	By Lemma~\ref{lem-inserting}, we have  
	\begin{equation}\label{eq:u2}
		M[f(\wt \phi)  g(u)]= \frac\gamma{2(Q-\beta)^{2}}  \int_\R  \LF_\cS^{(\beta, \pm\infty), (\gamma, u)}[f(\phi)]g(u)\,du.
	\end{equation}
	Combining~\eqref{eq:u1} and~\eqref{eq:u2} we get 
	$M^u_\bullet[f(\wt \phi)] = \frac\gamma{2(Q-\beta)^{2}}  \LF_\cS^{(\beta, \pm\infty), (\gamma, u)}[f(\wt\phi)]$, since $g$ can be arbitrary.
	Setting $u=0$ and varying $f$ we conclude the proof.
\end{proof}

\subsection{Uniform embedding of the quantum sphere and quantum disk}\label{subsec-haar}
With notation as in Section \ref{subsec:equivalence},  Theorem~\ref{thm-two-disk-equivalence} says that the uniform embedding of $\cM_2^\disk(W)$ in $(\cS, -\infty, +\infty)$ is given by a constant multiple of $\LF_\cS^{(\beta,\pm\infty)}$. More precisely
\eqbn
\begin{split}
	&\haar_{\cS,-\infty,+\infty}  \ltimes \cMtwo(W) = \frac\gamma{2(Q-\beta)^{2}} \LF_\cS^{(\beta,\pm\infty)},\qquad \beta = Q + \frac\gamma2 - \frac W\gamma. 
\end{split}
\eqen
It is in fact a general phenomenon that uniform embedding of random surfaces appearing in the framework of \cite{wedges}   are  given by a Liouville field.
In this section we  demonstrate this point by proving Theorem~\ref{thm-haar}, which concerns the uniform embedding of the quantum sphere and disk. Unlike the rest of the paper, which focuses on LQG surfaces with disk topology, in this subsection we treat the sphere and disk in parallel. 
	

We first give a precise definition of the notation~$\ltimes$ that represents uniform embedding. Let $G$ be a locally compact Lie group.  
Suppose $\Omega$ is a  Polish  space with a continuous $G$-action $(g,x)\mapsto g\cdot x$. Namely, $g_1\cdot (g_2\cdot x)=(g_1\cdot g_2) \cdot x$ for all $g_1,g_2\in G$ and $x\in \Omega$; moreover, $(g,x)\mapsto g\cdot x$ is continuous.
Let $\Omega/G$ be such that  $y\in \Omega/G$ if and only if $y=\{ gx: g\in G \}$ for some $x\in \Omega$.
We let $\pi:\Omega\rta \Omega/G$ be the quotient map and endow $\Omega/G$ with the quotient topology.
We endow the Borel  $\sigma$-algebra on $G,\Omega$ and $\Omega/G$.
Suppose $\haar$ is a \emph{right invariant Haar measure}. That is, $\int_G f(gh) \,\haar(dg)=\int_G f(g)\, \haar(dg)$ for each non-negative measurable function $f$ on $G$ and each $h\in G$. 
\begin{definition}\label{def:uniform}
	For each $y\in \Omega/G$, choose $x\in \pi^{-1}(y)$. We write $\haar \ltimes y$ for the pushforward measure of $\haar$ under $G\ni g\mapsto g\cdot x\in \Omega$, i.e.\ for each Borel $E \subset \Omega$ we have $\mathbf m \ltimes y (E) = \int_\Omega 1_{g\cdot x \in E} \mathbf \,m(dg)$.
	For a $\sigma$-finite measure $\wt \nu$ on $\Omega/G$,  we write the measure $\int_{\Omega/G} [\haar \ltimes y] \,\wt\nu(dy)$ as $\haar \ltimes \wt\nu$.  Namely, 
	$\haar \ltimes \wt\nu(E)=\int_{\Omega/G} [\haar \ltimes y](E) \,\wt\nu(dy)$ for each Borel set $E\subset\Omega$.
\end{definition}
\begin{lemma}\label{lem:haar-invariant}
	For each $\sigma$-finite measure $\nu$ on $\Omega$, let $\pi_*\nu$ be the pushforward of $\nu$ by $\pi$.
	Then the pushforward of $\haar\times\nu$ under $(g,x)\mapsto g\cdot x$ equals $\haar\ltimes \pi_*\nu$.
\end{lemma}
\begin{proof}
	For each non-negative continuous function $f$ on $\Omega$,  let $I_f(x)=\int_{G} f(g\cdot x)\, \haar(dg)$. 
	Since $\haar$ is right invariant, $I_f(x)$ only depends on $\pi (x)$. 
	Equivalently, there exists a non-negative  continuous function $\wt I_f$ on $\Omega/G$ such that $I_f=\wt I_f\circ \pi$.  
	For each $\sigma$-finite measure $\nu$ on $(\Omega,\cF)$, by Fubini's Theorem,
	\begin{equation}\label{eq:haar1}
		\int_{G\times \Omega} f(g\cdot x)\, \haar(dg) \nu(dx)= \int_\Omega I_f(x)\,\nu(dx)= \int_{\Omega/G} \wt I_f(y) \,\pi_*\nu(dy).
	\end{equation}
	The last integral in~\eqref{eq:haar1} is precisely $\int_{\Omega/G} \left( \int f(x') \,  \haar \ltimes y (dx') \right)\,\pi_*\nu(dy)$. This concludes the proof.
\end{proof}

The proof of Theorem~\ref{thm-haar} relies on the LCFT description of the three-pointed quantum disk $\QD_{0,3}$ and sphere $\QS_3$. The description of $\QD_{0,3}$ was obtained in~\cite{cercle-quantum-disk} which we recovered and refined in Proposition~\ref{prop-2-pt-weight-length} and Remark~\ref{rmk:equivalence-disk}. In Appendix~\ref{app:sphere}, following the same proof we recover and refine the LCFT description of $\QS_3$ obtained in~\cite{ahs-sphere}. In particular, we prove a more general result (Proposition \ref{prop-QS3-LF}) in analogy to Proposition~\ref{prop-2-pt-weight-length}.
The original definition of $\QS_3$ is recalled in  Appendix \ref{app:sphere} but the LCFT description in Proposition~\ref{prop-AHS} is what we need for the rest of this section. The unpointed quantum sphere $\QS$ in Theorem~\ref{thm-haar} can be obtained  by deweighting the cubic power of the total quantum area of a sample from $\QS_3$ and forgetting the three marked points; see Definition~\ref{def-QS}.

We first recall the basic setup of LCFT on $\C$ following~\cite{dkrv-lqg-sphere,krv-dozz,vargas-dozz-notes}, and then state the LCFT description of $\QS_3$ as Proposition~\ref{prop-AHS} (see Appendix \ref{app:sphere} for the proof). Let $P_\C$ be the law of the GFF on $\C$ normalized to have average zero on the unit circle, which has covariance kernel (see~\cite[(2.1), Remark 2.1]{krv-dozz})
	\begin{align}\label{eq:covariance-sph}
		G_\C(z,w) &= -\log|z-w| + \log|z|_+ + \log|w|_+.\nonumber
	\end{align}
	\begin{definition} 
		\label{def-LF-sphere}
		Let $(h, \mathbf c)$ be sampled from $P_\C \times [e^{-2Qc}dc]$ and set $\phi =  h(z) -2Q \log |z|_+ +\mathbf c$. 
		We write  $\LF_{\C}$ as the law of $\phi$ and call a sample from  $\LF_{\C}$  a \emph{Liouville field on $\C$.}
	\end{definition}
	\begin{definition}\label{def-RV-sph}
		Let $(\alpha_i,z_i) \in  \R \times \C$ for $i = 1, \dots, m$, where $m\ge 1$ and the $z_i$ are distinct. 
		Let $(h, \mathbf c)$ be sampled from $ C_{\C}^{(\alpha_i,z_i)_i}  P_\C \times [e^{(\sum_i \alpha_i  - 2Q)c}dc]$ where
		\[C_{\C}^{(\alpha_i,z_i)_i} = 
		\prod_{i=1}^m |z_i|_+^{-\alpha_i(2Q -\alpha_i)} e^{\sum_{j=i+1}^m \alpha_i \alpha_j G_\C(z_i, z_j)}.\]
		Let \(\phi(z) = h(z) -2Q \log |z|_+  + \sum_{i=1}^m \alpha_i G_\C(z, z_i) + \mathbf c\).
		We write  $\LF_{\C}^{(\alpha_i,z_i)_i}$ for the law of $\phi$ and call a sample from  $\LF_{\C}^{(\alpha_i,z_i)_i}$ 
		the \emph{Liouville field on $\C$ with insertions $(\alpha_i,z_i)_{1\le i\le m}$}.
	\end{definition}	
\begin{proposition}\label{prop-AHS}
	Suppose $(\C, \phi,u_1,u_2,u_3)$ is an embedding of $\QS_3$,
	where $u_1, u_2, u_3$ are three fixed distinct points on $\C$. Then the law of $\phi$ is $\frac{\pi \gamma}{2 (Q-\gamma)^{2}}\LF_\C^{(\gamma, u_1),(\gamma, u_2),(\gamma, u_3)}$.
\end{proposition}

To make sense of $\haar_{\wh\C} \ltimes \QS$, consider the function space $H^{-1}(\C)$ defined as $H^{-1}(\bbH)$ in Remark~\ref{rmk:ftn-space} with $\C$ in place of $\bbH$.  The conformal coordinate  change \(f \bullet_\gamma \phi = \phi \circ  f^{-1} + Q  \log \left| (f^{-1})' \right|\)  
	defines a continuous group action of $\conf(\wh \C)$ on $H^{-1}(\C)$, where $\conf(\wh \C)$
	is  conformal automorphism group  on $\wh \C$.  By the definition of quantum surface, we can view
	$\QS$ as a measure on $H^{-1}(\C)/\conf(\wh \C)$. Since $\conf(\wh \C)$ is a locally compact Lie group,
	it has a unique right invariant Haar measure modulo a multiplicative constant, and moreover, the measure is left invariant as well since $\conf(\wh\C)$ is unimodular; (see e.g.\ \cite[Corollary 5.5.5]{faraut-lie-groups}).
	From Definition~\ref{def:uniform}, we get the precise meaning of  $\haar_{\wh\C} \ltimes \QS$. The uniform embedding $\haar_{\bbH}\ltimes \QD$ of $\QD$ is defined in the same way.

The starting point of the proof of Theorem~\ref{thm-haar} is the LCFT description of the uniform embedding of $\QS_3$ and $\QD_{0,3}$ instead of $\QS$  and $\QD$.
To make sense of $\haar_{\wh\C} \ltimes \QS_3$,  we view $\QS_3$ as a measure on the quotient space of $\Omega_{\C}\times \C^3$ under the $\conf(\wh \C)$-action
$(h,a,b,c)\mapsto (f\bullet_\gamma h, f(a),f(b),f(c))$. Then $\haar_{\wh\C} \ltimes \QS_3$ is a measure on $\Omega_{ \C} \times \C^3$. We similarly define $\haar_{\wh\C} \ltimes \QS_3$.
The following lemma gives a concrete realization of 
$\haar_{\bbH}\ltimes \QD_{0,3}$ and $\haar_{\wh\C} \ltimes \QS_3$.
\begin{lemma}\label{lem:haar3}
	Let $(\mathbb \C, \phi, a,b,c)$ be an embedding of a sample from $\QS_3$. 
	Let $\mathfrak{f}$ be a sample from  a Haar measure $\haar_{\wh\C}$ on $\conf(\wh\C)$ 
	that is  independent of $(\phi,a,b,c)$. 
	Then  the law of $(\mathfrak f\bullet_\gamma \phi, \mathfrak f(a),\mathfrak f(b),\mathfrak f(c))$ is $\haar_{\wh\C} \ltimes \QS_3$.
	In particular, it  does not depend on the law of $(\phi,a,b,c)$. Similarly, let $(\bbH, \phi, a, b, c)$ be an embedding of a sample from $\QD_{0,3}$, and $\mathfrak{g}$ an independent sample from a Haar measure $\mathbf m_{\bbH}$ on $\mathrm{conf}(\bbH)$. Then the law of $(\mathfrak g \bullet_\gamma \phi, \mathfrak g(a), \mathfrak g(b), \mathfrak g(c))$ is $\mathbf m_{\bbH} \ltimes \QD_{0,3}$.
\end{lemma}
\begin{proof}
	This immediately follows from Lemma~\ref{lem:haar-invariant}.
\end{proof}

We now give an explicit description of $\haar_{\wh \C}$ and $\haar_{\bbH}$.
\begin{lemma}\label{lem-haar}
	Let  $\mathfrak f$ be sampled from a Haar measure $\haar_{\wh \C}$	of $\conf(\wh\C)$. 
	Then there exists a constant   $C\in (0,\infty)$ such that the law of   
	$(\mathfrak{f}(0), \mathfrak{f}(1) , \mathfrak{f}(-1))$ is \(	C|(p-q)(q-r)(r-p)|^{-2} \,d^2p\, d^2q\, d^2r\).
	
		Similarly, let  $\mathfrak g$ be sampled from a Haar measure $\haar_{\bbH}$	of $\conf(\bbH)$.
		Then there exists a constant   $C\in (0,\infty)$ such that the law of   
		$(\mathfrak{g}(0), \mathfrak{g}(1) , \mathfrak{g}(-1))$ is \(	C|(p-q)(q-r)(r-p)|^{-1}\, dp\, dq\, dr\) restricted to the set of triples $(p,q,r)\in \R^3$ that are counterclockwise aligned on $\bdy\bbH$. 
\end{lemma}
\begin{proof}
	We prove the first assertion; the second follows from the same arguments.
	By the uniqueness of Haar measure, it suffices to show that if $(\mathbf p, \mathbf q, \mathbf r)$ is sampled from  \(|(p-q)(q-r)(r-p)|^{-2} \,d^2p\, d^2q\, d^2r\)
	and  $\mathfrak{f}$ is the unique Mobius  transformation mapping $(0,1,-1)$ to $ (\mathbf p,  \mathbf  q, \mathbf  r)$,
	then the law of $\mathfrak{f}$ is a Haar measure on $\conf(\wh\C)$. Namely for each $g \in\conf(\wh\C)$, $g\circ \mathfrak{f}$  agrees in law with $\mathfrak{f}$. 
	This is equivalent to the statement that $(g(\mathbf p),g(\mathbf q),g(\mathbf r))$ agrees in law with $ (\mathbf p,  \mathbf  q, \mathbf  r)$.
	This is straightforward to check when $g$ is a translation, dilation, or inversion.  Since these generate $\conf(\wh\C)$, we are done.
\end{proof}
We will give the LCFT description of $\haar_{\wh \C} \ltimes \QS_3$ and  $\haar_{\bbH} \ltimes \QD_{0,3}$ in Proposition \ref{prop:3pt-haar} below. In its proof we need  Proposition~\ref{prop-hrv-invariance} and
 its sphere counterpart, which we recall now.	
\begin{proposition}[{\cite[Theorem 3.5]{dkrv-lqg-sphere}}]\label{prop-RV-invariance}
	For $\alpha\in\R$, set $\Delta_\alpha := \frac\alpha2(Q - \frac\alpha2)$. 
	Let $f\in \conf(\wh\C) $ and  $(\alpha_i, z_i)\in \R \times \C$ be such that $f(z_i) \neq \infty$ for all $1\le i\le m$.  
Recall the notation $f_*$ in Proposition~\ref{prop-hrv-invariance}.	Then  \[\LF_{\C}=  f_*\LF_{\C} \quad \textrm{and}\quad \LF_{\C}^{(\alpha_i,f(z_i))_i} = \prod_{i=1}^m |f'(z_i)|^{-2\Delta_{\alpha_i}} f_*\LF_{\C}^{(\alpha_i,z_i)_i}.\]
\end{proposition}

\begin{proposition}\label{prop:3pt-haar}
	Suppose the Haar measures $\mathbf m_{\wh\C}, \mathbf m_{\bbH}$ are such that the constant $C$ in Lemma~\ref{lem-haar} is equal to 1. Then for non-negative measurable functions $f$ and $g$ on $H^{-1}(\C)$ and $\C^3$, respectively,
	\begin{equation}
		\label{eq-thm-QS3}
		\begin{split}
			&\haar_{\wh\C} \ltimes \QS_3[f(\phi)g(p,q,r)]\\
			&\qquad= \frac{\pi \gamma}{2(Q-\gamma)^2} \int_{\C^3}  \LF_{ \C}^{(\gamma,p),(\gamma,q),(\gamma,r)}[f(\phi)] g(p,q,r) \, d^2p\, d^2q\, d^2r,
		\end{split}
	\end{equation}
	and for non-negative measurable functions $f$ and $g$ on $H^{-1}(\bbH)$ and $\R^3$, respectively,
	\begin{equation}
		\begin{split}
			&\label{eq-thm-QD3}
			\haar_{\bbH} \ltimes \QD_{0,3}[f(\phi)g(p,q,r)] \\
			&\qquad= \frac\gamma{2(Q-\gamma)^2} \int_{\R^3}  \LF_{\bbH}^{(\gamma,p),(\gamma,q),(\gamma,r)}[f(\phi)] g(p,q,r) \, dp\, dq\, dr.
		\end{split}	
	\end{equation}
\end{proposition}

\begin{proof}
	We prove~\eqref{eq-thm-QS3}; the proof of~\eqref{eq-thm-QD3} is similar. 
	In Lemma~\ref{lem:haar3}, we choose $(a,b,c)=(0,1,-1)$. By Proposition~\ref{prop-QS3-LF}, the law of $\phi$
	is $\frac{2\pi \gamma}{(Q-\gamma)^2} \cdot \LF_{ \C}^{(\gamma,0),(\gamma,1),(\gamma,-1)}$. 
	Given three distinct points $p,q,r$ in $ \C^3$. Suppose $f\in \conf(\wh\C)$ maps $(0,1,-1)$ to $(p,q,r)$. Then we can explicit get
	$f(z) = \frac{(pq -2qr+rp)z + p(q-r)}{(2p-q-r)z + q-r}$  and
	\eqb \label{eq-fpqr-derivatives}
	\begin{split}
		&f'(0) = \frac{2(p-q)(q-r)(r-p)}{(q-r)^2},\qquad f'(1) = \frac{2(p-q)(q-r)(r-p)}{4(r-p)^2},\\ 
		&f'(-1) = \frac{2(p-q)(q-r)(r-p)}{4(p-q)^2}.
	\end{split}
	\eqe 	
	Recall notations from Proposition~\ref{prop-RV-invariance}. Since $\Delta_{\gamma}=1$, we have 
	\eqbn
	\begin{split}
		f_* \LF_{ \C}^{(\gamma, 0), (\gamma, 1), (\gamma,-1)} 
		&= |f'(0)f'(1)f'(-1)|^2\LF_{ \C}^{(\gamma, p), (\gamma, q), (\gamma, r)}\\
		&=C|(p-q)(q-r)(r-p)|^2 \LF_{ \C}^{(\gamma, p), (\gamma, q), (\gamma, r)}.
	\end{split}
	\eqen
	Since the law  of	$(\mathfrak{f}(0), \mathfrak{f}(1), \mathfrak{f}(-1))$ is $|(p-q)(q-r)(r-p)|^{-2}\, d^2p\, d^2q\, d^2r$, and 
	$\haar_{\wh\C} \ltimes \QS_3$ describes the law of $(\mathfrak f\bullet_\gamma \phi, \mathfrak f(0),\mathfrak f(1),\mathfrak f(-1))$, we obtain~\eqref{eq-thm-QS3}.
\end{proof}

To pass from the uniform embedding of  $\QD_{0,3}$ and $\QS_3$ to that of $\QD$  and $\QS$, we need Lemma~\ref{lem-inserting} and its sphere counterpart, which we state below and prove in Appendix \ref{app:sphere}.
\begin{lemma}\label{lem-inserting-general}
	We have
	\[
	\LF_{\C}^{(\alpha_i,z_i)_i}\left[ f(\phi)\int_\C g(u)\,\mu_\phi(du)\right]=
	\int_\C  \LF_{\C}^{(\alpha_i,z_i)_i, (\gamma, u)}[f(\phi)]g(u)\,d^2 u.
	\]
	for non-negative measurable functions $f$ and $g$.
\end{lemma}

\begin{proposition}\label{prop-1.2-constants}
	Suppose the Haar measures $\mathbf m_{\wh\C}, \mathbf m_{\bbH}$ are such that the constant $C$ in Lemma~\ref{lem-haar} is equal to 1, then 
	\[\mathbf m_{\wh\C} \ltimes \QS = \frac{\pi \gamma}{2(Q-\gamma)^2} \LF_\C \quad \text{ and }\quad \mathbf m_{\bbH} \ltimes \QD = \frac{\gamma}{2(Q-\gamma)^2} \LF_\bbH.\]
\end{proposition}

\begin{proof}
	Repeatedly applying Lemma~\ref{lem-inserting-general}, we get 
	\begin{equation}\label{eq:add-3pt}
		\begin{split}
			&\LF_{\C}\left[ f(\phi)\int_{\C^3} g(p,q,r)\mu_\phi(dp)\mu_\phi(dq)\mu_\phi(dr)\right]\\
			&\qquad\qquad=
			\int_{\C^3}  \LF_{ \C}^{(\gamma,p),(\gamma,q),(\gamma,r)}[f(\phi)] g(p,q,r) \, d^2p\, d^2q\, d^2r.
		\end{split}
	\end{equation}
	Setting $g=1$ in Proposition~\ref{prop:3pt-haar} and~\eqref{eq:add-3pt}, we have 
	$\haar_{\wh\C} \ltimes \QS_3[f(\phi)] = \frac{\pi \gamma}{2(Q-\gamma)^2} \LF_{\C}\left[ f(\phi) \mu_\phi(\C)^3\right]$.
	By the definition of $\haar_{\wh\C} \ltimes \QS_3$ in Lemma~\ref{lem:haar3}
	the marginal law of the field under $\haar_{\wh\C} \ltimes \QS_3$ is  $\mu_\phi(\C)^3\haar_{\wh\C} \ltimes \QS$.
	Therefore $\haar_{\wh\C} \ltimes \QS=\frac{\pi \gamma}{2(Q-\gamma)^2} \LF_{\C}$. The proof of $\mathbf m_{\bbH} \ltimes \QD = \frac{\gamma}{2(Q-\gamma)^2} \LF_\bbH$ is identical.
\end{proof}

\begin{proof}[Proof of Theorem~\ref{thm-haar}]
	This follows from Proposition~\ref{prop-1.2-constants} by the uniqueness of Haar measure modulo multiplication by a constant. 
\end{proof}

Our proof of Theorem~\ref{thm-haar} demonstrates how to go from $\QS_3=C\LF_{\C}^{(\gamma,0),(\gamma,1),(\gamma,-1)}$ 
to $\haar_{\wh\C}\ltimes \QS= C\LF_{ \C}$ through Proposition~\ref{prop:3pt-haar} and de-weighting. 
Similar arguments can also give results such as $\haar_{\wh\C, 0} \ltimes \QS_1 = C\LF_{\C}^{(\gamma,0)}$, where $\haar_{\wh\C,0}$
is a Haar measure on the subgroup of $\conf(\wh \C)$ fixing $0$. We do not need these statements so we omit the details. 

\section{Quantum surface and Liouville correlation function}\label{sec-rz-lengths}

In this section we consider disks with two or three marked boundary points and we derive the law of boundary lengths of these surfaces.  More specifically, we consider surfaces sampled from $\cM_{0,2}^\disk(W)$ and $\cMthree(W;\alpha)$, both thick and thin variants. The proofs are based on the integrability of boundary LCFT from \cite{rz-boundary}. Interestingly, we will see in Propositions \ref{prop:thin-length-law} and \ref{prop-disk-03-length} below that the same formulas apply for thick and thin variants of the same disk, which provides a probabilistic interpretation of identities satisfied by the reflection coefficient of boundary LCFT.

\subsection{Reflection coefficient and thick quantum disk}\label{sub-R-thick}
We recall the double gamma function $\Gamma_b(z)$ which is prevalent in LCFT. See e.g.\ \cite{spreafico2009} for more detail.
For $b>0$, $\Gamma_b(z)$ is the meromorphic function in $\C$ such  that for $\Re z > 0$,
\begin{equation}\label{eq:double-gamma}
\ln \Gamma_b(z) = \int_0^\infty \frac1t \left( \frac{e^{-zt}-e^{-(b+\frac1b)t/2}}{(1-e^{-b t})(1-e^{-\frac1b t})} - \frac{(\frac12(b+\frac1b) - z)^2}2 e^{-t} +\frac{z-\frac12(b+\frac1b)}t  \right)\, dt
\end{equation}
and it satisfies the shift equations
\eqb\label{eq-shift}
\frac{\Gamma_b(z)}{\Gamma_b(z+b)} := \frac1{\sqrt{2\pi}} \Gamma(b z) b^{-bz + \frac12}, \qquad \frac{\Gamma_b(z)}{\Gamma_b(z+\frac1b)} = \frac1{\sqrt{2\pi}} \Gamma(\frac1b z) (\frac1b)^{-\frac1b z + \frac12}.
\eqe
These shift equations allow us to meromorphically extend $\Gamma_b(z)$ from $\{ \Re z > 0\}$ to $\C$, where it has simple poles at $-n b - m \frac1b$ for nonnegative integers $m,n$. We also define the double sine function 
$$
S_b(z) := \frac{\Gamma_b(z)}{\Gamma_b(b+\frac1b-z)}.
$$
We will only work with $\Gamma_{\frac{\gamma}2}$, except in the proof of Lemma \ref{lem-thm-big}, where $\Gamma_{\frac2{\gamma}}$ also appears. Inspecting~\eqref{eq:double-gamma}, we see that  $\Gamma_{\frac2{\gamma}}=\Gamma_{\frac{\gamma}2}$.

We can now recall the boundary Liouville reflection coefficient from~\cite{rz-boundary}. 
For $\mu_1, \mu_2 > 0$, let $\sigma_j \in \C$ be such that $\mu_j = e^{i\pi \gamma (\sigma_j - \frac Q2)}$ and $\Re \sigma_j = \frac Q2$ for $j=1,2$. Let
\begin{align}
	\label{eq-R}
	\begin{split}
		\ol R(\beta, \mu_1,\mu_2) =&\, \frac{(2\pi)^{\frac2\gamma(Q-\beta)-\frac12} (\frac2\gamma)^{\frac\gamma2(Q-\beta)-\frac12}}{(Q-\beta)\Gamma(1-\frac{\gamma^2}4)^{\frac2\gamma(Q-\beta)}}\\&\qquad\cdot\frac{\Gg(\beta-\frac\gamma2)e^{i\pi(\sigma_1 + \sigma_2 - Q)(Q-\beta)}}{\Gg(Q-\beta)\Sg(\frac\beta2 +\sigma_2 -\sigma_1)\Sg(\frac\beta2 + \sigma_1 - \sigma_2)}.
	\end{split}
\end{align}
For $\mu>0$, let 
\begin{equation}\label{eq:R0}
\ol R(\beta, \mu, 0) = \ol R(\beta, 0,\mu)= \mu^{\frac2\gamma(Q-\beta)}\frac{(2\pi)^{\frac2\gamma(Q-\beta)-\frac12} (\frac2\gamma)^{\frac\gamma2(Q-\beta)-\frac12}}{(Q-\beta)\Gamma(1-\frac{\gamma^2}4)^{\frac2\gamma(Q-\beta)}} \frac{\Gg(\beta-\frac\gamma2)}{\Gg(Q-\beta)}.
\end{equation}

For $\mu_1, \mu_2 \geq 0$ not both zero, the meromorphic function $\beta \mapsto \ol R(\beta, \mu_1, \mu_2)$ is positive and finite in $(\frac\gamma2, Q + \frac\gamma2)$, has a pole at $\frac\gamma2$ and a zero at $Q + \frac\gamma2$ (along with other poles and zeros). In particular, although the term $\frac1{Q-\beta}$ suggests that there should be a pole at $\beta = Q$, this cancels with a zero coming from $\frac{1}{\Gg(Q-\beta)}$ to give $\ol R(Q, \mu_1, \mu_2) = 1$. 
The function $\ol R$ is called the normalized reflection coefficient. The unnormalized version is defined by
\begin{equation}\label{eq-R-real}
R(\beta, \mu_1, \mu_2) = -\Gamma(1-\frac2\gamma(Q-\beta))\ol R(\beta, \mu_1, \mu_2). 
\end{equation}

The following solvability result was proved in \cite{rz-boundary}; see Theorem 1.7 and Section 1.3 there.
\begin{proposition}[\cite{rz-boundary}]\label{prop-RZ-R}
	Let $\beta \in (\frac\gamma2, Q)$ and $W = \gamma( \gamma + \frac2\gamma - \beta)$, and let $\mu_1, \mu_2 \geq 0$ not both be zero. 
	Recall the field $\wh h$ from the definition of $\cM^\disk_2(W)$ in Definition~\ref{def-thick-disk}. We have 
	\[
	\E\left[\left( \mu_1 \nu_{\wh h}(\R) + \mu_2 \nu_{\wh h}(\R + \pi i)  \right)^{\frac2\gamma (Q-\beta)} \right] = \ol R(\beta, \mu_1,\mu_2).\]
	
\end{proposition}

\begin{remark}\label{rmk:gamma^22}
Proposition~\ref{prop-RZ-R} is only stated in \cite{rz-boundary} for $\beta\in(\frac\gamma2,Q)$. However, it extends to the case $\beta = Q$, where $W=\frac{\gamma^2}{2}$. In this case, it simply says that a zeroth moment is equal to $1 = \ol R(Q, \mu_1, \mu_2)$. 
When $\beta \leq \frac \gamma2$ the expectation is infinite, since $\frac2\gamma(Q-\beta) \geq \frac4{\gamma^2}$ and the  moment of the Gaussian multiplicative chaos of order at least $\frac4{\gamma^2}$ is infinite \cite[Section 2]{rhodes-vargas-review}.
\end{remark}

\begin{lemma}\label{lem-disk-perim-law}
	For $W \in [\frac{\gamma^2}2, \gamma Q)$ and $\beta = \gamma + \frac2\gamma - \frac W\gamma$, writing $L_1, L_2$ for  the left and right boundary lengths
	  of a quantum disk from $\cM_{0,2}^\disk(W)$, the law of $\mu_1 L_1 + \mu_2 L_2$ is 
	\[\1_{\ell>0} \ol R(\beta; \mu_1,\mu_2) \ell^{-\frac2{\gamma^2} W}\,d\ell. \]
\end{lemma}
\begin{proof}
	To simplify notation we explain the proof for $\mu_1 = 1$ and $\mu_2 = 0$ --- the general case follows identically. 
	For $0< \ell < \ell'$ we have 
	\eqb\label{eq-thick-length}
	\begin{split}
		\cM_{0,2}^\disk [ \nu_{{\wh h}+c}(\R) \in (\ell, \ell')] 
		&= \E\left[ \int_{-\infty}^\infty \1_{e^{\frac\gamma2c} \nu_{\wh h}(\R) \in (\ell, \ell')} \frac\gamma2 e^{(\beta - Q)c}dc \right]\\
		&= \E \left[\int_\ell^{\ell'} \nu_{\wh h}(\R)^{\frac2\gamma(Q-\beta)} y^{\frac2\gamma (\beta - Q)} \cdot  y^{-1} \, dy \right],
	\end{split}
	\eqe
	where we have used the change of variables $y = e^{\frac\gamma2c}\nu_{\wh h}(\R)$ (so $dc = \frac2\gamma y^{-1} dy$). Interchanging integral and expectation and applying Proposition~\ref{prop-RZ-R} and Remark~\ref{rmk:gamma^22}, we obtain the result. 
\end{proof}

For $W >\frac{\gamma^2}2$, the integral $\cM_2^\disk(W)[e^{-\mu_1 L_1 - \mu_2 L_2}]$ is infinite. The below proposition shows that we can obtain a finite integral by subtracting an appropriate polynomial which makes the integrand sufficiently small for small boundary lengths. Furthermore, the integral can be expressed in terms of the reflection coefficient $R$. We will see in Proposition \ref{prop:thin-length-law} below that the formula also extends to the case of thin disks, in which case it is not necessary to subtract a polynomial.
\begin{proposition}\label{prop:thick-length-law}
	For $W \in (\frac{\gamma^2}2, \gamma^2)$ and $\beta = \gamma + \frac2\gamma - \frac W\gamma$, and writing $L_1, L_2$ for
	 the left and right boundary lengths of a quantum disk from $\cMtwo(W)$, we have 
	\eqb
	\cM_2^\disk(W)[e^{-\mu_1 L_1 - \mu_2 L_2}-1] =  \frac\gamma{2 (Q-\beta)}  R(\beta; \mu_1, \mu_2).\eqe
	More generally, suppose $W \in (\frac{\gamma^2}2, \gamma Q)$ and there is a positive integer $n$ such that $W \in (n \frac{\gamma^2}2, (n+1)\frac{\gamma^2}2)$. Let $P_n(x) = \sum_{k=0}^{n-1} \frac{x^k}{k!}$ be the $n$-term Taylor polynomial of $e^x$. Then 
	\[\cMtwo(W)[e^{-\mu_1L_1 - \mu_2 L_2} - P_n(-\mu_1 L_1 - \mu_2 L_2)] = \frac\gamma{2(Q-\beta)} R(\beta, \mu_1, \mu_2). \]
\end{proposition}
\begin{proof} 
	Write $\alpha = \frac{2W}{\gamma^2}\in (1,2)$. Using integration by parts, we have 
	\[\int_0^\infty (1-e^{-\ell} ) \cdot \ell^{-\alpha} \,d\ell = \frac1{\alpha-1} \int_0^\infty e^{-\ell} \ell^{-(\alpha - 1)} \, d\ell =  \frac{\Gamma(2-\alpha)}{\alpha-1}.\]
	Thus, by Lemma~\ref{lem-disk-perim-law}, we have 
	\eqbn
	\begin{split}
		\cM_2^\disk(W)[1-e^{-\mu_1 L_1 - \mu_2 L_2}] 
		&= \ol R(\beta, \mu_1, \mu_2) \int_0^\infty (1-e^{-\ell} ) \cdot \ell^{-\alpha} \,d\ell\\ 
		&= \frac{\Gamma(2-\alpha)}{\alpha-1} \ol R(\beta, \mu_1, \mu_2).
	\end{split}
	\eqen
	The more general version similarly follows from the following identity for $n$ a positive integer and $p \in (n, n+1)$ 
	\[\int_0^\infty (e^{-\ell} - P_n(\ell)) \ell^{-p} \, d\ell = \Gamma(1-p).  \]
	Indeed, repeatedly integrating by parts gives
	\eqbn
	\begin{split}
		\int_0^\infty (e^{-\ell} - P_n(\ell)) \ell^{-p} \, d\ell &= \frac1{1-p} \int_0^\infty (e^{-\ell} - P_{n-1}(\ell)) \ell^{-(p-1)} \, d\ell = \dots \\
		&= \frac1{\prod_{k=1}^n (k-p)}\int_0^\infty e^{-\ell} \ell^{(p-n)} \, d\ell,
	\end{split}
	\eqen
	and since the last integral equals $\Gamma(n+1-p)$, repeatedly using $\Gamma(x+1) = x \Gamma(x)$ yields the result. 
\end{proof}

\subsection{Thin quantum disks and thick/thin duality}
\label{sec:thin-duality}
The reflection coefficient $R$ satisfies the following reflection identity; see~\cite[Eq (3.28)]{rz-boundary}.
\eqb\label{eq-reflection}
R(\beta; \mu_1, \mu_2)R(2Q - \beta; \mu_1, \mu_2) = 1.
\eqe
In Section~\ref{sub-R-thick}, we saw that for $\beta \in (\frac\gamma2, Q)$ the function $R$ describes quantum lengths for the thick quantum disk. 
In this section, we  give an analogous interpretation for $R$ in the regime $\beta \in (Q, Q+\frac\gamma2)$ via the \emph{thin quantum disk} defined in~\cite{ahs-disk-welding}. See Figure~\ref{fig-thin-disk-decomp}.

\begin{definition}[Thin quantum disk]\label{def-thin-disk}
	For $W \in (0, \frac{\gamma^2}2)$, we can define the infinite measure $\cMtwo(W)$ on two-pointed beaded surfaces as follows. 
	Sample $T$ from $ (1-\frac2{\gamma^2} W)^{-2}\mathrm{Leb}_{\R_+}$, then sample a Poisson point process $\{(u, \cD_u)\}$ from the measure $\1_{t\in [0,T]}dt \times \cMtwo(\gamma^2 - W)$, and concatenate the $\cD_u$'s according to the ordering induced by $u$.  We call a sample from $\cMtwo(W)$ a \emph{thin quantum disk} with weight $W$. We call the total sum of the left (resp., right) boundary lengths of all the  $\cD_u$'s
		the \emph{left (resp., right) boundary length} of the thin quantum disk.
\end{definition} 

\begin{figure}[ht!]
	\begin{center}
		\includegraphics[scale=0.8]{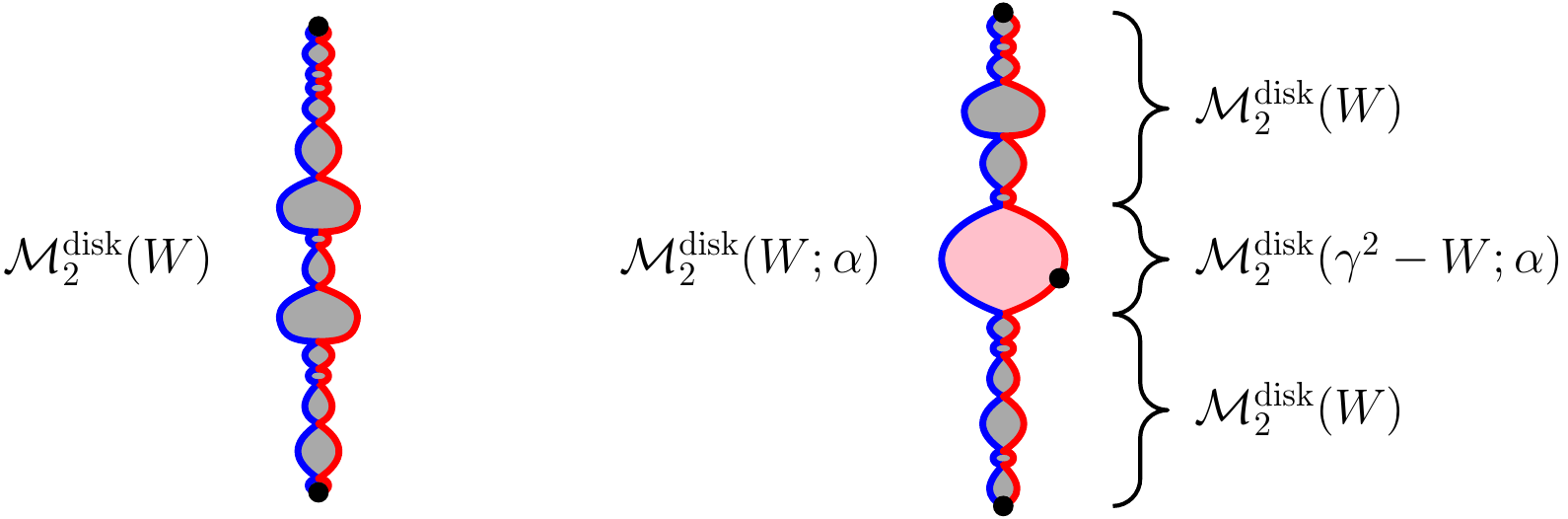}%
	\end{center}
	\caption{\label{fig-thin-disk-decomp} 
			Suppose $0< W< \frac{\gamma^2}2$. 
		 \textbf{Left:}  In Definition~\ref{def-thin-disk}, we define the weight $W$ thin quantum disk with weight $W$ via concatenation of an ordered Poissonian collection of weight $(\gamma^2-W)$ thick quantum disks. 
		\textbf{Right: }In Definition~\ref{def-disk-03} we define the measure $\cMthree(W; \alpha)$ on quantum surfaces obtained by sampling three quantum surfaces from $(1 - \frac2{\gamma^2}W)^2 \cMtwo(W) \times \cMthree(\gamma^2 - W;\alpha) \times \cMtwo(W)$ (depicted in grey, pink, grey) and concatenating them.
		\textbf{Both: } The length of the left boundary (depicted in blue) is given by the sum of the left boundary lengths of the constituent components, and the analogous statement is true for the length of the right boundary (depicted in red). 
		}
\end{figure}

The choice of the constant $(1-\frac2{\gamma^2} W)^{-2}$ above is justified by the following proposition, which states that the quantum disk boundary length distribution extends analytically from thick to thin quantum disks, hence giving a probabilistic meaning  to $\ol R$ and $R$ for $\beta\in  (Q, Q + \frac\gamma2)$.

\begin{proposition}\label{prop:thin-length-law}
	For $W \in ( 0, \frac{\gamma^2}2)$ and $\beta = \gamma +\frac2\gamma - \frac W\gamma \in (Q, Q + \frac\gamma2)$, let $L_1$ and $L_2$ be the left and right boundary lengths of a thin quantum disk from $\cM_2^\disk(W)$. For constants $\mu_1, \mu_2 \geq 0$ not both zero, the law of $\mu_1L_1 + \mu_2L_2$ is
	\[\1_{\ell > 0} \ol R(\beta, \mu_1, \mu_2) \ell^{-\frac{2}{\gamma^2} W}\,d\ell, \]
	and
	\[\cM_2^\disk(W)[e^{-\mu_1 L_1 - \mu_2 L_2}] = \frac\gamma{2 (Q-\beta)} R(\beta; \mu_1, \mu_2).\]
\end{proposition}
\begin{proof}
	Note that  $\cM_2^\disk(W)$ is defined using Poisson point processes. Our proposition will be an immediate consequence of Campbell's formula for the Laplace functional for Poisson point processes (see e.g.\  \cite[Section 3.2]{kingman-ppp}): For any measure space $(S,m)$ and measurable function $f: S \to (0,\infty)$ such that $\int_S \min (f(x), 1) \, m(dx) < \infty$, we have for a Poisson point process $\Pi$ on $(S,m)$ that
	\[\E[\exp(- \sum_{X \in \Pi} f(X))] = \exp\left(\int_S (e^{-f(x)}-1)\, m(dx) \right). \]
	For fixed $T>0$ we can set $S_T = [0,T] \times \wt S$ where $\wt S$ is the space of two-pointed quantum disks, and $m_T = \Leb_{[0,T]} \times \cMtwo(\gamma^2 - W)$. Letting $\Pi_T$ be a Poisson point process on $(S_T, m_T)$ and $f(\cD) = \mu_1 \ell_1 + \mu_2 \ell_2$ where $\ell_1, \ell_2$ are the quantum lengths of the boundary arcs of $\cD$, we have
	\[\E [ \exp(- \sum_{(t,\cD) \in \Pi_T} f(\cD))] = \exp\left(-T \cMtwo(\gamma^2-W)[1-e^{-f(\cD)}]\right).\]
	Integrating against $\1_{T>0}(1-\frac2{\gamma^2} W)^{-2} \, dT$, we get 
	\[\cMtwo(W)[e^{-\mu_1 L_1 - \mu_2 L_2}] = \frac{1}{(1-\frac2{\gamma^2} W)^2\cMtwo(\gamma^2-W)[1-e^{-f(\cD)}])}.  \]
	Proposition~\ref{prop:thick-length-law} gives $\cMtwo(\gamma^2-W)[1-e^{-f(\cD)}]) = \frac\gamma{2 (Q-\beta)}  R(2Q-\beta; \mu_1, \mu_2)$, and combining with the reflection identity~\eqref{eq-reflection} and $1-\frac{2W}{\gamma^2}=\frac{2(\beta-Q)}{\gamma}$ yields the second claim. 
	
	The first assertion then follows from the fact that $\mu_1 L_1 + \mu_2 L_2$ has a power law with exponent $-\frac2{\gamma^2}W$ \cite[Lemma 2.17]{ahs-disk-welding}, and a similar computation as in Proposition~\ref{prop:thick-length-law} to derive the coefficient of the power law.
\end{proof}

\subsection{Quantum disk with a third marked boundary point}
We consider the following variant of $R$ and $\ol R$ which has an additional parameter $\alpha$.
\alb
\ol H^{(\beta, \beta, \alpha)}_{(0,1,0)} &= \left(\frac{2\pi}{(\frac\gamma2)^{\frac{\gamma^2}4} \Gamma(1- \frac{\gamma^2}4)}\right)^{\frac2\gamma(Q - \beta - \frac12\alpha)}
\frac{\Gg(\frac12\alpha)^2 \Gg(Q - \beta + \frac12\alpha) \Gg(\beta+\frac12\alpha-\frac\gamma2)}{\Gg(\frac2\gamma) \Gg(Q-\beta)^2\Gg(\alpha)}
,\\
H^{(\beta, \beta, \alpha)}_{(0,1,0)} &= \frac2\gamma \Gamma(\frac2\gamma(\frac12\alpha + \beta - Q))\ol H^{(\beta, \beta, \alpha)}_{(0,1,0)}.
\ale
The notations $\ol H^{(\beta, \beta, \alpha)}_{(0,1,0)} $ and $H^{(\beta, \beta, \alpha)}_{(0,1,0)}$ are inherited from~\cite{rz-boundary} where more general more parameters are considered. The following proposition is proved in \cite{rz-gmc-interval}.

\begin{proposition}\label{prop-rz-H}
	Suppose $\alpha > 0,\frac\alpha2 + \beta > \frac\gamma2$, and  $\beta < Q$. 
	Let $\wt \psi = h + (\alpha - 2Q) \log |\cdot|_+ + \frac\beta2 G_\bbH(\cdot,0) + \frac\beta2 G_\bbH(\cdot,1)$, with 
	$h$ sampled from $ P_\bbH$. Then
	\[ \E\left[\nu_{\wt \psi}((0,1))^{\frac2\gamma(Q - \beta - \frac12\alpha)} \right] = \ol H^{(\beta, \beta, \alpha)}_{(0,1,0)}  .\]
\end{proposition}
\begin{proof}
The restriction of $\wt \psi$ to $(0,1)$ agrees with that of $h - \beta \log \left|\cdot\right| - \beta \log \left|\cdot - 1\right|$, so $\nu_{\wt \psi}(dx)|_{[0,1]} = x^{-\frac\gamma2 \beta} (1-x)^{-\frac\gamma2 \beta}\mu_h(dx)|_{[0,1]}$. Thus the moment we consider agrees with $M(\gamma, p, a, b)$ of \cite{rz-gmc-interval} with $a = b = -\frac\gamma2\beta$ and $p = \frac2\gamma(Q-\beta-\frac12\alpha)$, and \cite[Theorem 1.1]{rz-gmc-interval} shows this quantity equals $\ol H^{(\beta, \beta, \alpha)}_{(0,1,0)}$.
\end{proof}

Recall $\cMthree(W)$  from Definition~\ref{def-three-pt-thick}.
We now extend  $\cMthree(W)$ when $W> \frac{\gamma^2}2$ to have a third marked point with general $\alpha$ insertion.
\begin{definition}\label{def-3-disk-thick}
	For $W > \frac{\gamma^2}2$ and $\alpha \in \R$, let $\cMthree(W; \alpha)$ be the law on quantum surfaces $(\cS, \phi, -\infty, +\infty, 0)$ with $\phi $ sampled from $\frac\gamma2 (Q - \beta)^{-2} \LF_\cS^{(\beta, -\infty), (\beta, +\infty), (\alpha, 0)}$. 
	We call the boundary arc  between the two $\beta$ singularities which contains (resp.\ does not contain) the $\alpha$ singularity  the \emph{marked} (resp.\ \emph{unmarked}) boundary arc. 
\end{definition}
By Proposition~\ref{prop-2-pt-weight-length} we have $\cMthree(W; \gamma) = \cMthree(W)$. 
The next proposition describes the law of  the unmarked boundary arc of $\cMthree(W; \alpha) $ for some range of $\alpha,\beta$.

\begin{proposition}\label{prop-H-density}
	Suppose $\alpha > 0,\frac\alpha2 + \beta > \frac\gamma2$, and  $\beta < Q$ are as in Proposition~\ref{prop-rz-H}. When a quantum disk is sampled from $\cMthree(W;\alpha)$, the law of 
	its  unmarked boundary  length is 
	\eqb\label{eq-law-unmarked-arc}
	\1_{\ell > 0}(Q-\beta)^{-2} \ol H^{(\beta, \beta, \alpha)}_{(0,1,0)} \ell^{\frac2\gamma(\beta + \frac12\alpha - Q)-1}\, d\ell .
	\eqe
\end{proposition}
\begin{proof}
By Lemma~\ref{lem-equiv-C-cS} we have $\LF_\bbH^{(\beta, \infty), (\beta, 0), (\alpha, 1)} = \exp_* \LF_\cS^{(\beta, \pm\infty), (\alpha, 0)}$, and by Lemma~\ref{lem-equiv-C-cS} and Proposition~\ref{prop-hrv-invariance} we have $\LF_\bbH^{(\alpha, \infty), (\beta, 0), (\beta, 1)} = f_*\LF_\bbH^{(\beta, \infty), (\beta, 0), (\alpha, 1)}$ where $f \in \mathrm{Conf}(\bbH)$ is the conformal map with $f(0) = 1, f(1) = \infty, f(\infty) = 0$. Therefore, the law of $\nu_\phi(\R+ \pi i)$ with $\phi$ sampled from $\LF_\cS^{(\beta, \pm\infty), (\alpha, 0)}$ agrees with the law of $\nu_\psi((0,1))$ with $\psi$ sampled from $\LF_\bbH^{(\alpha, \infty), (\beta, 0), (\beta, 1)}$. Proposition~\ref{prop-rz-H} and the argument of Lemma~\ref{lem-disk-perim-law} show that $\nu_\psi((0,1))$ has law given by $\1_{\ell>0}\frac2\gamma \ol H^{(\beta, \beta, \alpha)}_{(0,1,0)} \ell^{\frac2\gamma(\beta + \frac12\alpha - Q)-1}\, d\ell$, so recalling the factor $\frac\gamma2 (Q-\beta)^{-2}$ in the definition of $\cMthree(W; \alpha)$, we obtain the stated result. 
\end{proof}
We note  that if $\alpha \geq Q$, then the quantum length of the marked boundary arc is a.s.\ infinite because the field blows up sufficiently quickly near the marked point. Nevertheless,  the unmarked boundary arc a.s.\ has finite quantum length as shown in Proposition~\ref{prop-H-density}.


The functions $R$ and $H$ are closely related as shown in~\cite[Lemma 3.4]{rz-boundary}. As a corollary of that relation we have 
\eqb \label{eq-reflection-H}
H^{(\beta,\beta, \alpha)}_{(0,1,0)} = R(\beta, 1,0)^2 H^{(2Q-\beta, 2Q-\beta, \alpha)}_{(0,1,0)} \quad \text{for all } \alpha, \beta \in \R. 
\eqe
We now give probabilistic meaning to~\eqref{eq-reflection-H} for some range of $\alpha $ and $\beta$.

We  first recall a fact from~\cite{ahs-disk-welding}  which will help us define a  variant of the  thin quantum disk with an additional $\alpha$-insertion.
\begin{lemma}[{\cite[Proposition 4.4]{ahs-disk-welding}}]\label{lem-thin-weighted}
	For $W \in (0, \frac{\gamma^2}2)$ we have
	\[\cMthree(W) = (1 - \frac2{\gamma^2}W)^2 \cMtwo(W) \times \cMthree(\gamma^2 - W) \times \cMtwo(W), \]
	where the right hand side is the infinite measure on ordered collection of  quantum surfaces obtained by concatenating samples from the three measures.
\end{lemma}

\begin{definition}\label{def-disk-03}
	Suppose $W \in (0, \frac{\gamma^2}2)$ and $\alpha \in \R$. Given a sample $(S_1,S_2,S_3)$ from  
	\[
	(1 - \frac2{\gamma^2}W)^2 \cMtwo(W) \times \cMthree(\gamma^2 - W;\alpha) \times \cMtwo(W),
	\]
	let $S$ be their  concatenation in the sense of Lemma~\ref{lem-thin-weighted} with $\alpha$ in place of $\gamma$. 
	We define the infinite measure $\cMthree(W; \alpha)$ to be the law of $S$. Let $L$ be the sum of the left boundary lengths of $S_1$ and $S_3$, and 
	the unmarked boundary length of $S_2$.  We call $L$ the \emph{unmarked boundary length} of $S$.
\end{definition}
See Figure~\ref{fig-thin-disk-decomp} for an illustration of Definition~\ref{def-disk-03}. The measure  $\cMthree(W; \alpha)$  does not naturally arise in either the quantum surface or the LCFT perspective, but is quite natural in our context.
The next proposition says that  $\ol H^{(\beta, \beta, \alpha)}_{(0,1,0)}$ describes the law of  its unmarked boundary length and gives a probabilistic realization of~\eqref{eq-reflection-H}.
\begin{proposition}\label{prop-disk-03-length}
	For $W \in (0, \frac{\gamma^2}2)$,  let $\beta = \gamma + \frac2\gamma - \frac W\gamma \in (Q, Q+\frac\gamma2)$. Suppose $\alpha > 2(\beta - Q)$. Then the law of the unmarked boundary length $L$ of a sample from  $\cMthree(W; \alpha)$ is
	\begin{equation}\label{eq:H-density}
	\1_{\ell > 0}(Q-\beta)^{-2} \ol H^{(\beta, \beta, \alpha)}_{(0,1,0)} \ell^{\frac2\gamma(\beta + \frac12\alpha - Q)-1}\, d\ell .
	\end{equation}
Moreover, for $\mu>0$, we have
\begin{equation}\label{eq:H-Laplace}
\cMthree(W; \alpha)[e^{-\mu L}] =  \frac\gamma2 (Q-\beta)^{-2} H^{(\beta, \beta, \alpha)}_{(0,1,0)} \mu^{-\frac2\gamma(\beta + \frac12\alpha - Q)}
\end{equation}
\end{proposition}
\begin{proof}
	By Proposition~\ref{prop-H-density} the law of the unmarked boundary length $L'$ of a sample from $\cMthree(\gamma^2 - W; \alpha)$ is $\frac2\gamma \ol H^{(2Q- \beta, 2Q- \beta, \alpha)}_{(0,1,0)} \ell^{b-1}\, d\ell$ with $b = \frac2\gamma (\frac12\alpha + Q-\beta)>0$. Therefore for $\mu>0$ we have
	\alb
	&\cMthree(\gamma^2 - W;\alpha) [e^{-\mu L'} ]\\
	&\qquad=(Q-\beta)^{-2} \ol H^{(2Q-\beta,2Q- \beta, \alpha)}_{(0,1,0)} \Gamma(\frac2\gamma (\frac12\alpha + Q-\beta)) \mu^{\frac2\gamma (\frac12\alpha + Q-\beta)} \\
	&\qquad=\frac{\gamma}{2}(Q-\beta)^{-2} H^{(2Q-\beta, 2Q -\beta, \alpha)}_{(0,1,0)}\mu^{\frac2\gamma(\beta-Q-\frac12\alpha)}.
	\ale
	Now by Definition~\ref{def-disk-03}, for $\mu>0$ we have
	\eqbn
	\begin{split}
		\cMthree(W; \alpha)[e^{-\mu L}]	=	&\,(1 - \frac2{\gamma^2}W)^2 \cMtwo(W) [e^{-\mu L_2}]\\ 
		&\times \cMthree(\gamma^2 - W;\alpha) [e^{-\mu L'}]
		\times \cMtwo(W)[e^{-\mu L_2}],
	\end{split}
	\eqen
	where $L_2$ in $ \cMtwo(W) [e^{-\mu L_2}]$ means the right boundary length of a sample  from $\cMtwo(W)$.
	By  Proposition~\ref{prop:thin-length-law}, we have 
	\[
	\cM_2^\disk(W)[e^{-\mu L_2}] = \frac\gamma{2 (Q-\beta)} R(\beta; 0, \mu)
	=\frac\gamma{2 (Q-\beta)} R(\beta; 0, 1)\mu^{\frac2\gamma(Q-\beta)}.
	\]
Using~\eqref{eq-reflection-H} we get~\eqref{eq:H-Laplace}, which further implies~\eqref{eq:H-density}. 
\end{proof}

\section{SLE observables via conformal welding}\label{sec:welding}
In this section, we prove Proposition~\ref{prop-curve-weight}, which is a conformal welding result. Although the measures involved are infinite, a constant of proportionality that arises is finite and encodes the information of the SLE observable in Theorem~\ref{thm-conformal-derivative0}.

\subsection{Conformal welding of quantum disks}
\label{sec:disk-welding}
In this section we recall the main result from our companion paper~\cite{ahs-disk-welding}, saying that $\SLE_\kappa(\rho_-;\rho_+)$  arise as the interface 
between two quantum disks  conformally welded  together.

We start by extending the definition of a quantum surface to the case where the surface is decorated by a curve. Recall from Section 
\ref{subsec-GFF} that a $\gamma$-LQG surface with $n$ marked points is defined to be an equivalence class of tuples $(D,h,z_1,\dots,z_n)$ where $D\subset\C$ is a domain, $h$ is a distribution on $D$, and $z_j\in \partial D\cup D$ for $j=1,\dots,n$. A curve-decorated quantum surface with $n$ marked points is similarly defined to be an equivalence class of tuples $(D,h,z_1,\dots,z_n,\eta)$ where $\eta:[0,t_\eta]\to D$ is a curve on $D$ and $t_\eta$ is the duration of $\eta$. More precisely, we say that $(D,h,z_1,\dots,z_n,\eta)\sim_\gamma(\wt D,\wt h,\wt z_1,\dots,\wt z_n,\wt\eta)$ if there is a conformal map $f:D\to\wt D$ such that  $\wt h = f\bullet_\gamma h$, $\wt z_j=f(z_j)$ for $j=1,\dots,n$, and $\wt\eta(t)=f(\eta(t))$. 

For $W>0$, let $\cMtwo(W; \ell, r)$ be the measure on weight $W$ quantum disks restricting to the event that the left and right boundary arcs have lengths $\ell $ and $ r$, respectively. 
More precisely, 
\eqb
\cMtwo(W)= \iint_0^\infty\cMtwo(W; \ell, r)\,d\ell\, dr.
\label{eq-disintegration}
\eqe
In particular, $|\cMtwo(W; \ell, r)|d\ell\, dr$ is  the law of the left and right boundary lengths, and the normalized probability measure $\cMtwo(W; \ell, r)^\#$ is $\cMtwo(W)$ conditioned on the boundary lengths being $\ell, r$. The identity \eqref{eq-disintegration} a priori only specifies $\cMtwo(W; \ell, r)$ for almost every $\ell, r$. 
But a canonical  version of $\{\cMtwo(W; \ell, r): \ell,r>0\}$ can be chosen such that
it is continuous in $\ell,r$ in a proper topology. See \cite[Section 2.6]{ahs-disk-welding} for details.

For fixed $\ell, r, x$, a pair of quantum disks from $\cMtwo(W_1; \ell, x) \times \cMtwo(W_2; x, r)$ can a.s.\ be \emph{conformally welded} along their length $x$ boundary arcs according to quantum length, yielding a quantum surface with two boundary marked points joined by an interface. This follows from the local absolute continuity of weight $W$ quantum disks with respect to weight $W$ quantum wedges, and the conformal welding theorem for quantum wedges \cite[Theorem 1.2]{wedges}. See e.g.\ \cite{shef-zipper}, \cite[Section 3.5]{wedges}, or \cite[Section 4.1]{ghs-mating-survey} for more information on conformal welding in the setting of LQG surfaces. 

For $W_1, W_2>0$, we now define an infinite measure $\cMtwo(W_1 + W_2; \ell, r) \otimes \SLE_\kappa (W_1 - 2; W_2 - 2)$ on curve-decorated quantum surfaces. 
When $W_1 + W_2 \geq \frac{\gamma^2}2$, we first sample $\phi$ such that the law of the $(\cS,\phi, -\infty, +\infty)$ viewed as a quantum surface is
$\cMtwo(W_1 + W_2; \ell, r)$ and then independently sampling an independent $\SLE_\kappa (W_1 - 2; W_2 - 2)$  curve $\eta$ in $(\cS, -\infty, +\infty)$ and parametrize $\eta$ by its quantum length. We denote the law of the curve-decorated surface $(\cS,\phi,\eta, -\infty, +\infty)$ by
$\cMtwo(W_1 + W_2; \ell, r) \otimes \SLE_\kappa (W_1 - 2; W_2 - 2)$.
When $W_1 + W_2 < \frac{\gamma^2}2$, we first sample a quantum surface with the topology of a chain of beads from  $\cMtwo(W_1 + W_2; \ell, r)$, then decorate each bead by an independent $\SLE_\kappa(W_1 - 2; W_2 - 2)$ between the two marked points of the bead. 
We denote the law of this chain of curve-decorated surfaces  $\cMtwo(W_1 + W_2; \ell, r) \otimes \SLE_\kappa (W_1 - 2; W_2 - 2)$.

The next result shows that the conformal welding of two quantum disks gives the  type of curve-decorated surface defined above.
For $W_-, W_+> 0$ and $\ell, x, r > 0$, we write 
$$\mathrm{Weld}(\cMtwo(W_-; \ell, x), \cMtwo(W_+; x, r))$$ 
for the measure on curve-decorated quantum surfaces obtained by first sampling $(\cD_-, \cD_+) $ from $\cMtwo(W_-; \ell, x) \times \cMtwo(W_+; x, r)$ 
and then conformally welding $\cD_-, \cD_+$ along their length $x$ boundary arcs. This conformal welding is a.e.\ well defined; see \cite[Theorem 2.2]{ahs-disk-welding} for details.
\begin{proposition}[{\cite[Theorem 2.2]{ahs-disk-welding}}]\label{prop-2pt-welding}
	Suppose $W_-, W_+ > 0$. There exists a constant $c_{W_-, W_+} \in (0, \infty)$ such that for all $\ell, r>0$ the following identity holds as measures on the space of curve-decorated quantum surfaces:
	\eqbn
	\begin{split}
		&\cMtwo(W_- + W_+; \ell, r) \otimes \SLE_\kappa (W_- - 2; W_+ - 2) \\ 
		&\qquad\qquad= c_{W_-, W_+} \int_0^\infty \mathrm{Weld}\left( \cMtwo(W_-; \ell, x) , \cMtwo(W_+; x, r)\right) \, dx. 
	\end{split}
	\eqen
\end{proposition}

\subsection{Conformal welding of $\cMtwo$ and $\cMthree$}
For $W >0$ with $W \neq \frac{\gamma^2}2$,  let $\cMtwo(W; \ell) := \int_0^\infty \cMtwo(W; \ell, r) \, dr$. 
Then  $\{ \cMtwo(W; \ell) \}_{\ell>0}$ is the 
disintegration  of $ \cMtwo(W)$ over its left  boundary length. Namely  samples from $\cMtwo(W; \ell)$ have left boundary length $\ell$ and
\(	\cMtwo(W) = \int_0^\infty \cMtwo(W; \ell) \, d\ell\).
Recall $\cMthree(W; \alpha)$ from Definitions~\ref{def-3-disk-thick} and~\ref{def-disk-03}, where we insert a third boundary marked point.
We now gives a concrete description of  its disintegration  over the unmarked boundary arc length. We start from the thick disk case $W>\frac{\gamma^2}{2}$.
\begin{lemma}\label{lem-fixed-bdy-length}
	For $W > \frac{\gamma^2}2$, $\beta = Q + \frac\gamma2 - \frac W\gamma$, and $\alpha> \max(0, \gamma - 2\beta)$, 
	sample $h$ from $P_\cS$ (the GFF on $\cS$). Let \(	\wt h = h - (Q-\beta) \left|\Re \cdot \right| + \frac\alpha2 G_\cS (\cdot, 0)\) and \( L = \nu_{\wt h}(\R + \pi i)\). 
	For $\ell > 0$, let $\LF_{\cS, \ell}^{(\beta, \pm\infty), (\alpha, 0)}$ be the law of $\wt h + \frac2\gamma\log \frac\ell L$ under the reweighted measure $\frac2\gamma \frac{\ell^{\frac2\gamma(\beta + \frac\alpha2 - Q) -1}}{L^{\frac2\gamma(\beta + \frac\alpha2 - Q)}} P_\cS(dh)$.
	Let $\cMthree(W; \alpha; \ell)$ be law of the  marked quantum surface $(\cS, \phi, -\infty, +\infty, 0)$ 
	where  $\phi  $ is sampled from $\frac\gamma2 (Q-\beta)^{-2}\LF_{\cS, \ell}^{(\beta, \pm\infty), (\alpha, 0)}$.  
	Then samples from $\cMthree(W; \alpha; \ell)$ have unmarked boundary arc length $\ell$ and
	\eqb\label{eq-disint-3pt}
	\begin{split}
		&\cMthree(W;\alpha) = \int_0^\infty \cMthree(W; \alpha; \ell) \, d\ell\qquad \textrm{and}\,\,\\  &|\cMthree(W; \alpha; \ell)|
		 = (Q-\beta)^{-2} \ol H^{(\beta, \beta, \alpha)}_{(0,1,0)} \ell^{\frac2\gamma(\beta + \frac12\alpha - Q)-1}.
	\end{split}
	\eqe
\end{lemma}
\begin{proof}
	The first assertion is clear since $\nu_{\wt h + \frac2\gamma \log \frac\ell L}(\R+\pi i) = \frac \ell L \nu_{\wt h} (\R+ \pi i) = \ell$. We now prove that
	\begin{equation}\label{eq:LF-disint}
	\LF_{\cS}^{(\beta, \pm\infty), (\alpha, 0)}=\int_0^\infty\LF_{\cS, \ell}^{(\beta, \pm\infty), (\alpha, 0)}\, d\ell.
	\end{equation}
	For any nonnegative measurable function $F$ on $H^{-1}(\cS)$ we have
	\eqbn
\begin{split}
	&\int_0^\infty \int F(\wt h + \frac2\gamma \log \frac\ell L) \frac2\gamma \frac{\ell^{\frac2\gamma(\beta + \frac\alpha2 - Q) -1}}{L^{\frac2\gamma(\beta + \frac\alpha2 - Q)}} \,P_\cS(dh)\, d\ell\\ 
	&\qquad\qquad\qquad\qquad= \int \int_\R F(\wt h + c) e^{(\beta + \frac\alpha2 - Q)c}\, dc \,P_\cS(dh) 
	\end{split}
	\eqen
	using Fubini's theorem and the change of variables $c = \frac2\gamma \log \frac \ell L$. Therefore \eqref{eq:LF-disint} holds. 
	By Definition~\ref{def-3-disk-thick} of $\cMthree(W;\alpha) $, we have $\cMthree(W;\alpha) = \int_0^\infty \cMthree(W; \alpha; \ell) \, d\ell$.
	The second identify in~\eqref{eq-disint-3pt} then directly  follows from Proposition~\ref{prop-H-density}. 
\end{proof}

If $W \in (0, \frac{\gamma^2}2)$, $\beta = Q + \frac\gamma2 - \frac W\gamma$, $\alpha > 0$ and $\frac12\alpha > \beta - Q$, then for each $\ell > 0$ we can similarly define the corresponding measure $\cMthree(W;\alpha; \ell)$ on quantum surfaces with unmarked boundary arc length $\ell$ via
\eqb\label{eq-thin-alpha-disk}
\begin{split}
	&\cMthree(W; \alpha; \ell):= (1 - \frac2{\gamma^2}W)^2 \int_0^\ell \int_0^{\ell - x} \cMtwo(W;x) \times \cMthree(\gamma^2 - W; \alpha; y)\\ &\qquad\qquad\qquad\qquad\qquad\qquad\qquad\qquad\qquad\times \cMtwo(W; \ell - x - y)\, dy \, dx
\end{split}
\eqe
where the integrand is understood as concatenation of surfaces in the sense of Lemma~\ref{lem-thin-weighted}.
\begin{lemma}
For $W\in (0,\frac{\gamma^2}{2})$, \eqref{eq-disint-3pt} still holds with the  $\cMthree(W; \alpha; \ell) $ defined above.
\end{lemma}
\begin{proof}
The first claim is immediate from Definition~\ref{def-disk-03}, and the second then follows from Proposition~\ref{prop-disk-03-length}.
\end{proof}

Recall  that the special case of $\cMthree(W;\alpha)$ with $\alpha=\gamma$ is $\cMthree(W)$ from Definition~\ref{def-three-pt-thick}.
We  now give  a variant of Proposition~\ref{prop-2pt-welding} for $\cMthree(W;\gamma)$.
 The $\mathrm{Weld}$ notation in our next two results is used analogously as in Proposition~\ref{prop-2pt-welding}.

\begin{lemma}\label{lem-curve-unweighted}
	For $W_-, W_+ > 0$ with $W_+, W_- + W_+ \neq \frac{\gamma^2}2$, there is a constant $c_{W_-, W_+} \in (0, \infty)$ such that
	for each $\ell>0$
	\eqb\label{eq-curve-unweighted}
	\begin{split}
		&\cMthree(W_- + W_+; \gamma; \ell)\otimes \SLE_\kappa(W_--2;W_+-2)\\ 
		&= c_{W_-, W_+}\int_0^\infty \mathrm{Weld}\left(\cMtwo(W_-; \ell, x), \cMthree(W_+; \gamma; x)\right)\,dx.
	\end{split}
	\eqe
\end{lemma}
\begin{proof}
	In Proposition~\ref{prop-2pt-welding}, sample a marked point from quantum length measure on the boundary arc of length $r$ (thus weighting by $r$). The result then follows from Proposition~\ref{prop-2-pt-weight-length} or Lemma~\ref{lem-thin-weighted}, depending on whether the quantum disks are thick or thin. 
\end{proof}

We now extend Lemma~\ref{lem-curve-unweighted} to $\cMthree(W;\alpha)$; see Figure~\ref{fig:conformal-derivative} for an illustration.
We first introduction an $\alpha$ variant of $\SLE_\kappa(W_--2;W_+-2)$.
Given a curve $\eta$ on $\cS$ from $-\infty$ to $\infty$,  let $D$ be the connected component of $\cS \backslash \eta$ containing $0$ on its boundary, and let $\psi_\eta : D \to \cS$ be the conformal map fixing $0$ and sending the first (resp. last) point on $\partial D$ hit by $\eta$ to $-\infty$ (resp. $+\infty$).
For $\alpha\in \R$, let $\Delta(\alpha) = \frac\alpha2 (Q - \frac\alpha2)$. 
For $W_-, W_+ > 0$, we define the measure $\sm(W_-, W_+, \alpha)$ 
on curves on  $\cS$ such that its Radon-Nikodym derivative with respect to $\SLE_\kappa(W_--2;W_+-2)$ is:
\[\frac{d\sm(W_-, W_+, \alpha)}{d\SLE_\kappa(W_--2;W_+-2)} (\eta) = \psi_\eta'(0)^{1-\Delta(\alpha)}.\]

When $W_- + W_+ \geq \frac{\gamma^2}2$, we define $\cMthree(W_- + W_+;\alpha; \ell)\otimes \sm(W_-, W_+, \alpha)$ 
in the exact same way as $	\cMthree(W_- + W_+; \gamma; \ell)\otimes \SLE_\kappa(W_--2;W_+-2) $ in Lemma~\ref{lem-curve-unweighted} 
 with $ \sm(W_-, W_+, \alpha)$ in place of $\SLE_\kappa(W_--2;W_+-2)$.
When $W_- + W_+ < \frac{\gamma^2}2$, we still define $\cMthree(W_- + W_+;\alpha; \ell)\otimes \sm(W_-, W_+, \alpha)$ as a chain of curve-decorated quantum surfaces
as  $	\cMthree(W_- + W_+; \gamma; \ell)\otimes \SLE_\kappa(W_--2;W_+-2) $, except that for the quantum surface containing the additional boundary marked point, 
we use $ \sm(W_-, W_+, \alpha)$ instead  of $\SLE_\kappa(W_--2;W_+-2)$ to decorate that surface.

\begin{proposition}\label{prop-curve-weight}
	For $W_- \geq \frac{\gamma^2}2$ and $W_+ > 0$ with $W_+\neq \frac{\gamma^2}2$, there is a constant $c_{W_-, W_+} \in (0, \infty)$ such that for all $\alpha \in \R$ and $\ell > 0$
	\eqb\label{eq-curve-weighted}
	\begin{split}
		&\cMthree(W_- + W_+;\alpha; \ell)\otimes \sm(W_-, W_+, \alpha) \\
		&=   c_{W_-, W_+} \int_0^\infty \mathrm{Weld}\left( \cMtwo(W_-; \ell, x) , \cMthree(W_+; \alpha; x) \right)\, dx.
	\end{split}
	\eqe
\end{proposition}
\begin{figure}[ht!]
	\begin{center}
		\includegraphics[scale=0.5]{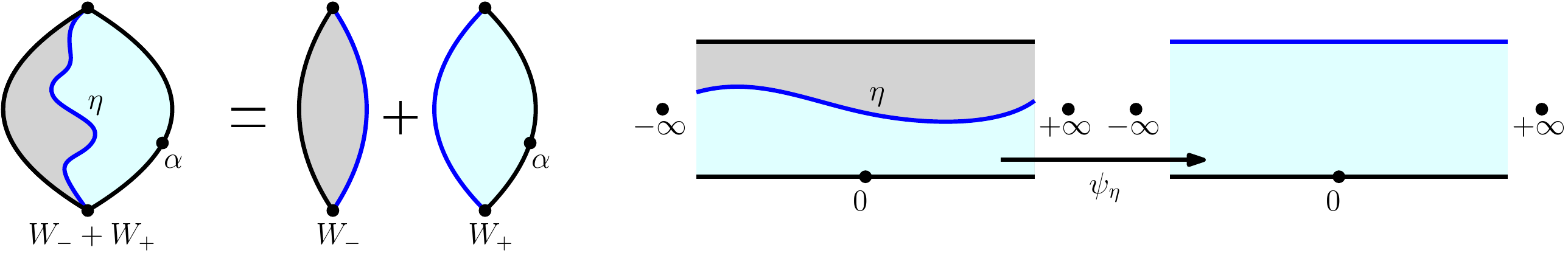}%
	\end{center}
	\caption{\label{fig:conformal-derivative} Proposition~\ref{prop-curve-weight} when $W_+ > \frac{\gamma^2}2$. \textbf{Left: } Illustration of~\eqref{eq-curve-weighted}. \textbf{Right: } Definition of $\psi_\eta$. }
\end{figure}
In the next section we will use Proposition~\ref{prop-curve-weight} to compute $|\sm(W_-, W_+, \alpha)| $, which equals 
$\E[\psi'(0)^{1-\Delta(\alpha)}]$ by definition. 
The key to the proof of Proposition~\ref{prop-curve-weight} is the following lemma based on Girsanov theorem. 

\begin{lemma}\label{lem-change-insertion}
	Let $\alpha_1, \alpha_2, \beta \in \R$ and $\ell > 0$. Then we have the weak convergence of measures
	\eqb \label{eq-change-insertion}
	\lim_{\eps \to 0}
	\eps^{\frac14(\alpha_2^2 - \alpha_1^2)}e^{\frac{(\alpha_2-\alpha_1)}2 \phi_\eps(0)}\LF_{\cS,\ell}^{(\beta, \pm\infty), (\alpha_1, 0)}(d\phi) = \LF_{\cS,\ell}^{(\beta,\pm\infty), (\alpha_2, 0)},
	\eqe
	and moreover $\left|\eps^{\frac14(\alpha_2^2 - \alpha_1^2)}e^{\frac{(\alpha_2-\alpha_1)}2 \phi_\eps(0)}\LF_{\cS,\ell}^{(\beta, \pm\infty), (\alpha_1, 0)}(d\phi)\right| \big/ \left|\LF_{\cS,\ell}^{(\beta,\pm\infty), (\alpha_2, 0)}\right| = 1+ o_\eps(1)$ where  the error $o_\eps(1)$ converge to 0 uniformly in $\ell$. 
\end{lemma}
\begin{proof}
	When $\phi$ is sampled from $(\LF_{\cS, 1}^{(\beta, \pm\infty), (\alpha, 0)})^\#$, the law of $\phi  + \frac2\gamma \log \ell$ is $(\LF_{\cS, \ell}^{(\beta, \pm\infty), (\alpha, 0)})^\#$.  Moreover, by Lemma~\ref{lem-fixed-bdy-length}
		\[
		\frac{\left|e^{\frac{(\alpha_2-\alpha_1)}2 \phi_\eps(0)}\LF_{\cS,\ell}^{(\beta, \pm\infty), (\alpha_1, 0)}(d\phi)\right|}{	\left|e^{\frac{(\alpha_2-\alpha_1)}2 \phi_\eps(0)}\LF_{\cS,1}^{(\beta, \pm\infty), (\alpha_1, 0)}(d\phi)\right|}
		 =  \frac{ \left|\LF_{\cS,\ell}^{(\beta,\pm\infty), (\alpha_2, 0)}\right|}{\left|\LF_{\cS,1}^{(\beta,\pm\infty), (\alpha_2, 0)}\right|} = \ell^{\frac2\gamma(\beta + \frac{\alpha_2}2 - Q)-1.}
		\]
		Therefore, it suffices to prove~\eqref{eq-change-insertion} for $\ell=1$. To this end, for $\eps >0$, let $G_{\cS, \eps}(z,0) := \E[h(z) h_\eps(0)]$. For a distribution $h$, let $\wt h^j := h - (Q-\beta)\left|\Re \cdot\right| + \frac{\alpha_j}2 G_\cS(\cdot, 0)$ for $j = 1,2$, and let $\wt h^{2,\eps} := \wt h^1 + \frac{\alpha_2 -\alpha_1}2 G_{\cS, \eps}$.	
	Let $f$ be a bounded and continuous functional on $H^{-1}(\cS)$ (see Remark~\ref{rmk:ftn-space}).
	Then 
	\begin{align*}
	&\int \eps^{\frac14(\alpha_2^2 - \alpha_1^2)}e^{\frac{(\alpha_2-\alpha_1)}2 (\wt h^1_\eps(0) - \frac2\gamma \log {\nu_{\wt h^1}(\R+ \pi i)})} f(\wt h^1 - \frac2\gamma \log  {\nu_{\wt h^1}(\R+ \pi i)})\\
	&\qquad\qquad\frac2\gamma \nu_{\wt h^1}(\R+ \pi i)^{-\frac2\gamma(\beta+\frac{\alpha_1}2 - Q)} P_\cS(dh) \\ 
	&= \int (1+o_\eps(1)) \E[e^{\frac{\alpha_2-\alpha_1}2h_\eps(0)}]^{-1} e^{\frac{\alpha_2-\alpha_1}2 h_\eps(0)} f(\wt h^1 - \frac2\gamma \log {\nu_{\wt h^1}(\R+ \pi i)}) \\
	&\qquad\qquad\frac2\gamma {\nu_{\wt h^1}(\R+ \pi i)^{-\frac2\gamma(\beta+\frac{\alpha_2}2 - Q)}} P_\cS(dh) 
	\\ &= \int (1+o_\eps(1)) f(\wt h^{2,\eps} - \frac2\gamma \log{\nu_{\wt h^{2,\eps}}(\R+ \pi i)}) 
	\frac2\gamma {\nu_{\wt h^{2,\eps} }(\R+ \pi i)^{-\frac2\gamma(\beta+\frac{\alpha_2}2 - Q)}} P_\cS(dh) \\
	&\stackrel{\eps \to 0}{\longrightarrow} \int f(\wt h^2 - \frac2\gamma \log {\nu_{\wt h^2}(\R+ \pi i)}) \frac2\gamma {\nu_{\wt h^2}(\R+ \pi i)^{-\frac2\gamma(\beta+\frac{\alpha_2}2 - Q)}} P_\cS(dh)\\
	&= \int f (\phi) \LF_{\cS,\ell}^{(\beta,\pm\infty), (\alpha_2, 0)}(d\phi).
	\end{align*}
In the first equality, we are using that the average of $-(Q-\beta)\left|\Re \cdot \right|+ \frac{\alpha_1}2 G_\cS(\cdot, 0)$ on $\partial B_\eps(0) \cap \cS$ is $-\alpha_1 \log \eps + o_\eps(1)$, and $\E[e^{\frac{\alpha_2-\alpha_1}2h_\eps(0)}] = (1+o_\eps(1))\eps^{-\frac14(\alpha_2-\alpha_1)^2}$. The second equality uses Girsanov's theorem, and the final limit uses the dominated convergence theorem and $\nu_{\wt h^2}(\R+ \pi i) = (1+o_\eps(1)) \nu_{\wt h^{2,\eps}}(\R+ \pi i)$ with error $o_\eps(1)$ uniform in $h$ (indeed $\sup_{z \in \R+ \pi i} |G_\cS(z,0) - G_{\cS, \eps}(z,0)| = o_\eps(1)$).  Since $f$ can be arbitrary we obtain~\eqref{eq-change-insertion} for $\ell=1$. 
\end{proof}

\begin{proof}[Proof of Proposition~\ref{prop-curve-weight}]
	We will weight~\eqref{eq-curve-unweighted} to obtain the proposition. We explain first the case $W_+ > \frac{\gamma^2}2$, then the modifications needed for $W_+ < \frac{\gamma^2}2$.
	
	Consider $W_+ > \frac{\gamma^2}2$ and let $\beta_+ = Q+\frac\gamma2-\frac{W_+}\gamma, \beta = Q + \frac\gamma2 - \frac W\gamma$. 
	Sample $(Y, \eta)$ from $\frac\gamma2 (Q-\beta)^{-2}\LF_{\cS, \ell}^{(\beta, \pm\infty), (\gamma, 0)} \times \SLE_\kappa(W_--2;W_+-2)$, so $(\cS, Y, -\infty, +\infty, 0)/{\sim_\gamma}$ has law given by the left hand side of~\eqref{eq-curve-unweighted}. Let $\xi_\eta$ be the map from the connected component of $\cS\backslash \eta$ containing the boundary arc $\R \pi i$ to $\cS$ such that $\xi_\eta$ fixes $\pm\infty$ and $\pi i$. Set
	\eqb\label{eq-Y}
	X = Y\circ \psi_\eta^{-1} + Q \log \left| (\psi_\eta^{-1})' \right|, \quad Z = Y\circ \xi_\eta^{-1} + Q \log \left| (\xi_\eta^{-1})' \right|.
	\eqe
	By Lemma~\ref{lem-curve-unweighted}, the conditional law of $(\cS, Z, \pm\infty)/{\sim_\gamma}$ given $X$ is $\cM_2^\disk(\frac{\gamma^2}2; \ell, \nu_X(\R + \pi i))^\#$, and the marginal law of $X$ is 
	\eqb\label{eq-law-X}
	c_{W_-, W_+}\int_0^\infty \int_0^\infty\ol R(\beta_1; 1,0) x^{-\frac2\gamma W_-}\frac\gamma2(Q-\beta)^{-2} \LF_{\cS,x}^{(\beta_+,-\infty), (\beta_+, +\infty), (\gamma, 0)} \,dx.
	\eqe
	Here, the expression $\frac\gamma2(Q-\beta)^{-2} \LF_\cS^{(\beta_+,-\infty), (\beta_+, +\infty), (\gamma, 0)}$ comes from Proposition~\ref{prop-2-pt-weight-length}, and the prefactor arises from the weighting induced by welding since $|\cMtwo(W_-; \ell)| = \ol R(\beta_1; 1,0) \ell^{-\frac2\gamma W_-}$.

	By Lemma~\ref{lem-change-insertion}, if we weight the law of $(X, Z)$ by $\eps^{\frac14(\alpha^2-\gamma^2)} e^{\frac{\alpha-\gamma}2 X_\eps(0)}$, as $\eps \to 0$ the marginal law of $X$ converges to 
	\eqb\label{eq-law-X-limit}
	c_{W_-, W_+}\int_0^\infty\ol R(\beta_1; 1,0) x^{-\frac2\gamma W_-}\frac\gamma2(Q-\beta)^{-2} \LF_{\cS,x}^{(\beta_+,-\infty), (\beta_+, +\infty), (\alpha, 0)} \,dx,
	\eqe
	and moreover the conditional law of $(\cS, Z, \pm\infty)/{\sim_\gamma}$ given $X$ is still $\cM_2^\disk(\frac{\gamma^2}2; \ell, \nu_X(\R + \pi i))^\#$ in the limit.

	For $\eps\in(0,1)$ let $\theta_\eps$ denote the uniform probability measure on $B_\eps(0)\cap\cS$ such that we have $h_\eps(0)=(h,\theta_\eps)$. Let $\theta^\eta_\eps=(\psi_\eta^{-1})_*\theta_\eps$ denote the pushforward of $\theta$ under $\psi_\eta^{-1}$. By Schwarz reflection we can extend $\psi_\eta^{-1}: \cS \to \cS$ to a holomorphic map $f$ from $\R \times (-\pi, \pi)$ to itself. Since $f'$ is holomorphic, $\log |f'|$ is harmonic and hence $(\log |f'|, \theta_\eps) = \log |f'(0)|$ by the mean value property of harmonic functions. Thus, by~\eqref{eq-Y}
	\eqb
	 X_\eps(0) 
	 = (Y\circ\psi^{-1}+Q\log|(\psi_\eta^{-1})'| , \theta_\eps)
	 = (Y,\theta^\eta_\eps) + Q\log|(\psi_\eta^{-1})'(0)|,
	 \label{eq-XYeps}
	\eqe
	and so	weighting $X$ by $\eps^{\frac14(\alpha^2-\gamma^2)} e^{\frac{\alpha-\gamma}2 X_\eps(0)}$ corresponds to weighting $(Y, \eta)$ by 
	\eqb \label{eq-weighted-Y} \eps^{\frac14(\alpha^2-\gamma^2)}e^{\frac{\alpha-\gamma}2((Y,\theta^\eta_\eps) - Q \log |\psi_\eta'(0)|)}  =  \left(\frac{\eps}{\psi_\eta'(0)} \right)^{\frac14(\alpha^2-\gamma^2)} e^{\frac{\alpha-\gamma}2 (Y,\theta^\eta_\eps)} \cdot |\psi_\eta'(0)|^{\frac14\alpha^2 - \frac Q2\alpha + 1}.
	\eqe
	
	Now we note that for any fixed curve $\eta_0$ in $\ol\cS$ from $-\infty$ to $+\infty$ that does not hit $0$, we have a distortion estimate $|(\psi_{\eta_0}^{-1})'(z) - (\psi_{\eta_0}^{-1})'(0)| / |(\psi_{\eta_0}^{-1})'(0)| = o_\eps(1)$ for $|z| < \eps$, with $o_\eps(1)$ not depending on $\eta_0$. This follows e.g.\ from  \cite[Theorem 3.21]{lawler-book}, which gives the analogous bound for interior points and can be applied to $\psi_{\eta_0}^{-1}$ after extension by Schwarz reflection. Thus, when $h$ is sampled from $P_\cS$ we have  $\E[e^{\frac{\alpha-\gamma}2 (h, \theta^{\eta_0}_\eps)}] = (1+o_\eps(1))\left(\frac\eps{\psi'_{\eta_0}(0)}\right)^{-(\frac{\alpha-\gamma}2)^2}$, where $o_\eps(1)$ does not depend on $\eta_0$. Using this fact and the argument of Lemma~\ref{lem-change-insertion}, we obtain Lemma~\ref{lem-change-insertion} with $\eps$ replaced by $\eps/\psi_{\eta_0}'(0)$ and $\phi_\eps$ replaced by $(\phi, \theta_\eps^{\eta_0})$, where the $o_\eps(1)$ errors do not depend on $\eta_0$. Therefore, for any bounded measurable function $g$ on the space of curves in $\cS$ from $-\infty$ to $+\infty$ equipped with the Hausdorff  topology and any bounded continuous function $F: (H^{-1}(\cS))^3 \to \R$ we have
	\begin{align*}
		&\lim_{\eps\to0}\int \left(\int \left(\frac{\eps}{\psi_\eta'(0)} \right)^{\frac14(\alpha^2-\gamma^2)} e^{\frac{\alpha-\gamma}2 (Y,\theta^\eta_\eps)} F(X,Y,Z) \frac\gamma2 (Q-\beta)^{-2}\LF_{\cS, \ell}^{(\beta, \pm\infty), (\gamma, 0)}(dY) \right)\\
		&\qquad\qquad \qquad \qquad \qquad \qquad \qquad \qquad \qquad \qquad \qquad  g(\eta) \sm(W_-, W_+, \alpha)(d\eta) \\
		&= \int \left(\int F(\wt X,\wt Y,\wt Z) \frac\gamma2 (Q-\beta)^{-2}\LF_{\cS, \ell}^{(\beta, \pm\infty), (\alpha, 0)}(d\wt Y) \right)g(\wt \eta) \sm(W_-, W_+, \alpha)(d \wt \eta)
	\end{align*}
	where $X,Z$ (resp. $\wt X, \wt Z$) are the functions of $(Y, \eta)$ (resp. $(\wt Y, \wt \eta)$) given by~\eqref{eq-Y}. That is, as $\eps \to 0$ the weighted law of $(X,Y, Z, \eta)$ converges to the law of $(\wt X, \wt Y, \wt Z, \wt \eta)$.  
	Thus, when $(\wt Y, \wt \eta)$ is sampled from  $\frac\gamma2 (Q-\beta)^{-2}\LF_{\cS, \ell}^{(\beta, \pm\infty), (\alpha, 0)} \times \sm(W_-, W_+, \alpha)$, the law of $\wt X$ is~\eqref{eq-law-X-limit}, and the conditional law of $(\cS, \wt Z, \pm\infty)$ given $\wt X$ is $\cM_2^\disk(\frac{\gamma^2}2; \ell, \nu_X(\R + \pi i))^\#$.
	This concludes the proof in the case $W_+ > \frac{\gamma^2}2$.
	
	For the case $W_+ < \frac{\gamma^2}2$, the quantum surface to the right of the curve is no longer simply connected. By Lemma~\ref{lem-thin-weighted}, this quantum surface can be described as the concatenation of $(\cD_1, \cD_\bullet, \cD_2)$ sampled from 
	\eqb\label{eq-law-X-thin}
	\ol R(\beta_1; 1,0) L^{-\frac2\gamma W_-}(1 - \frac2{\gamma^2}W_+)^2 \cMtwo(W_+) \times \cMthree(\gamma^2 - W_+) \times \cMtwo(W_+),
	\eqe
	where $L$ is the sum of the left boundary lengths of $(\cD_1, \cD_\bullet, \cD_2)$. Parametrizing $\cD_\bullet$ as $(\cS, X, -\infty, +\infty, 0)$, and arguing exactly as before, we obtain the proposition.
\end{proof}

\section{The shift relations and proof of Theorem~\ref{thm-conformal-derivative0}}\label{sec-shift}
In this section   we use the welding equation from Proposition~\ref{prop-curve-weight}
and the  integrability for quantum disks from LCFT and mating-of-trees  to prove Theorem~\ref{thm-conformal-derivative0} as outlined in Section~\ref{subsec:welding}.
In Section~\ref{subsec-MOT}, we recall the mating-of-trees theorem for the quantum disk  which gives the joint law of the boundary lengths in $\cMtwo(2)$.
Using this theorem we further derive the  analogous result  for $\cMtwo(\frac{\gamma^2}2)$. 
In Section~\ref{subsec-special-cases} we obtain Theorem~\ref{thm-conformal-derivative0} in the cases where $\beta_-$ corresponds to $W_- \in \{ \frac{\gamma^2}2, 2\}$ and $\beta_+$ to generic $W_+$. This is based on  exact results on the length distribution of quantum disks from Section~\ref{subsec-MOT} for the two special weights, and the ones 
from  Section~\ref{sec-rz-lengths} via LCFT for generic weight.
Finally, we derive shift relations in Section~\ref{subsec-proof} and complete the proof of Theorem~\ref{thm-conformal-derivative0}.


\subsection{Integrability of weights $2$ and  $\gamma^2/2$ quantum disk via mating of trees}
\label{subsec-MOT}

Although Lemma~\ref{lem-disk-perim-law} and Proposition~\ref{prop:thick-length-law} and their thin quantum disk counterparts uniquely characterize 
the total mass  of $\cM_2^\disk(W;\ell, r)$ in terms of the reflection coefficient  $R$, it is quite complicated in  general. For $W\in \{2, \gamma^2/2 \}$, 
we have  a much simpler description from \cite{ahs-disk-welding}.
\begin{proposition}[{\cite[Propositions 7.7 and 7.8]{ahs-disk-welding}}]\label{prop-MOT} 
	For $\ell, r > 0$ we have
	\[
	|\cM_2^\disk(2;\ell, r)| =C_1  (\ell + r)^{-\frac4{\gamma^2}-1} \quad \textrm{and}\quad 
	|\cM_2^\disk(\frac{\gamma^2}2;\ell, r)| =C_2   \frac{(\ell r)^{4/\gamma^2 - 1}}{(\ell^{4/\gamma^2}+r^{4/\gamma^2})^2}.
	\]
\end{proposition}
We do not need the values of $C_1$ and $C_2$, but they are easy to derive using results from  Section~\ref{sec-rz-lengths}.
\begin{proposition}\label{prop-left-right-law-2}
	For $\ell, r > 0$ we have
	\[|\cM_2^\disk(2;\ell, r)| = \frac{(2\pi)^{\frac4{\gamma^2} -1}}{(1-\frac{\gamma^2}4)\Gamma(1-\frac{\gamma^2}4)^{\frac4{\gamma^2}}} (\ell + r)^{-\frac4{\gamma^2}-1},
	\]
	\[ |\cM_2^\disk(\frac{\gamma^2}2;\ell, r)| = \frac4{\gamma^2} \frac{(\ell r)^{4/\gamma^2 - 1}}{(\ell^{4/\gamma^2}+r^{4/\gamma^2})^2}.
		\]
\end{proposition}
\begin{proof}
	 By Lemma~\ref{lem-disk-perim-law} with $W=2$ , the law of  the quantum length of the left boundary arc of $\cMtwo(2)$ is $\1_{\ell > 0}  \ol R(\gamma, 1,0)  \ell^{-\frac4{\gamma^2}}$ 	with $\ol R(\gamma, 1,0)$ as in~\eqref{eq:R0}. By Proposition~\ref{prop-MOT}, we have
	 \eqbn
	 \begin{split}
	 	\ol R(\gamma, 1,0)=  \int_0^\infty C_1  (1 + r)^{-\frac4{\gamma^2}-1}\, dr
	 	 = \frac{C_1 \gamma^2}4, 
	 \end{split}
	 \eqen
	 and applying the shift equations~\eqref{eq-shift} to~\eqref{eq:R0}, we have
	 \eqbn
	 \begin{split}
	 	\ol R(\gamma, 1,0)&= \frac{(2\pi)^{\frac4{\gamma^2} - \frac32} (\frac2\gamma)^{-\frac12 - \frac{\gamma^2}4}}{(1-\frac{\gamma^2}4)\Gamma(1-\frac{\gamma^2}4)^{\frac4{\gamma^2}-1}} \frac{\Gg(\frac\gamma2)}{\Gg(Q)} \frac{\Gg(Q)}{\Gg(\frac2\gamma)} \frac{\Gg(\frac2\gamma)}{\Gg(\frac2\gamma-\frac\gamma2)} \\
	 	&= \frac{\gamma^2}4 \frac{(2\pi)^{\frac4{\gamma^2} -1}}{(1-\frac{\gamma^2}4)\Gamma(1-\frac{\gamma^2}4)^{\frac4{\gamma^2}}} .
	 	\end{split}
	 	\eqen
This gives $C_1$. 
	Similarly, by Lemma~\ref{lem-disk-perim-law} with $W=\frac{\gamma^2}{2}$,
	the law of the quantum length of the left boundary arc of $\cMtwo(\frac{\gamma^2}2)$ is $\1_{\ell>0}\ell^{-1} d\ell$. By Proposition~\ref{prop-MOT}  and using the  change of variables $t = r^{4/\gamma^2}$, we have
	\[
	1 = \int_0^\infty C_2 \frac{r^{4/\gamma^2 - 1}}{(1+r^{4/\gamma^2})^2}\, dr = \frac{C_2\gamma^2}4 \int_0^\infty \frac{1}{(1+t)^2} =  \frac{C_2\gamma^2}4. \qedhere
	\]
\end{proof}

\subsection{Special cases of Theorem~\ref{thm-conformal-derivative0} }\label{subsec-special-cases}
In this section, we leverage exact formulas for $|\cMtwo(2; \ell, r)|$ and $|\cMtwo(\frac{\gamma^2}2; \ell, r)|$ to show Proposition~\ref{prop-shift-1}, which is Theorem~\ref{thm-conformal-derivative0} in the cases where $\kappa \in (0, 4)$ and $\rho_- \in \{ 0, \frac\kappa2\}$. 

We will use the parameters  from LQG and conformal welding  to express the moment in Theorem~\ref{thm-conformal-derivative0}. 
	More precisely,  for $\gamma\in (0,2), \lambda\in \R$ and $\beta_-, \beta_+ < Q + \frac\gamma2$ we set 
	\eqb\label{eq-m-lambda}
	\mm(\beta_-, \beta_+) := \E[\psi'(1)^\lambda]
	\eqe where $\E[\psi'(1)^\lambda]$ is the moment in Theorem~\ref{thm-conformal-derivative0} with 
	\begin{equation}\label{eq:parameter}
		\kappa = \gamma^2\in (0,4),  \quad \rho_-= \gamma^2 - \gamma \beta_->-2, \quad \textrm{and} \quad \rho_+ = \gamma^2 - \gamma \beta_+>-2.
	\end{equation}
	We first make some  basic observations on  $\mm(\beta_-, \beta_+) $.
	\begin{lemma}\label{lem:mono}
		If $\mm(\beta_-, \beta_+)<\infty$ and $\lambda < \lambda'$	
		then \(\mm(\beta_-, \beta_+) < m^{\lambda'}_\gamma(\beta_-, \beta_+)\). Moreover, $m^0_\gamma(\beta_-, \beta_+) = 1$.
	\end{lemma}
	\begin{proof}
		From the definition of $\psi$, we  see that $\psi'(1) > 1$ a.s., giving the  monotonicity property. The second observation is trivial. 
	\end{proof}
	Since we can conformally map from $(\bbH,0, \infty, 1)$ to $(\cS,-\infty, +\infty, 0)$,
	from the definition of  $\sm(W_-, W_+, \alpha)$ above Proposition~\ref{prop-curve-weight}, we see that  
	\begin{equation}\label{eq:moment-weld}
		|\sm(W_-, W_+, \alpha)| = \mm(\beta_-, \beta_+) \quad \textrm{with }
		W_\pm = \gamma(\gamma + \frac2\gamma - \beta_{\pm})
	\end{equation}
	Now we can compute $\mm(\beta_-, \beta_+)$ by computing $|\sm(W_-, W_+, \alpha)|$ via Proposition~\ref{prop-curve-weight}. Based on this idea,
the following lemma computes $\mm(\gamma, \beta_+)$ for  $\beta_+ \neq Q$ and a certain range of  $\lambda$, modulo a $\beta_+$-dependent multiplicative constant.
The range of $\lambda$ below does not contain $0$. Later we will remove this restriction so that the constant can be recovered from $m^0_\gamma(\beta_-, \beta_+) = 1$.
\begin{lemma}\label{lem-shift-1} 
	For any $\beta_+ \in (-\infty, Q) \cup (Q, Q + \frac\gamma2)$ and $\alpha \in (2|\beta_+ - Q|, 4Q - 2\beta_+)$, set $\lambda = 1 - \frac\alpha2 (Q-\frac\alpha2)$. 
	Then there is a constant $C = C_\gamma(\beta_+) \in (0, \infty)$ not depending on $\alpha$ such that 
	\eqb\label{eq-lem-shift-1}
	\mm(\gamma, \beta_+) =C_\gamma(\beta_+) \, \Gamma(\frac2\gamma(Q-\beta_++\frac12\alpha))\Gamma(\frac2\gamma(2Q-\beta_+ - \frac12\alpha)).
	\eqe
\end{lemma}
\begin{proof}
	Set $W_- = 2$ and $W_+ = \gamma(\gamma+ \frac2\gamma - \beta_+)$ and consider Proposition~\ref{prop-curve-weight} with these parameters. 
	Set $\beta = \beta_+ 	- \frac2\gamma$ so  that $W_-+W_+= \gamma(\gamma+ \frac2\gamma - \beta)$.
	By Proposition~\ref{prop-H-density}  and~\eqref{eq:moment-weld},  since $\alpha > 0 \vee 2(Q-\beta_+)$,
	the unmarked boundary arc's length of a sample from $\cMthree(W_- + W_+;\alpha; \ell)\otimes \sm(W_-, W_+, \alpha) $
	has the power  law distribution $\1_{x > 0} \mathfrak C x^{\frac2\gamma(\beta + \frac12\alpha - Q)-1} dx$ where
	\begin{equation}\label{eq:C1}
	\mathfrak C =(Q-\beta)^{-2} \ol H^{(\beta, \beta, \alpha)}_{(0,1,0)}   |\sm(W_-, W_+, \alpha)|=   (Q-\beta)^{-2} \ol H^{(\beta, \beta, \alpha)}_{(0,1,0)}\mm(\gamma, \beta_+).
	\end{equation}		
	
	We now evaluate $\mathfrak C$ via the right hand side of~\eqref{eq-curve-weighted}.
By Proposition~\ref{prop-H-density} if $\beta_+< Q$ or Proposition~\ref{prop-disk-03-length} if $\beta_+ \in (Q, Q+\frac\gamma2)$, the right hand side of~\eqref{eq-curve-weighted} gives, with $\wt C = {\wt C}_\gamma(\beta_+)$ a constant not depending on $\alpha$,
	\begin{align} 
		\mathfrak C &=  c_{W_-, W_+} \int_0^\infty |\cMtwo(2; 1, \ell)| \cdot (Q-\beta_+)^{-2} \ol H^{(\beta_+, \beta_+, \alpha)}_{(0,1,0)} \ell^{\frac2\gamma(\beta_+ + \frac12\alpha - Q)-1} \, d\ell \nonumber \\
		&= {\wt C}_\gamma(\beta_+) \,  \ol H^{(\beta_+, \beta_+, \alpha)}_{(0,1,0)} \Gamma(\frac2\gamma(\beta_+ +\frac12\alpha - Q)) \Gamma(\frac2\gamma(2Q - \beta_+ - \frac12\alpha)). \label{eq:C2}
	\end{align}
	The second equality follows from $|\cMtwo(2; 1, \ell)| \propto (1+\ell)^{-\frac4{\gamma^2}-1}$
	(Proposition~\ref{prop-MOT}) and the beta function integral $\frac{\Gamma(x)\Gamma(y)}{\Gamma(x+y)} = \int_0^\infty \frac{\ell^{x-1}}{(1+\ell)^{x+y}}\,d\ell$ for $x = \frac2\gamma(\beta_+ + \frac12\alpha -Q) > 0$ and $y = \frac2\gamma(2Q - \beta_+ - \frac12\alpha)>0$. 
	Note that the hypotheses of Proposition~\ref{prop-H-density} (if $\beta_+< Q$) or Proposition~\ref{prop-disk-03-length} (if $\beta_+\in (Q, Q +\frac\gamma2)$) and the inequalities $x,y>0$ all hold because of our conditions on $\alpha, \beta_+$. 
	
	Comparing~\eqref{eq:C1} and~\eqref{eq:C2}, we get 
		\begin{align*}
			\mm(\gamma, \beta_+)= {\wt C}_\gamma(\beta_+) \,   (Q-\beta)^{2}  \Gamma(\frac2\gamma(\beta_+ +\frac12\alpha - Q)) \Gamma(\frac2\gamma(2Q - \beta_+ - \frac12\alpha)) \frac{\ol H^{(\beta_+, \beta_+, \alpha)}_{(0,1,0)}}{\ol H^{(\beta, \beta, \alpha)}_{(0,1,0)}}.
		\end{align*}
			Using the shift relation~\eqref{eq-shift} for $\Gg$ and $\beta_+ = \beta + \frac2\gamma$, we can express $\ol H^{(\beta_+, \beta_+, \alpha)}_{(0,1,0)} \big/
			\ol H^{(\beta, \beta, \alpha)}_{(0,1,0)}$ as 
			\eqbn
			\begin{split}
				&\left( \frac{2\pi}{(\frac\gamma2)^{\frac{\gamma^2}4} \Gamma(1-\frac{\gamma^2}4)} \right)^{\frac4{\gamma^2}}
				\frac{\Gg(Q-\beta)^2}{\Gg(Q - \beta_+)^2}
				\frac{\Gg(Q-\beta_+ +\frac12\alpha)}{\Gg(Q-\beta+\frac12\alpha)} 
				\frac{\Gg(\beta_+ + \frac12\alpha - \frac\gamma2)}{\Gg(\beta + \frac12\alpha-\frac\gamma2)}\\ 
				&\qquad\qquad\qquad\qquad\qquad\qquad={\wh C}_\gamma(\beta_+) \frac{\Gamma(\frac2\gamma(Q-\beta_+ +\frac12\alpha))}{\Gamma(\frac2\gamma(\beta_+ + \frac12\alpha-Q))}
				\\
				&\text{where }
				{\wh C}_\gamma(\beta_+) = \left( \frac{2\pi}{(\frac\gamma2)^{\frac{\gamma^2}4} \Gamma(1-\frac{\gamma^2}4)} \right)^{\frac4{\gamma^2}}
				\frac{\Gg(Q-\beta_+ + \frac2\gamma)^2}{\Gg(Q - \beta_+)^2}  \left(\frac2\gamma\right)^{\frac4\gamma (\beta_+ - Q)}.
			\end{split}
			\eqen
		Setting $C_\gamma(\beta)={\wt C}_\gamma(\beta_+) {\wh C}_\gamma(\beta_+) (Q-\beta)^{2} $ we  conclude the proof.
\end{proof}

The following lemma is the counterpart of Lemma \ref{lem-shift-1} with $\beta_-=Q$ instead of $\beta_-=\gamma$. The proof follows the exact same steps as that of Lemma \ref{lem-shift-1}, with $\cMtwo(\frac{\gamma^2}2)$  in place of $\cMtwo(2)$.
\begin{lemma}\label{lem-shift-2}
	Let $\beta_+ \in (-\infty, Q) \cup (Q, Q + \frac\gamma2)$. Let $\alpha \in (2|\beta_+ - Q|, 4Q - 2\beta_+)$ and $\lambda = 1 - \frac\alpha2 (Q-\frac\alpha2)$. 
	Then there is a constant $C = C_\gamma(\beta_+) \in (0, \infty)$ not depending on $\alpha$ such that 
	\[\mm(Q,\beta_+) = C_\gamma(\beta_+) \, \Gamma(\frac\gamma2(Q - \beta_+ + \frac12\alpha)) \Gamma(\frac\gamma2(2Q - \beta_+ - \frac12\alpha)). \]
\end{lemma}
\begin{proof}
	This proof is essentially the same as Lemma~\ref{lem-shift-1}. We consider the density of the unmarked boundary arc length of a sample from 
		$\cMthree(W_- + W_+;\alpha; \ell)\otimes \sm(W_-, W_+, \alpha)$, which is given by 	$\1_{x > 0} \mathfrak C x^{\frac2\gamma(\beta + \frac12\alpha - Q) - 1} dx$ 
		with
		\[\mathfrak C = (Q-\beta)^{-2} \ol H^{(\beta, \beta, \alpha)}_{(0,1,0)}\mm(Q, \beta_+),\quad \text{where } \beta = \beta_+ 	- \frac\gamma2.\]
		
		We now use  Propositions~\ref{prop-curve-weight}  and~\ref{prop-MOT} to compute the 
		$\mathfrak C$ from the right side  of~\eqref{eq-curve-weighted}. By Proposition~\ref{prop-H-density} if $\beta_+< Q$ or Proposition~\ref{prop-disk-03-length} if $\beta_+ \in (Q, Q+\frac\gamma2)$, the right hand side of~\eqref{eq-curve-weighted} gives, with $\wt C =  {\wt C}_\gamma(\beta_+)$ a constant not depending on $\alpha$,
	\alb
	\mathfrak C &=  c_{W_-, W_+} \int_0^\infty |\cMtwo(\frac{\gamma^2}2; 1, \ell)| \cdot (Q-\beta_+)^{-2} \ol H^{(\beta_+, \beta_+, \alpha)}_{(0,1,0)} \ell^{\frac2\gamma(\beta_+ + \frac12\alpha - Q)-1} \, d\ell\\
	&= {\wt C}_\gamma(\beta_+) \,\ol H^{(\beta_+, \beta_+, \alpha)}_{(0,1,0)} \Gamma(\frac\gamma2(\beta_+ + \frac12\alpha-\gamma))\Gamma(\frac\gamma2(2Q - \beta_+ - \frac12\alpha)),
	\ale
	where the second equality follows from $|\cMtwo(\frac{\gamma^2}2; 1, \ell)| \propto \frac{\ell^{\frac4{\gamma^2}-1}}{(1+\ell^{\frac4{\gamma^2}})^2}$ (Proposition~\ref{prop-MOT}) and the beta function integral $\frac{\Gamma(x)\Gamma(2-x)}{\Gamma(2)} = \int_0^\infty \frac{t^{x-1}}{(1+t)^{2}}\,dt = \frac4{\gamma^2}\int_0^\infty \frac{\ell^{\frac4{\gamma^2} x-1}}{(1+\ell^{\frac4{\gamma^2}})^2} \, d\ell$ with the change of variables $t = \ell^{4/\gamma^2}$ and with $x = \frac\gamma2(\beta_+ + \frac12\alpha -\gamma) \in (0,2)$. Note that the hypotheses of Proposition~\ref{prop-H-density} (if $\beta_+< Q$) or Proposition~\ref{prop-disk-03-length} (if $\beta_+\in (Q, Q +\frac\gamma2)$) and the bound $x \in (0,2)$ all hold because of our conditions on $\alpha, \beta_+$. 
	
	The rest of the argument is identical to the proof of Lemma~\ref{lem-shift-1} except this time 
			\eqbn
			\begin{split}
			\frac{\ol H^{(\beta_+, \beta_+, \alpha)}_{(0,1,0)}}{
				\ol H^{(\beta, \beta, \alpha)}_{(0,1,0)} }
			=&\,  \frac{2\pi}{(\frac\gamma2)^{\frac{\gamma^2}4} \Gamma(1-\frac{\gamma^2}4)} 
			\frac{\Gg(Q-\beta_+ + \frac\gamma2)^2}{\Gg(Q - \beta_+)^2}  \left(\frac\gamma2\right)^{\frac\gamma2 (2\beta_+ -\gamma - Q)}\\
			&\qquad\cdot\frac{\Gamma(\frac\gamma2(Q-\beta_+ + \frac12\alpha))}{\Gamma(\frac\gamma2(\beta_+ + \frac12\alpha - \gamma))}	
			\end{split}
			\eqen
		since $ \beta_+ -\beta =\frac\gamma2$ instead of $\frac2\gamma$.
		We omit the rest of the details. \qedhere
\end{proof}

The following is equivalent to Theorem~\ref{thm-conformal-derivative0} for $\kappa <4$ and $\rho_- \in \{ 0, \frac\kappa2 -2\}$.
\begin{proposition}\label{prop-shift-1}
	Let $\beta_+< Q + \frac\gamma2$, and let  $\lambda_0 = \frac1\kappa(\rho_+ + 2)(\rho_+ + 4 - \frac\kappa2)$ where $\kappa = \gamma^2$ and $\rho_+ = \gamma^2 - \gamma \beta_+$. For $\lambda < \lambda_0$, let $\alpha$ be either solution to $1 - \frac\alpha2 (Q-\frac\alpha2) = \lambda$. Then
	\alb
	\mm(\gamma, \beta_+) &=
	\frac{\Gamma(\frac2\gamma(Q-\beta_++\frac12\alpha))\Gamma(\frac2\gamma(2Q-\beta_+ - \frac12\alpha))}{\Gamma(\frac2\gamma(Q-\beta_++\frac\gamma2))\Gamma(\frac2\gamma(Q-\beta_+ + \frac2\gamma))},\\
	\mm(Q,\beta_+)&=
	\frac{\Gamma(\frac\gamma2(Q-\beta_++\frac12\alpha))\Gamma(\frac\gamma2(2Q-\beta_+ - \frac12\alpha))}{\Gamma(\frac\gamma2(Q-\beta_++\frac\gamma2))\Gamma(\frac\gamma2(Q-\beta_+ + \frac2\gamma))}.
	\ale
\end{proposition}
\begin{proof}
	We prove the $\mm(\gamma, \beta_+)$ identity using Lemma~\ref{lem-shift-1}; the proof of the $\mm(Q, \beta_+)$ identity using Lemma~\ref{lem-shift-2} is the same. 
	
	First fix any $\beta_+ \neq Q$. 
	By Proposition~\ref{prop-finite-moment} $\mm(\gamma, \beta_+) < \infty$ for all $\lambda < \lambda_0$, so we can apply Morera's theorem and Fubini's theorem to see that  $\lambda \mapsto \mm(\gamma, \beta_+)$ is holomorphic on $\{ \lambda \in \C \: : \: \Re \lambda < \lambda_0\}$. By Lemma~\ref{lem-shift-1}, this function agrees with the right hand side of~\eqref{eq-lem-shift-1} for some interval in $\R$, so by the uniqueness of holomorphic extensions~\eqref{eq-lem-shift-1} is true for all $\lambda \in \R$ with $\lambda < \lambda_0$. Setting $\lambda = 0$ and $\alpha = \gamma$, we deduce that $C(\beta_+)$ in Lemma~\ref{lem-shift-1}
	equals $\Gamma(\frac2\gamma(Q-\beta_++\frac\gamma2))^{-1}\Gamma(\frac2\gamma(Q-\beta_+ +\frac2\gamma))^{-1}$, completing the proof for $\beta_+ \neq Q$.
	
	When $\beta_+ = Q$ and $\lambda < 0$ we obtain the result from the $\beta_+ \neq Q$ case by taking the approximating sequence $(\kappa^n, \rho_-^n, \rho_+^n) = (\kappa, \gamma^2 - \gamma \beta_-, \gamma^2 - \gamma \beta_+^n)$ with $\beta_+ := Q - \frac1n$ in Lemma~\ref{lem-couple-map}. Then the same holomorphic extension argument as above allows us to address all $\lambda < \lambda_0$. 
\end{proof}
\begin{remark}
	The expressions in 
	Proposition~\ref{prop-shift-1} can be  written as hypergeometric functions:
	\alb
	m_\gamma^\lambda (\gamma, \beta_+) &= {_2}F_1\left(\frac2\gamma( \frac\alpha2 - \frac\gamma2), \frac2\gamma(\frac\alpha2 - \frac2\gamma), \frac2\gamma(Q-\beta_++ \frac\alpha2);1\right),\\ m_\gamma^\lambda (Q, \beta_+) &= {_2}F_1\left( \frac\gamma2(\frac\alpha2 - \frac\gamma2), \frac\gamma2(\frac\alpha2 - \frac2\gamma), \frac\gamma2(Q-\beta_++ \frac\alpha2);1\right).
	\ale
\end{remark}

\subsection{Proof of Theorem~\ref{thm-conformal-derivative0} via shift equations}\label{subsec-proof}
In this section we complete the proof of Theorem~\ref{thm-conformal-derivative0}. We first state a composition relation for $\mm$, then derive shift relations, and finally show that these relations determine $\mm$.

\begin{lemma}[Composition relation]\label{lem-composition}
	For $\beta, \beta_-, \beta_+ < Q + \frac\gamma2$ and $\lambda<0$, we have 
	\[\mm(\beta + \beta_- - Q - \frac\gamma2, \beta_+) = \mm(\beta, \beta_- + \beta_+ - Q - \frac\gamma2) \mm(\beta_-, \beta_+). \]
\end{lemma}
\begin{proof}
	Let $\rho = \gamma^2 - \gamma \beta$, $\rho_\pm = \gamma^2 - \gamma \beta_\pm$.
	Independently sample an $\SLE_\kappa(\rho; \rho_- + \rho_+ + 2)$ curve $\eta_1$  and an $\SLE_\kappa(\rho_-; \rho_+)$ curve $\eta_2$ in $\bbH$ from $0$ to $\infty$, let $D_j$ be the connected component of $\bbH \backslash \eta_j$ containing $1$ on its boundary for $j=1,2$, and let $\psi_j:D_j \to \bbH$ be the conformal map such that $\psi_j(1) = 1$ and the first (resp. last) point on $\partial D_j$ traced by $\eta_j$ is mapped to $0$ (resp. $\infty$). Let $\eta := \psi_1^{-1}(\eta_2)$ and $\psi := \psi_2 \circ \psi_1$. 
	The theory of imaginary geometry \cite[Proposition 7.4]{ig1}  tells us that the law of $\eta$ is $\SLE_\kappa(\rho + \rho_- + 2; \rho_+)$. Thus, since $\psi'(1) = \psi_1'(1)\psi_2'(1)$ and $\psi_1, \psi_2$ are independent, 
	\eqbn
	\begin{split}
		\mm(\beta + \beta_- - Q - \frac\gamma2, \beta_+) &= \E[\psi'(1)^\lambda] = \E[\psi_1'(1)^\lambda]\E[\psi_2'(1)^\lambda]\\
		&=  \mm(\beta, \beta_- + \beta_+ - Q - \frac\gamma2) \mm(\beta_-, \beta_+).
	\end{split} 
	\eqen
	Here the assumption $\lambda<0$ ensures the finiteness of the two sides.
\end{proof}

We immediately deduce the following shift relations.
\begin{lemma}[Shift relations for $\mm$]\label{lem-shift-m}
	For $\beta_-, \beta_+ < Q +\frac\gamma2$,  $\lambda < 0$ and  $\alpha$ a solution to $1 - \frac\alpha2 (Q-\frac\alpha2) = \lambda$,
	\alb
	\frac{\mm(\beta_- - \frac2\gamma, \beta_+)}{\mm(\beta_-, \beta_+)} &= 
	\frac{\Gamma(\frac2\gamma(2Q + \frac\gamma2 -\beta_--\beta_+ + \frac12\alpha))\Gamma(\frac2\gamma(3Q + \frac\gamma2 -\beta_- - \beta_+ - \frac12\alpha))}{\Gamma(\frac2\gamma(2Q + \gamma -\beta_- - \beta_+ ))\Gamma(\frac2\gamma(3Q -\beta_- - \beta_+))},\\
	\frac{\mm(\beta_- - \frac\gamma2, \beta_+)}{\mm(\beta_-, \beta_+)} &= 
	\frac{\Gamma(\frac\gamma2(2Q+\frac\gamma2-\beta_--\beta_++\frac12\alpha))\Gamma(\frac\gamma2(3Q + \frac\gamma2-\beta_- -\beta_+  - \frac12\alpha))}{\Gamma(\frac\gamma2(2Q + \gamma -\beta_- - \beta_+ ))\Gamma(\frac\gamma2(3Q -\beta_- - \beta_+))}.
	\ale
\end{lemma}
\begin{proof}
	For the first identity, set $\beta = \gamma$ in Lemma~\ref{lem-composition}, then use Proposition~\ref{prop-shift-1} to eliminate the term $\mm (\gamma, \beta_- + \beta_+ - Q - \frac\gamma2)$. 
	For the second identity, set $\beta = Q$ in Lemma~\ref{lem-composition}, then use Proposition~\ref{prop-shift-1} to eliminate the term $\mm (Q, \beta_- + \beta_+ - Q - \frac\gamma2)$. 
\end{proof}

We now use the shift relations to prove Theorem~\ref{thm-conformal-derivative0} in some regime.
\begin{proposition}\label{prop-SLE-moment-main}
	Theorem~\ref{thm-conformal-derivative0} holds when $\kappa \in (0,4) \backslash \Q$ and $\lambda < 0$. Namely,
	using the identification of parameters from~\eqref{eq:parameter},  for $\gamma^2  = \kappa \in (0, 4) \backslash \Q$, $\beta_-, \beta_+ < Q + \frac\gamma2$, $\lambda < 0$ and $\alpha$ a solution to $1 - \frac\alpha2 (Q-\frac\alpha2) = \lambda$, we have
	\eqb \label{eq-lem-SLE-moment-main}
	\begin{split}
		\mm(\beta_-, \beta_+) 
		=  &\,\frac{\Gg(2Q + \gamma - \beta_- - \beta_+) \Gg(3Q - \beta_- - \beta_+)}{\Gg(2Q + \frac\gamma2 - \beta_- - \beta_+ + \frac12\alpha)\Gg(3Q + \frac\gamma2 - \beta_- - \beta_+ - \frac12\alpha)} \\
		&\qquad\cdot\frac{\Gg(Q-\beta_+ + \frac12\alpha) \Gg(2Q - \beta_+ - \frac12 \alpha)}{ \Gg(Q - \beta_+ +\frac\gamma2) \Gg(Q - \beta_+ +\frac2\gamma)}. 
	\end{split}
	\eqe
\end{proposition}
\begin{proof}
	
	We first show that there is a function $c(\beta_+,\alpha)$ such that 
	\eqb \label{eq-lem-moment-mod-const}
	\begin{split}
		&\mm(\beta_-, \beta_+) \\
		&\qquad= c(\beta_+,\alpha) \frac{\Gg(2Q + \gamma - \beta_- - \beta_+) \Gg(3Q - \beta_- - \beta_+)}{\Gg(2Q + \frac\gamma2 - \beta_- - \beta_+ + \frac12\alpha)\Gg(3Q + \frac\gamma2 - \beta_- - \beta_+ - \frac12\alpha)}. 
	\end{split}
	\eqe
	Let $\wtmm(\beta_-, \beta_+)$ denote the right hand side of~\eqref{eq-lem-moment-mod-const} divided by $c(\beta_+,\alpha)$. Using the shift relations of $\Gg$~\eqref{eq-shift} it is easy to check that the equations of Lemma~\ref{lem-shift-m} still hold when $\mm$ is replaced by $\wtmm$. Consequently, for all $\beta_+, \beta_- < Q + \frac\gamma2$ we have
	\[\frac{\mm(\beta_- - \frac2\gamma, \beta_+)}{\wtmm(\beta_- - \frac2\gamma, \beta_+)} = \frac{\mm(\beta_-, \beta_+)}{\wtmm(\beta_- , \beta_+)} = \frac{\mm(\beta_- - \frac\gamma2, \beta_+)}{\wtmm(\beta_- - \frac\gamma2, \beta_+)}.\]
	Keep $\beta_+$ fixed. Since $\gamma^2 \not \in \Q$, starting from $\beta_- = 0$ and making upward jumps $\beta_- \mapsto \beta_- + \frac2\gamma$ and downward jumps $\beta_- \mapsto \beta_- - \frac\gamma2$, we conclude that for $\beta_-$ in a dense subset of $(-\infty, Q + \frac\gamma2)$ we have $\mm(\beta_-, \beta_+)/\wtmm(\beta_-, \beta_+) = \mm(0, \beta_+)/\wtmm(0, \beta_+) =: c(\beta_+, \alpha)$, i.e.\ \eqref{eq-lem-moment-mod-const} holds for a dense set of $\beta_- \in (-\infty, Q + \frac\gamma2)$.

	Since $\lambda < 0$, by Lemmas~\ref{lem:mono} and~\ref{lem-composition} we have $\mm(\beta_- - (Q+\frac\gamma2 - \beta), \beta_+) = \mm(\beta, \beta_-+\beta_+ - Q-\frac\gamma2) \mm(\beta_-, \beta_+) < \mm(\beta_-, \beta_+)$ for all $\beta < Q + \frac\gamma2$, so $\mm(\beta_-, \beta_+)$ is monotone in $\beta_-$. The right hand side of~\eqref{eq-lem-moment-mod-const} is continuous in $\beta_-$, so by monotonicity we can extend~\eqref{eq-lem-moment-mod-const} from a dense set to the full range $\beta_- \in (-\infty, Q + \frac\gamma2)$. Thus we have shown~\eqref{eq-lem-moment-mod-const}.

	Now, both Proposition~\ref{prop-shift-1} and~\eqref{eq-lem-moment-mod-const}  give expressions for  $\mm(Q, \beta_+)$, in the latter case, in terms of $c(\beta_+,\alpha)$. Comparing these yields 
	\eqbn
	\begin{split}
		&c(\beta_+,\alpha) \frac{\Gg(Q + \gamma - \beta_+)\Gg(2Q - \beta_+)}{\Gg(Q + \frac\gamma2 - \beta_+ + \frac12\alpha)\Gg(2Q + \frac\gamma2 -\beta_+ -\frac12\alpha)}\\
		&\qquad\qquad\qquad\qquad\qquad= \frac{\Gamma(\frac\gamma2(Q - \beta_+ + \frac12\alpha)) \Gamma(\frac\gamma2(2Q - \beta_+ - \frac12\alpha))}{\Gamma(\frac\gamma2(Q  - \beta_+ + \frac\gamma2 )) \Gamma(\frac\gamma2(Q   - \beta_+ + \frac2\gamma))}.
	\end{split}
	 \eqen
	We may simplify this using the shift relations for $\Gg$ to get 
	\[c(\beta_+,\alpha) = \frac{\Gg(Q-\beta_+ + \frac12\alpha) \Gg(2Q - \beta_+ - \frac12 \alpha)}{ \Gg(Q - \beta_+ +\frac\gamma2) \Gg(Q - \beta_+ +\frac2\gamma)}, \]
	and eliminating $c(\beta_+,\alpha)$ from~\eqref{eq-lem-moment-mod-const} gives~\eqref{eq-lem-SLE-moment-main}.
\end{proof}
Given Proposition~\ref{prop-SLE-moment-main}, we extend Theorem~\ref{thm-conformal-derivative0} to all rational $\kappa \in (0,4]$ by continuity, to all $\lambda<\lambda_0$ by holomorphicity, and to all $\kappa >4$ by SLE duality.
\begin{lemma}\label{lem-thm-small}
	Theorem~\ref{thm-conformal-derivative0} holds for $\kappa \in (0, 4]$ and $\rho_-, \rho_+ > -2$. 
\end{lemma}
\begin{proof}
	We first prove the result for $\lambda < 0$. Extend the definition of $\mm$ in~\eqref{eq-m-lambda} to $\gamma = 2$; this is the only place in the paper where we consider $\gamma = 2$ (corresponding to $\kappa = 4$ and $Q = 2$).
	As before, for each $\gamma \in (0,2]$ it suffices to prove~\eqref{eq-lem-SLE-moment-main} for all $\beta_-,\beta_+< Q + \frac\gamma2$. 
	
	We first show that
	\eqb\label{eq-thm1-special-cases}
	\text{\eqref{eq-lem-SLE-moment-main} holds for }\kappa \in (0, 4], \lambda < 0,\{\beta_-, \beta_+ \leq Q \} \cup \{\beta_- = \gamma, \beta_+ < Q + \frac\gamma2 \}.
	\eqe
	For $\kappa = \gamma^2 \in (0, 4]$, $\lambda < 0$, and $\beta_\pm \leq Q$,~\eqref{eq-lem-SLE-moment-main}   follows from  Proposition~\ref{prop-SLE-moment-main} and Lemma~\ref{lem-couple-map} applied to the sequence $(\kappa^n, \rho_-^n, \rho_+^n) = (\kappa^n, \rho_-, \rho_+)$, where $(\kappa^n)_{n \geq 1}$ is an increasing sequence of irrational numbers with limit $\kappa$, and $\rho_\pm^n=\rho_\pm = \gamma^2 - \gamma \beta_\pm$. 
	Likewise, when $\kappa \in (0, 4]$, $\lambda < 0$, and $\rho_- = 0, \rho_+ \in (0, \frac\kappa2-2)$,~\eqref{eq-lem-SLE-moment-main} follows from Proposition~\ref{prop-SLE-moment-main} and Lemma~\ref{lem-couple-map-nonsimple}. Thus we have verified~\eqref{eq-thm1-special-cases}.

	Now consider any $\beta_-, \beta_+ < Q + \frac\gamma2$ and $\lambda < 0$. Using Lemma~\ref{lem-composition} yields $\mm( \gamma, \beta_- + \beta_+ - Q - \frac\gamma2) \mm (\beta_-, \beta_+) = \mm (\beta_- - \frac2\gamma, \beta_+)$ and, for any sufficiently negative $\beta \ll 0$,  
	\eqbn
	\begin{split}
		\mm(\beta, \beta_+ - \frac2\gamma)& \mm(\gamma, \beta_+) = \mm (\beta -\frac2\gamma, \beta_+)\\ 
		&= \mm(\gamma + \frac2\gamma + \beta - \beta_-, \beta_- + \beta_+ - 2Q) \mm(\beta_- - \frac2\gamma, \beta_+).  
	\end{split}
	\eqen
	Eliminating $\mm(\beta_- - \frac2\gamma, \beta_+)$ yields
	\eqb \label{eq-thm1-mero-ugly}
	\mm(\beta_-, \beta_+) = \frac
	{\mm(\beta, \beta_+ - \frac2\gamma) \mm(\gamma, \beta_+)}
	{\mm(\gamma + \frac2\gamma + \beta - \beta_-, \beta_- + \beta_+ - 2Q)\mm( \gamma, \beta_- + \beta_+ - Q - \frac\gamma2) }.
	\eqe
	For $\beta$ negative enough, each of the four factors on the right side of~\eqref{eq-thm1-mero-ugly} can be evaluated by~\eqref{eq-thm1-special-cases}, which gives meromorphic functions in $\beta_-$ and $\beta_+$ on a complex neighborhood of $(-\infty,Q+\frac{\gamma}2)$. This means that $\mm(\beta_-, \beta_+) $ is meromorphic in $\beta_-$ and $\beta_+$ on a complex neighborhood of $(-\infty,Q+\frac{\gamma}2)$. This shows that~\eqref{eq-lem-SLE-moment-main} holds for all $\beta_-, \beta_+ < Q + \frac\gamma2$ and $\lambda < 0$.
	
	Now, we extend from $\lambda < 0$ to the full result. Indeed, as in the proof of Proposition~\ref{prop-shift-1}, by holomorphic extension in $\lambda$~\eqref{eq-lem-SLE-moment-main} holds for all 	$\lambda <\lambda_0 = \frac1\kappa(\rho_+ + 2)(\rho_+ + 4 - \frac\kappa2)$. 
	For $\lambda \geq \lambda_0$ and $\eps>0$, by Lemma~\ref{lem:mono} we have $\mm(\beta_-, \beta_+) \geq m^{\lambda_0 - \eps}_\gamma(\beta_-, \beta_+)$. 
	Since $\lambda_0$ is achieved when   $\alpha=2(\beta_+-Q)$, by the explicit formula in~\eqref{eq-lem-SLE-moment-main}, we have $\lim_{\eps \to 0+} m^{\lambda_0 - \eps}_\gamma(\beta_-, \beta_+) = \infty$ hence $ \mm(\beta_-, \beta_+)=\infty$.
\end{proof}

The following lemma treats the case $\kappa>4$ using SLE duality.
\begin{lemma}\label{lem-thm-big}
	Theorem~\ref{thm-conformal-derivative0} holds for $\kappa \in (4, \infty)$, $\rho_->-2$ and $\rho_+ > \frac\kappa2-4$. 
\end{lemma}
\begin{proof}
	By SLE duality (see \cite[Theorem 5.1]{zhan-duality1} and \cite[Theorem 1.4]{ig1}) the right boundary of an $\SLE_\kappa(\rho_-; \rho_+)$ has the law of an $\SLE_{\wt \kappa}(\wt\rho_-; \wt\rho_+)$ curve, where $\wt\kappa = \frac{16}{\kappa} < 4$, $\wt\rho_- = \frac{\wt\kappa}2 -2 +\frac{\wt\kappa}4\rho_-$ and $\wt\rho_+ = \wt\kappa + \frac{\wt\kappa}4 \rho_+ - 4$. Hence when $\lambda < \wt \lambda_0 = \frac1{\wt \kappa} (\wt \rho_+ + 2) (\wt \rho_+ + 4 - \frac{\wt \kappa}2)$, by Lemma~\ref{lem-thm-small} we have
	\alb
	\E[\psi'(1)^\lambda] &= \frac{F(\wt\alpha, \wt\kappa, \wt\rho_-, \wt\rho_+)}{F(\sqrt{\wt\kappa}, \wt\kappa, \wt\rho_-,\wt \rho_+)}, \quad \text{ where } \wt \alpha \text{ solves } 1 - \frac{\wt \alpha}2 (\frac{\sqrt{\wt\kappa}}{2} + \frac2{\sqrt{\wt\kappa}} -\frac{\wt\alpha}2) = \lambda. 
	\ale 
	Once can easily verify that $\lambda_0 = \wt \lambda_0$, and 
	\eqb\label{eq-invariants}
	\frac2{\sqrt\kappa} + \frac{\rho_-}{\sqrt\kappa} = \frac2{\sqrt{\wt\kappa}} + \frac{\wt \rho_-}{\sqrt{\wt \kappa}}, \quad \frac4{\sqrt\kappa} + \frac{\rho_+}{\sqrt\kappa} = \frac4{\sqrt{\wt\kappa}} + \frac{\wt \rho_+}{\sqrt{\wt \kappa}}, \quad \frac2{\sqrt\kappa} + \frac{\sqrt\kappa}2 = \frac2{\sqrt{\wt \kappa}} + \frac{\sqrt{\wt \kappa}}2.
	\eqe
	The last identity above means we can take $\alpha = \wt \alpha$.
	Comparing $F(\alpha, \kappa, \rho_-, \rho_+)$ and $F(\wt \alpha, \wt \kappa, \wt \rho_-, \wt\rho_+)$, using~\eqref{eq-invariants} we can pair up their terms so their arguments agree. Since $\Gamma_{\sqrt \kappa /2} = \Gamma_{\sqrt{\wt \kappa}/2}$ as $\frac{\sqrt\kappa}2= \frac2{\sqrt{\wt \kappa}}$, we get termwise equality. Thus $F(\alpha, \kappa, \rho_-, \rho_+)=F(\wt \alpha, \wt \kappa, \wt \rho_-, \wt\rho_+)$, and similarly $F(\sqrt\kappa, \kappa, \rho_-, \rho_+) = F(\sqrt{\wt \kappa}, \wt \kappa, \wt \rho_-, \wt\rho_+)$. We conclude that 
	\[
	\E[\psi'(1)^\lambda] = \frac{F(\wt\alpha, \wt\kappa, \wt\rho_-, \wt\rho_+)}{F(\sqrt{\wt\kappa}, \wt\kappa, \wt\rho_-,\wt \rho_+)} = \frac{F(\alpha,\kappa,\rho_-,\rho_+)}{F(\sqrt\kappa, \kappa, \rho_-, \rho_+)}, \]
	hence Theorem~\ref{thm-conformal-derivative0} holds for $\kappa \in (4, \infty)$, $\rho_->-2$, $\rho_+ > \frac\kappa2-4$ and $\lambda < \lambda_0$.


	It remains to check that $\E[\psi'(1)^\lambda] = \infty$ for all $\lambda \geq \lambda_0$. 
	As before, we have $\psi'(1) > 1$ a.s., so the function $x \mapsto \E[\psi'(1)^x]$ is increasing on $\R$, and from the explicit formula we have just shown, we see that $\E[\psi'(1)^{\lambda}] \geq \lim_{\eps \to 0^+}  \E[\psi'(1)^{\lambda_0-\eps}] = \infty$.  
\end{proof}

\begin{proof}[Proof of Theorem~\ref{thm-conformal-derivative0}]
	The heart of the argument is Proposition~\ref{prop-SLE-moment-main}, and Lemmas~\ref{lem-thm-small} and~\ref{lem-thm-big} tie up the remaining details. 
\end{proof}

\appendix
\section{Backgrounds on  Schramm-Loewner evolutions}
In this section we provide further background on SLE$_\kappa(\rho_-;\rho_+)$ that is relevant to Theorem~\ref{thm-conformal-derivative0}.

\subsection{The Loewner evolution definition of SLE$_\kappa(\rho_-;\rho_+)$}\label{app:sle-def}
Let $\BB H$ be the upper half-plane. For a continuous function $(W_t)_{t\geq 0}$ that we call the \emph{driving function} consider the solution $g_t(z)$ of the Loewner differential equation  
\eqbn
g_t(z) =  \int_0^t \frac{2}{g_s(z) - W_s}\, ds,\qquad g_0(z)=z,\,z\in\BB H.
\eqen
For each $z\in\BB H$ let $\tau_z$ denote the supremum of times $t>0$ such that $g_t(z)$ is well-defined. For certain choices of $W$ one can show that there exists a unique continuous curve $\eta$ in $\BB H$ from 0 to $\infty$ such that if $K_t\subset\BB H$ denotes the set of points in $\BB H$ which are disconnected from $\infty$ by $\eta([0,t])$ then $K_t=\{z\in\BB H\,:\,\tau_z \leq t \}$. We say that $W$ is the Loewner driving function of $\eta$. By setting $W_t=\sqrt{\kappa}B_t$ for a standard Brownian motion $(B_t)_{t\geq 0}$ and $\kappa>0$ we get the curve $\eta$ which is known as a \emph{Schramm-Loewner evolution with parameter $\kappa$} (SLE$_\kappa$). See e.g.\ \cite{lawler-book,schramm0} for more details.

SLE$_\kappa(\rho_-;\rho_+)$ is the natural generalization of SLE$_\kappa$ when we keep track of two additional marked points on the domain boundary. Let $\rho_-,\rho_+>-2$. Given a standard Brownian motion $(B_t)_{t \geq 0}$ consider the solutions $W,V^\pm$ of the following stochastic differential equations 
\eqb\label{eq-SDE}
W_t = \sqrt\kappa B_t + \int_0^t \frac{\rho_-}{W_s - V_s^-}\, ds
+ \int_0^t \frac{\rho_+}{W_s - V_s^+}\, ds, \qquad 
V_t^\pm =  \int_0^t \frac{2}{V_s^\pm - W_s}\, ds
\eqe
with initial condition $(W_0, V_0^-, V_0^+) = (0,0,0)$. The uniqueness in law of the solution was proved in  \cite[Theorem 2.2]{ig1}.
Moreover, one can show that there is a unique curve $\eta$ from 0 to $\infty$ in $\BB H$ which has Loewner driving function given by $W$. 
We call $\eta$ an SLE$_\kappa(\rho_-;\rho_+)$. See \cite{lsw-restriction,dubedat-rho,ig1} for further details.


\subsection{Finiteness of  moments}
In this section we prove the following finiteness of moment statement. 

\begin{proposition}\label{prop-finite-moment}
For $\kappa \in (0,4)$ and $\rho_-, \rho_+ > -2$, sample $\eta \sim \SLE_\kappa(\rho_-; \rho_+)$ in $\bbH$ from $0$ to $\infty$. Let $\lambda_0 = \frac1\kappa(\rho_+ + 2) (\rho_+ + 4 - \frac\kappa2)$. Let $D$ be the connected component of $\bbH \backslash \eta$ containing $1$ on its boundary, and let $\psi$ be the conformal map from $D$ to $\bbH$ with $\psi(1) = 1$ and mapping the first (resp.\ last) point on $\partial D$ traced by $\eta$ to 0 (resp.\ $\infty$). Then \(\E[\psi'(1)^\lambda] < \infty \) when 
$\lambda < \lambda_0$.
\end{proposition}

\begin{lemma}\label{lem-finite-moment-0}
Proposition~\ref{prop-finite-moment} holds when $\rho_-=0$. 
\end{lemma}
\begin{proof}
By \cite[Theorem 1.8]{miller-wu-dim},  we have \(\P[\psi'(1) > y] = y^{-\lambda_0 + o(1)}\) as $ y \to \infty$. 
Since \(\E[\psi'(1)^\lambda] = \lambda^{-1} \int_0^\infty y^{\lambda - 1} \P[\psi'(1) > y] \, dy\), we conclude.
\end{proof}
%

\begin{proof}[Proof of Proposition~\ref{prop-finite-moment}]
We first inductively show that Proposition~\ref{prop-finite-moment} holds for $\rho_- = 2n$ for nonnegative integers $n$. The case $n=0$ is shown in Lemma~\ref{lem-finite-moment-0}, and if we have proved the statement for some $n$, then we obtain it for $n+1$ by using Lemma~\ref{lem-composition} with $\beta = \gamma, \beta_- = \gamma - \frac2\gamma n$ and $\beta_+ = \gamma - \frac{\rho_+}\gamma$. Here we use  that $f(x) = \frac1\kappa(x+2)(x+4 - \frac\kappa2)$ is increasing on $[\frac\kappa2-2, \infty)$. 

Now, we extend the proof to arbitrary $\rho_- > -2$. Pick $n \in \N$ with $2n > \rho_-+2$, and apply Lemma~\ref{lem-composition} with $\beta =  \gamma + \frac2\gamma - \frac{2n-2 - \rho_-}\gamma$ and $\beta_\pm = \gamma - \frac{\rho_\pm}\gamma$ to get
\(m^\lambda(\beta + \beta_- - \gamma - \frac2\gamma; \beta_+) = m^\lambda(\beta, \beta_- + \beta_- - \gamma - \frac2\gamma) m^\lambda (\beta_-, \beta_+).\)
By our inductive argument, the left hand side is finite for $\lambda < \lambda_0$, and hence so is $m^\lambda(\beta_-, \beta_+)$. This translates to the desired finiteness. 
\end{proof}

\subsection{Continuity of moments in SLE parameters}
Now, we prove continuity results (Lemmas~\ref{lem-couple-map} and  \ref{lem-couple-map-nonsimple}) used in the proof of Theorem~\ref{thm-conformal-derivative0}. 
We start from the case when the curve does not touch the domain boundary.
\begin{lemma}\label{lem-couple-map}
	Consider a sequence $(\kappa^n, \rho_-^n, \rho_+^n)_{n \geq 1}$ such that $\kappa^n \in (0, 4], \rho_-^n, \rho_+^n \geq \frac{\kappa^n}2 - 2$, and which converges componentwise to $(\kappa, \rho_-, \rho_+)$.
	Let $\eta$ be sampled from $\SLE_\kappa(\rho_-;\rho_+)$ and let $\psi$ be the mapping out function of the domain to the right of $\eta$ fixing $0, 1, \infty$. Define $\eta_n$ and $\psi_n$ similarly for the parameters $\kappa^n,\rho^n_-,\rho^n_+$.
	Then there is a coupling of $\eta,\eta_n$ such that $\psi_n'(1) \to \psi'(1)$ in probability. 
\end{lemma} 
To prove Lemma~\ref{lem-couple-map}, we recall  \cite[Lemma A.5]{ahs-disk-welding}, whose proof builds on \cite{kemppainen-book}. 
It gives continuity of the mapping out function $g_t$ in the Loewner driving function $W$.
\begin{lemma}[Lemma A.5 in \cite{ahs-disk-welding}]
	Let $\eta$ and $\wt\eta$ be curves in $\BB H$ from 0 to $\infty$ with Loewner driving function $(W_t)_{t\geq 0}$ and $(\wt W_t)_{t\geq 0}$, respectively, and let $(g_t)_{t\geq 0}$ and $(\wt g_t)_{t\geq 0}$ denote the Loewner maps. For any $\eps\in(0,1)$ there is a $\delta\in(0,1)$ such that if
	$$ 
	A = \{ (t,z) \in [0,T]\times\ol{\BB H}\,:\,  \inf_{s\in[0,t]}|g_s(z)-W_s|>\eps \}
	\quad\text{and}\quad
	\sup_{t\in[0,T]}|W_t-\wt W_t|\leq\delta.
	$$ 
	then $	\sup_{(t,z)\in A}|g_t(z)- \wt g_t(z)| < \eps.$
	\label{lem:kempp}
\end{lemma}

\begin{lemma}\label{lem-couple-Wt}
Consider a sequence $(\kappa^n, \rho_-^n, \rho_+^n)_{n \geq 1}$ such that $\kappa^n \in (0, 4], \rho_-^n, \rho_+^n \geq \frac{\kappa^n}2 - 2$, and which converges componentwise to $(\kappa, \rho_-, \rho_+)$. 
As $n \to \infty$, the driving function of $\SLE_{\kappa^n}(\rho_-^n; \rho_+^n)$ converges in law to that of $\SLE_\kappa(\rho_-; \rho_+)$ in the uniform topology on compact sets.
\end{lemma}
\begin{proof}
Let $(B_t)_{t \geq 0}$ be a standard Brownian motion and let $(W_t, V_t^-, V_t^+)$ be the solution to~\eqref{eq-SDE} as defined and constructed in \cite[Definition 2.1, Theorem 2.2]{ig1}. Similarly let $(B_t^n)_{t \geq 0}$ be standard Brownian motion and $(W_t^n, V_t^{n, -}, V_t^{n,+})$ the corresponding stochastic process for $\SLE_{\kappa^n}(\rho_-^n; \rho_+^n)$. 
We claim that there exists  a coupling of these processes such that for fixed $T>0$ we have $\sup_{t \in [0,T]} |W_t - W_t^n| \to 0$ in probability as $n \to \infty$. 
This claim immediately yields the lemma. This claim would follow from easy stochastic calculus arguments if we consider  $\SLE_\kappa(\rho^-;\rho^+)$ with force points away from zero.   The adaptation  from non-zero force points to the $0^\pm$  case is also considered in the proof of the uniqueness in law  of the solution to the  SDE system (\cite[Theorem 2.2]{ig1}). A minor modification of that argument gives the result so we omit the details.
\end{proof}

\begin{proof}[Proof of Lemma~\ref{lem-couple-map}]
	Consider a coupling such that the convergence in Lemma~\ref{lem-couple-Wt} in almost sure.
By the argument of \cite[Lemma A.4]{ahs-disk-welding}, for any compact $K \subset (\bbH \backslash \eta)\cup \R_+$ we have $\psi_n \to \psi$ uniformly on $K$ in probability. 
By Schwarz reflection, $\psi$ (resp.\ $\psi_n$) can be extended to a conformal map $\wt\psi$ (resp.\ $\wt\psi_n$) from the right connected component of $\C\setminus(\eta \cup \ol\eta)$ (resp.\ $\C\setminus(\eta_n \cup \ol\eta_n)$) to $\C \setminus \R_-$. Cauchy's integral formula then gives convergence of $\psi_n'(1)$ to $\psi'(1)$ in probability.
\end{proof}
The following lemma gives the counterpart of Lemma \ref{lem-couple-map} in the boundary touching case. In this case, we only consider curves with a single force point at $0^+$, namely  
	$ \SLE_\kappa(\rho_+) := \SLE_\kappa(0;\rho_+)$.
	This simplifies the analysis of the driving function in Lemma \ref{lem-bessel-cts} and also suffices for our application.
\begin{lemma}\label{lem-couple-map-nonsimple}
Suppose $\kappa \leq 4$ and $\rho_+ \in (-2, \frac\kappa2 - 2)$, and let $\eta  $ be sampled from $\SLE_\kappa(\rho_+)$. 
Let $D$ be the connected component of component of $\bbH \backslash \eta$ with 1 on its boundary, and let $\psi:D \to \bbH$ be the conformal map fixing $1$ and mapping the first (resp.\ last) boundary point traced by $\eta$ to $0$ (resp.\ $\infty$). Letting $(\kappa^n)_{n \geq 1}$ be a sequence tending to $\kappa$ and sampling $\eta_n \sim \SLE_{\kappa^n}(\rho_+)$ with force point at $0^+$, we likewise define domains $D_n$ and maps $\psi_n:D_n\to\BB H$. Then there is a coupling of $\eta,\eta_n$ such that $\psi_n'(1) \to \psi'(1)$ in probability. 
\end{lemma}
\begin{lemma}
	In the setting of Lemma \ref{lem-couple-map-nonsimple} we can couple $(W^n,V^{+,n})$ and $(W,V^{+})$ such that for any $T$, $\sup_{t\in[0,T]}|W^n_t-W_t|+|V^{+,n}_t-V^+_t|$ converges a.s.\ to 0 and the zero set of $(W^n_t-V^{+,n}_t)_{t\in[0,T]}$ converges a.s.\ to the zero set of $(W_t-V^+_t)_{t\in[0,T]}$ for the Hausdorff topology. 
	\label{lem-bessel-cts}
\end{lemma}
\begin{proof}
Since $\rho_-=0$ the law of $V^+_t-W_t$ is given by a multiple of a Bessel process. Using this and the continuity property of Bessel processes in its dimension we get the  lemma.
\end{proof}

\begin{proof}[Proof of Lemma \ref{lem-couple-map-nonsimple}]
	Consider a coupling such that the convergence in Lemma \ref{lem-bessel-cts} is a.s.
	Let $\tau_n,\tau$ (resp.\ $\sigma_n,\sigma$) be such that
	\eqbn
	\partial D = \eta([\tau,\sigma])\cup[\eta(\tau),\eta(\sigma)],\qquad
	\partial D_n = \eta_n([\tau_n,\sigma_n])\cup[\eta_n(\tau_n),\eta_n(\sigma_n)].
	\eqen
	Lemma \ref{lem-bessel-cts} implies that $\tau_n\rta \tau$ and $\sigma_n\rta\sigma$ a.s. Let $\wt\psi:g_{\tau}(D)\to\BB H$ be such that $\psi=\wt\psi\circ g_{\tau}$ and define $\wt\psi_n$ similarly. Then the chain rule for differentiation gives the following
	\eqb
	\psi'_n(1) = (g^n_{\tau_n})'(1)\cdot\wt\psi'_n( g^n_{\tau_n}(1) ),\qquad
	\psi'(1) = g'_{\tau}(1)\cdot\wt\psi'( g_{\tau}(1) ).
	\label{eq:psi-prod}
	\eqe
	Extending \cite[equation (4.5)]{lawler-book} to points on $\R$ we get
	\eqbn
	\dot g'_t(1) = -\frac{g'_t(1)}{(g_t(1)-W_t)^2},
	\eqen
	and the analogous equation for $g^n_t$. By using this, $\tau_n\rta\tau$, $(W^n_t)\rta(W_t)$, and the fact that the denominator on the right side of \eqref{eq:psi-prod1} is bounded away from 0 during $[0,\tau]$, we get that a.s.,
	\eqbn
	(g^n_{\tau_n})'(1)\rta g'_{\tau}(1).
	\eqen
	Combining this with \eqref{eq:psi-prod}, in order to conclude the proof of the lemma it is sufficient to show that $\wt\psi'_n( g^n_{\tau_n}(1) )\rta\wt\psi'( g_{\tau}(1) )$ a.s.
	
	Let $\wt\eta:[\tau,\infty)\to\BB H\cup\{0 \}$ be defined by 
	$\wt\eta(t)=g_{\tau}(\eta(t))$, and define
	$\wt\eta_n:[\tau_n,\infty)\to\BB H\cup\{0 \}$ by $\wt\eta_n(t)=g^n_{\tau_n}(\eta_n(t))$.
	We will now argue that $\wt \eta_n([\tau_n,\sigma_n])$ converges in Hausdorff topology to $\wt \eta([\tau,\sigma])$, which is sufficient to conclude the proof of the lemma since it implies $\wt\psi'_n( g^n_{\tau_n}(1) )\rta\wt\psi'( g_{\tau}(1) )$ a.s. Let $h_{\op{max}}=\sup\{ \op{Im}(\wt\eta(t))\,:\,t\in[\tau,\sigma] \}$. For $\eps\in(0,1)$ pick $\bar s(\eps),s(\eps)>0$ such that $\wt \eta_n([\tau_n+\bar s(\eps),\sigma_n-s(\eps)])$ is an excursion above the line $\{z\,:\, \op{Im}(z)=\eps h_{\op{max}}\}$ which attains the value $h_{\op{max}}$. Notice that this a.s.\ uniquely specifies $\bar s(\eps),s(\eps)$.
	
	Since $\wt \eta([\tau,\sigma-s(\eps)])$ is a simple curve, Lemma \ref{lem:kempp} implies that for any neighborhood $A$ of $\wt \eta([\tau,\sigma-s(\eps)])$ we will have $\wt \eta_n([\tau_n,\sigma_n-s(\eps)])\subset A$ for all sufficiently large $n$. Since the half-plane capacity of $\wt \eta_n([\tau_n,\sigma_n-s(\eps)])$ converges to the half-plane capacity of $\wt \eta([\tau,\sigma-s(\eps)])$, this implies that $\wt \eta_n([\tau_n,\sigma_n-s(\eps)])$ converges to $\wt \eta([\tau,\sigma-s(\eps)])$ for the Hausdorff distance, and that $\wt \eta_n(\sigma_n-s(\eps))$ converges to $\wt \eta(\sigma-s(\eps))$.
	To conclude that $\wt \eta_n([\tau_n,\sigma_n])$ converges to $\wt \eta([\tau,\sigma])$ for the Hausdorff distance it thus suffices to prove 
	\eqb
	\lim_{\eps\rta 0} \sup_{n\in\N} \op{diam}( \wt \eta_n([\sigma_n-s(\eps),\sigma_n]) ) \rta 0.
	\label{eq-excursion-uniform}
	\eqe
	Let $L^\eps_n$ be a simple curve of diameter $o_\eps(1)$  which connects $\wt\eta_n(\sigma_n-s(\eps))$ to $\R$ and is disjoint from $\wt \eta_n([\tau_n,\sigma_n-s(\eps)])$ except at its end-points. The curve $L^\eps_n\cup\wt \eta_n([\tau_n,\sigma_n-s(\eps)])$ is simple and divides $\BB H$ into a bounded and an unbounded set; let $\wt D^\eps_n$ denote the bounded set. To prove \eqref{eq-excursion-uniform} it is sufficient to argue
	\eqb
	\begin{split}
		&(i)\,\lim_{\eps\rta 0} \sup_{n\in\N} \op{diam}( \wt \eta_n([\sigma_n-s(\eps),\sigma_n])\cap \wt D^\eps_n ) \rta 0
		\,\,\,\text{and}\,\,\,\\
		&(ii)\,\lim_{\eps\rta 0} \sup_{n\in\N} \op{diam}( \wt \eta_n([\sigma_n-s(\eps),\sigma_n]) \setminus \wt D^\eps_n ) \rta 0.
		\label{eq-excursion-uniform-2}
	\end{split}
	\eqe
	We see that (ii) holds since otherwise there would be a (random) constant $c>0$ independent of $\eps$ such that for arbitrarily large $n$ and all $y>1$ sufficiently large, $y$ times the harmonic measure of $\wt \eta_n([\sigma_n-s(\eps),\sigma_n])$ seen from $iy$ in $\BB H\setminus \wt\eta_n([\tau_n,\sigma_n])$ would be at least $c$; this contradicts the assumed convergence of $(W^n,V^{n,+})$.
	
	To prove that (i) holds we can first proceed similarly as in the proof of (ii) and use harmonic measure considerations and convergence of $(W^n,V^{n,+})$ to conclude that 
	\eqb
	\lim_{\eps\rta 0} \sup_{n\in\N} \op{diam}( \wt \eta_n([\tau_n,\tau_n+\bar s(\eps)])) \rta 0.
	\label{eq-small-diam}
	\eqe
	By Lemma \ref{lem:kempp} the map $g^n_{\tau_n}$ converges uniformly to $g_{\tau}$ away from the hull of $\eta_n|_{[0,\tau_n]}$.
	Combining this with \eqref{eq-small-diam} we get that for any $\delta>0$ we can find a sufficiently small $\eps>0$, such that $\eta_n([\tau_n,\tau_n+\bar s(\eps)])$ is contained in the $\delta$-neighborhood of the hull of $\eta_n|_{[0,\tau_n]}$. By reversibility of  SLE$_{\kappa_n}(\rho_+)$ the same property holds for $\eta_n([\sigma_n-s(\eps),\sigma_n])$ and the hull created by $\eta_n|_{[\sigma_n,\infty)}$. Applying the map $g^n_{\tau_n}$ this gives (i).
\end{proof}

\section{The LCFT description of  $\QS_3$}
\label{app:sphere}
In this section we prove Proposition~\ref{prop-AHS} and  Lemma~\ref{lem-inserting-general}. Our proofs closely follow  those of Proposition \ref{prop-2-pt-weight-length} and Lemma \ref{lem-inserting}, respectively, which are the corresponding statements for the disk case.

It will be convenient to work on the cylinder rather than $\C$. Define the cylinder $\cC$ by $\cC := ([0,2\pi] \times \R)/{\sim}$ where $(x,0) \sim (x,2\pi)$ for all $x\in \R$, and let $P_\cC$ be the law of the GFF $h_\cC$ on $\cC$ normalized to have average zero on $(0,2\pi)$. 
This way, $h_\C\sim P_\C$ and $h_\cC$ are related by the exponential map between $\C $ and $\cC$. We can then deduce the covariance kernel of $h_\cC$ from that of $h_\C$:
\(	G_\cC(z,w) = -\log|e^z-e^w| + \max(\Re z, 0) + \max(\Re w, 0). \)

As in the horizontal strip case, we have a radial-lateral decomposition of $h_\cC$. 
We write $H_1(\cC) \subset H(\cC)$ (resp.\ $H_2(\cC) \subset H(\cC)$) for the subspace of functions which are constant (resp.\ have mean zero) on $\{t\} \times [0,2\pi]$ for each $t \in \R$. We have the orthogonal decomposition $H(\cC) = H_1(\cC) \oplus H_2(\cC)$. In this case the projection of $h_\cC$ onto $H_1(\cC)$ has the distribution of $\{B_{t}\}_{t\in \R}$.

Now, we introduce the weight-$W$ quantum sphere of \cite{wedges}.
\begin{definition}
	\label{def-sphere}
	For $W >0$ and $\alpha = Q - \frac W{2\gamma}<Q$, let 
	\[Y_t =
	\left\{
	\begin{array}{ll}
		B_{t} - (Q -\alpha)t  & \mbox{if } t \geq 0 \\
		\wt B_{-t} +(Q-\alpha) t & \mbox{if } t < 0
	\end{array}
	\right. , \]
	where $(B_s)_{s \geq 0}$ is a standard Brownian motion  conditioned on	$B_{s} - (Q-\alpha)s<0$ for all $s>0$, and $(\wt B_s)_{s \geq 0}$ is an independent copy of $(B_s)_{s \geq 0}$.
	Let $h^1(z) = Y_{\Re z}$ for each $z \in \cC$.
	Let $h^2_\cC$ be independent of $h^1$  and have the law of the projection of  $h_\cC$ onto $H_2(\cC)$. Let $\wh h=h^1+h^2_\cC$.
	Let  $\mathbf c$ be a real number  sampled from $ \frac\gamma2 e^{2(\alpha-Q)c}dc$ independent of $\wh h$ and $\phi=\wh h+\mathbf c$.
	Let $\cM_2^\sph(W)$ be the infinite measure  describing the law of $(\cC, \phi, -\infty, +\infty)/{\sim_\gamma}$.
	We call a sample from $\cM_2^\sph(W)$ a (two-pointed) \emph{quantum sphere of weight $W$}.
\end{definition}
The case where $W = 4-\gamma^2$ is special since conditioned on the quantum surface, the two marked points are independently distributed according to the quantum area measure, motivating the following definition.
\begin{definition}\label{def-QS}
	Let  $(\cC, \phi, +\infty,-\infty)/{\sim_\gamma}$ be a sample from $\cM^\sph_2(4-\gamma^2)$.
	Let $\QS$ be the law of $(\cC, \phi)/{\sim_\gamma}$ under the reweighted measure $\mu_\phi(\cC)^{-2}\cM^\sph(4-\gamma^2)$.
	For $m\ge 0$,  let $(\cC,\phi)$ be a sample from $\mu_\phi(\cC)^m\QS$, and then
	independently sample $z_1,\cdots, z_m$ according to $\mu_\phi^\#$.
	Let $\QS_{m}$ be  the law of $(\cC, \phi, z_1,\cdots, z_m)/{\sim_\gamma}$.
	We  call a sample from $\QS_{m}$ a  \emph{quantum sphere with $m$ marked points}. 
\end{definition}
We have $\cM_2^\sph(4-\gamma^2) = \QS_2$ \cite[Proposition A.11]{wedges}.

Recall the Liouville field on the plane defined in Definition~\ref{def-RV}. When $\alpha_1=\alpha_2$ we often prefer to put the field on the cylinder. 
\begin{definition}\label{def-RV-cl}
	Let $(h, \mathbf c)$ be sampled from $ C_\cC^{(\alpha, \pm\infty), (\alpha_3,z_3)}  P_\cC \times [e^{(2\alpha + \alpha_3 - 2Q)c}dc]$ where $\alpha \in \R$, $(\alpha_3,z_3) \in  \R \times \cC$, and 
	\[
	C_\cC^{(\alpha, \pm\infty), (\alpha_3,z_3)} = e^{(-\alpha_3(Q-\frac{\alpha_3}2) +\alpha \alpha_3) \left|\Re z_3\right|}.
	\]
	Let \(\phi(z) = h(z) -(Q-\alpha) \left|\Re z\right|  +  \alpha_3 G_\cC(z, z_3) + \mathbf c\).
	We  write $\LF_\cC^{(\alpha, \pm\infty), (\alpha_3,z_3)}$ as the law of $\phi$. When $\alpha_3=0$, we write it as 
	$\LF_\cC^{(\alpha, \pm\infty)}$. 
\end{definition}

Our next lemma relates the fields of Definitions~\ref{def-RV-sph} and~\ref{def-RV-cl} under change of coordinates for one choice of conformal map. The proof is identical to that of Lemma~\ref{lem-equiv-C-cS}.

\begin{lemma}\label{lem-equiv-C-cC}
	Let $\alpha\in\R$ and $(\alpha_3, z_3) \in \R \times \cC$. Let $f: \cC \to \C$ be the unique conformal map satisfying $f(-\infty) = 0, f(+\infty) = -1$ and $f(z_3) = 1$. Then 
	\[\LF_{\C}^{(\alpha, -1), (\alpha, 0), (\alpha_3, 1)} = 2^{ - 2\Delta_{\alpha_3}} \cdot f_* \LF_\cC^{(\alpha, \pm\infty), (\alpha_3, z_3)}. \]
\end{lemma}

We give an LCFT description of the quantum sphere. 
\begin{theorem}\label{thm-two-sph-equivalence}
	Fix $W > 0$ and let $\phi$ be as in Definition~\ref{def-sphere} so that $(\cC, \phi, +\infty, -\infty)$ is an embedding of  a sample from $\cM_2^\sph(W)$.
	Let $T\in \R$ be sampled from the Lebesgue measure $dt$ independently of 	$\phi$.  
	Let $\wt \phi(z)=\phi (z+T)$.
	Then the law of $\wt \phi$ is given by $ \frac{\gamma}{4(Q-\alpha)^2} \LF_\cC^{(\alpha,\pm\infty)}$
	where $\alpha = Q - \frac W{2\gamma}$.  
\end{theorem}
\begin{proof}
	We follow the proof of Theorem~\ref{thm-two-disk-equivalence}, except that we set $a = (Q-\alpha)$, and no factor of $\frac12$ is incurred since the projection of $h \sim P_\cC$ to $\cH_1(\cC)$ is standard Brownian motion with no factor of 2 in its time parametrization. So the prefactor is instead $\frac\gamma{4a^2} = \frac{\gamma}{4 (Q-\alpha)^2}$. 
\end{proof}

Now, we give an LCFT description of a weight $W$ quantum sphere with a marked point added.

\begin{definition}\label{def-sphere-marked}
	Fix $W > 0$. 
	Let  $(D,\phi,a,b)$ be an embedding of a sample from  $\cM_2^\sph(W)$ and $\mu_\phi$ be the quantum are measure. 
	Let $A$ be the total $\mu_\phi$-area of $D$.
	Now consider $(D,\phi,a,b)$ from the reweighted measure $A\cM_2^\sph(W)$. Given $\phi$,  
	sample $\mathbf z$ from the probability measure proportional to  $\mu_\phi$.
	We write $\cM_{2, \bullet}^\sph(W)$ as the law of the marked quantum surface $(D,\phi,a,b,\mathbf z)/{\sim_\gamma}$.
\end{definition}

\begin{proposition}\label{prop-QS3-LF}
	For $W>0$,	let $\phi$ be sampled from $\frac{\pi \gamma}{2 (Q-\alpha)^{2}} \cdot \LF_\cC^{(\alpha, \pm\infty), (\gamma, 0)}$ where $\alpha = Q - \frac W{2\gamma}$. 
	Then  $(\cC, \phi, -\infty, +\infty, 0)/{\sim_\gamma}$ is  a sample from  $\cM_{2, \bullet}^\sph(W)$.
\end{proposition}
\begin{proof}
	The argument  is identical to that of Proposition~\ref{prop-2-pt-weight-length}, except that we use the following in place of Lemma~\ref{lem-inserting}.
	\eqb\label{eq-inserting-cC}
	\LF_\cC^{(\alpha, \pm\infty)}\left[ f(\phi)\int_\cC g(u)\,\mu_\phi(du)\right]=
	\int_\cC  \LF_\cC^{(\alpha, \pm\infty), (\gamma, u)}[f(\phi)]g(u) \,\Leb_\cC (du).
	\eqe
	In Proposition~\ref{prop-2-pt-weight-length}, the prefactor $\frac\gamma{2(Q-\beta)^2}$ agrees with that of Theorem~\ref{thm-two-disk-equivalence}. The prefactor $\frac{\pi \gamma}{2 (Q-\alpha)^{2}}$ in this proposition instead differs from that of Theorem~\ref{thm-two-sph-equivalence} by a factor of $2\pi$, because $\cC$ is defined from $\R\times [0,2\pi]$ hence $\Leb_\cC$ in~\eqref{eq-inserting-cC} contributes a factor of $2\pi$.
\end{proof}

Finally, we prove Proposition~\ref{prop-AHS} and  Lemma~\ref{lem-inserting-general}.

\begin{proof}[Proof of Proposition~\ref{prop-AHS}]
		By definition, $\cM_{2, \bullet}^\sph(4-\gamma^2) = \QS_3$. 
		Thus the result follows by setting $\alpha=\gamma$
		in Proposition~\ref{prop-QS3-LF} and using the change of coordinate from Proposition~\ref{prop-RV-invariance} and Lemma~\ref{lem-equiv-C-cC}. 
\end{proof}
\begin{proof}[Proof of Lemma~\ref{lem-inserting-general}]
	We focus on proving 
	\begin{equation}\label{eq:insert-sphere}
		\LF_{\C}\left[ f(\phi)\int_\C g(u)\,\mu_\phi(du)\right]=
		\int_\C  \LF_{\C}^{(\gamma, u)}[f(\phi)]g(u)\,d^2 u.
	\end{equation}
	Once this is done, we can add insertions $(\alpha_i,z_i)_i$ to both sides of~\eqref{eq:insert-sphere} using the sphere analog of Lemma~\ref{lem:inserting} and its proof. This gives the general case.

	The proof of~\eqref{eq:insert-sphere} is almost identical to that of Lemma~\ref{lem-inserting} so we only point out the modifications. 
	First, as in Lemma~\ref{lem-girsanov}, by the Girsanov theorem we have 
	\[
	\int f(h)\left(\int_\C g(u)\,\mu_h(du)\right)\, P_\C(dh)=\int_\C\E_\C\left[f(h +\gamma G_\C(\cdot, u))\right] g(u)\rho(u) \,du.
	\] 
	where $\E_\C$ is the expectation over $P_\C$ and  $\rho(u)$ is defined by  $\rho(u) du=\E_\C[\mu_h(du)]$.
	On the other hand, the sphere analog of Lemma~\ref{lem-interpret-partition-fn} gives \(C^{(\gamma,u)}_{\C} du =e^{ -2\gamma Q \log |z|_+}  \rho(u)du\).
	Now the same argument as in the proof of Lemma~\ref{lem-inserting} gives~\eqref{eq:insert-sphere}. 
\end{proof}
 
\bibliographystyle{hmralphaabbrv}
\bibliography{cibib}

\end{document}